\documentclass[a4paper,12pt]{article}
\usepackage[utf8]{inputenc}  
\usepackage[T1]{fontenc}
\usepackage[english]{babel}
\usepackage[top=1.5cm, bottom=1.5cm, left=1.5cm, right=1.5cm]{geometry}
\usepackage{xcolor}
\usepackage{xfrac}
\usepackage{amsthm}
\usepackage{amsmath}
\usepackage{amssymb}
\usepackage{mathrsfs}
\usepackage{ulem}
\usepackage{hyperref}
\usepackage{fourier}
\usepackage{graphicx}
\usepackage{caption} 
\usepackage{algorithm}
\usepackage{algorithmic}
\usepackage{nicematrix}
\usepackage{textcomp}
\usepackage{multido}
\usepackage{framed}
\usepackage{comment}
\usepackage{enumitem}
\usepackage{bm}
\usepackage{nicematrix}
\usepackage{booktabs}
\usepackage{multirow}
\usepackage{float}

\usepackage{tkz-tab} 
\usepackage{pgfplots}
\usepackage{tikz}
\pgfplotsset{compat=1.15}
\usepackage{mathrsfs}
\usetikzlibrary{arrows}

\DeclareMathOperator{\e}{e}
\DeclareMathOperator{\tr}{Tr}
\DeclareMathOperator{\T}{T}
\DeclareMathOperator{\hess}{Hess}

\DeclareMathOperator{\id}{Id}
\DeclareMathOperator{\sign}{sign}
\DeclareMathOperator{\Rel}{Re}

\newcommand{\Z}{\mathbb{Z}}
\newcommand{\R}{\mathbb{R}}
\newcommand{\C}{\mathbb{C}}

\newcommand{\oo}[2]{ \mathopen{ ( } #1, #2 \mathclose{ ) } }
\newcommand{\fo}[2]{\mathopen{ [ } #1,#2 \mathclose{ ) }}
\newcommand{\ff}[2]{\mathopen{ [ } #1, #2 \mathclose{ ] }}

\newcommand{\co}{\mathcal{C}}

\newcommand{\Hm}{\mathcal{H}}
\newcommand{\I}{\mathcal{I}}
\newcommand{\Em}{\mathcal{E}}
\newcommand{\la}{\mathcal{L}}
\DeclareMathOperator{\leg}{L}
\DeclareMathOperator{\A}{\mathbf{A}}
\DeclareMathOperator{\Al}{\mathcal{A}}

\newcommand{\p}{\partial}

\newcommand{\dd}{\mathop{}\mathopen{}\mathrm{d}} 
\newcommand{\n}{\nabla}
\newcommand{\ntot}{\n_{\text{tot}}}

\newcommand{\ve}{\varepsilon}
\newcommand{\lm}{\bm{\lambda}}
\newcommand{\mug}{\bm{\mu}}
\newcommand{\sig}{\sigma}
\newcommand{\ka}{\kappa}

\newcommand{\w}{\mathcal{W}}
\newcommand{\V}{\mathcal{V}}
\newcommand{\y}{\mathcal{Y}}
\newcommand{\mo}{\mathcal{M}}
\newcommand{\Ps}{\mathbb{P}_s}
\newcommand{\sm}{\mathbb{S}}

\newcommand{\go}{\mathcal{O}}

\newcommand{\intsigp}{\int_{0}^{\frac{\eta}{\ve}}}
\newcommand{\intsign}{\int_{\frac{-\eta}{\ve}}^{0}}
\newcommand{\intwp}{\int_{0}^{\eta}}
\newcommand{\intwn}{\int_{-\eta}^{0}}

\newcommand{\norm}[1]{\left\Vert #1\right\Vert}
\newcommand{\abs}[1]{\left| #1 \right|}

\newcommand{\fonction}[5]{\begin{array}{lrcl}
#1 \colon & #2 & \longrightarrow & #3 \\
& #4 & \longmapsto & #5 \end{array}}

\newcommand{\tend}[2]{\underset{#1 \to #2}{\longrightarrow}} 

\newcommand{\ti}[1]{\Tilde{#1}}
\newcommand{\un}[1]{\underline{#1}}
    
\newtheorem{theorem}{Theorem}[subsection]
\newtheorem{proposition}{Proposition}[subsection]
\newtheorem{corollaire}[theorem]{Corollary}
\newtheorem{lemma}{Lemma}[subsection]
\newtheorem{definition}{Definition}[subsection]
\newtheorem{def,prop}{Definition/Proposition}[subsection]
\newtheorem{remarque}{Remark}[subsection]
\newtheorem{def,th}{Definition/Theorem}[subsection]

\newtheorem{notation}{Notation}
\newtheorem{hypothèse}{Assumption}

\title{Instability of two-pulse periodic waves with long wavelength in some Hamiltonian PDEs}
\author{Thomas Courant
\thanks{Research of Thomas Courant was partially supported by Centre Henri Lebesgue ANR-11-LABX-0020-01.\\
Univ Rennes, CNRS, IRMAR - UMR 6625, F-35000 Rennes, France.
Email address: thomas.courant@univ-rennes.fr}}

\begin{document}

\maketitle

\begin{abstract}
    We consider quasilinear generalizations of the Korteweg–de Vries equation and dispersive perturbations of the Euler equations for compressible fluids, either in Lagrangian or in Eulerian coordinates. In particular, our framework includes hydrodynamic formulation of the nonlinear Schrödinger equations. The periodic waves we study exhibit on each period two pulses, one converging to a bright soliton and one converging to a dark soliton, when wavelength goes to infinity. We show that such waves, for sufficiently large periods, are spectrally unstable. To do so, we combine two approaches. The first one is to calculate the asymptotic expansion of the Hessian matrix of the action integral and concludes using \cite{benzoni-gavageCoperiodicStabilityPeriodic2016} as in \cite{benzoni-gavageStabilityPeriodicWaves2020}. This shows instability when both limiting solitary waves are stable. The second approach studies the convergence of the spectrum when the period goes to infinity and is applied in remaining cases, when one of the solitary waves is unstable. To carry out the latter, we prove the convergence of an appropriate renormalization of the periodic Evans function as in \cite{yangConvergencePeriodGoes2019}.
\end{abstract}



\tableofcontents

\section{Introduction}

The present paper studies stability/instability criteria for periodic traveling waves of Hamiltonian PDEs. For general background on nonlinear waves we refer the reader to \cite{kapitulaSpectralDynamicalStability2013} and to \cite{angulopavaNonlinearDispersiveEquations2009,haragusSpectraPeriodicWaves2008,kapitulaCountingEigenvaluesKrein2005} for specific result in Hamiltonian systems.

Because of the relatively large dimension of family of periodic traveling waves such stability/instability  criteria typically fail to set a neat dichotomy. This is in strong contrast with what happens for solitary waves where for a large class of equations the family of waves with a fixed endstate is one-dimensional, parametrized by velocity, and the sign of the second-order derivative of the Boussinesq momentum determines the stability/instability. Yet, as proved in \cite{benzoni-gavageStabilityPeriodicWaves2020}, following \cite{benzoni-gavageCoperiodicStabilityPeriodic2016}, one may expect a reduction of the complexity of periodic waves stability in some asymptotic regimes.

In \cite{benzoni-gavageCoperiodicStabilityPeriodic2016} the authors show some stability/instability criteria for periodic traveling waves of large classes of Hamiltonian PDEs. The goal of \cite{benzoni-gavageStabilityPeriodicWaves2020} was to see what happened for those criteria in two different limits. The first one, the harmonic limit, considers periodic waves with small amplitude around a constant state. The second one, the soliton limit, studies periodic waves with large period that converge to a solitary wave. The outcome is that in the harmonic limit, periodic waves were always orbitally stable, and in the soliton limit, periodic waves share the same stability properties as their limiting soliton. Here we are interested in periodic waves of large periods,  asymptotically described on each periodic cell by the superposition of two distinct solitary waves sharing the same endstate. We will call such waves \emph{two-pulse periodic waves}. In the end we show that they are always spectrally unstable.

For the classes of equations that we consider, wave profiles are in correspondence with trajectories of a two-dimensional Hamiltonian differential equation, hence with level sets of a two-dimensional function. As a consequence the three asymptotic regimes mentioned above, the two studied in \cite{benzoni-gavageStabilityPeriodicWaves2020} and the one studied here, contain all the generic ways in which a family of periodic waves may end.

The instability we prove is at the spectral and linear level, and holds both under co-periodic and localized perturbations. To our knowledge, the conversion of spectral instability into nonlinear instability for periodic traveling waves is a largely open question for quasilinear dispersive PDEs like the ones we are considering. In the semilinear case however we expect nonlinear orbital instability under localized perturbations to follow from the analysis of \cite{jinNonlinearModulationalInstability2019}. Yet we mention that based on studies for parabolic systems \cite{johnsonBehaviorPeriodicSolutions2014}, space-modulated stability and not orbital stability is expected to be the sharp notion of stability for periodic waves subject to localized perturbations. Incidentally we also stress that, even for the semilinear case, nonlinear stability with respect to localized perturbations for periodic waves of Hamiltonian systems is a largely open question. For a recent significant step in this direction see \cite{bukiedaOrbitalStabilityPlane2025}.

Half of our result is obtained by adapting to our present asymptotic regime the calculations of \cite{benzoni-gavageStabilityPeriodicWaves2020}. In this way, we show that if the two solitons are orbitally stable then the corresponding two-pulse waves of sufficiently large period are spectrally unstable. Actually this part of the proof shows that this is the case when the sum of the second-order derivatives of the Boussinesq momenta is positive, which holds in particular when both are positive. At the basis of the result lies a co-periodic instability result \cite{benzoni-gavageCoperiodicStabilityPeriodic2016}.

Let us mention that the other half of the instability result cannot be obtained by studying modulational instability. For an introduction to the latter we refer to  \cite{bronskiModulationalInstabilityEquations2016}. The idea is quite natural and the authors of \cite{benzoni-gavageStabilityPeriodicWaves2020} have also studied in their companion paper \cite{benzoni-gavageModulatedEquationsHamiltonian2021} what happens in the same regimes to the modulational criteria derived in \cite{benzoni-gavageSlowModulationsPeriodic2014}. Nevertheless, as we develop in Remark~\ref{rmq_stabilité_modulationnelle}, in our regime we show that when the arguments mentioned do not conclude to instability the modulational system is hyperbolic, hence the absence of modulational instability.

Instead, we complete our study with a direct spectral perturbation analysis. When doing so we draw inspiration from single-pulse soliton analyses in \cite{gardnerSpectralAnalysisLong1997, sandstedeStabilityPeriodicTravelling2001, yangConvergencePeriodGoes2019}. More precisely, adapting \cite{yangConvergencePeriodGoes2019} to our current regime, we show that a suitable rescaled version of the periodic Evans functions converges, at an exponential rate, to the product of the Evans functions of the two solitons. With this result in hands, we conclude that if at least one of the solitary waves is unstable then the corresponding two-pulse waves of sufficiently large period are spectrally unstable.

The situation analyzed in the latter part of the argument is probably easier to guess: the two solitary waves being superposed far away from each other, there is some room for solitary-wave instabilities to develop. The former situation, when both solitary waves are stable, is intuitively less clear. A possible heuristic scenario is that perturbations could change velocities of each solitary wave somewhat independently hence break the synchronized superposition.

The present paper is organized as follows. In the next section, we recall the general framework of \cite{benzoni-gavageStabilityPeriodicWaves2014,benzoni-gavageCoperiodicStabilityPeriodic2016,benzoni-gavageStabilityPeriodicWaves2020} and state our main result. Section~\ref{double_bosse} is devoted to the calculation of the asymptotic expansion of the action in the same spirit as \cite{benzoni-gavageStabilityPeriodicWaves2020} and Section~\ref{section_evans_function} is for the asymptotic expansion of the Evans function, in the same spirit as \cite{yangConvergencePeriodGoes2019}.

{\paragraph {\bf Acknowledgment} The author thanks Sylvie Benzoni-Gavage, Dmitry Pelinovsky and Björn de Rjik for the very interesting discussions we had about the results of this article. He also particularly thanks Louise Gassot and Miguel Rodrigues for supervising this work and more generally his PhD thesis.
}

\section{Framework and main results}

\subsection{A class of Hamiltonian PDEs}

We consider the same abstract systems as in \cite{benzoni-gavageStabilityPeriodicWaves2014,benzoni-gavageCoperiodicStabilityPeriodic2016,benzoni-gavageStabilityPeriodicWaves2020}. So we are interested in the following Hamiltonian PDE,
\begin{equation}\label{eq_hamiltonienne}
    \p_t U = \p_x\left(B\delta\Hm[U]\right),
\end{equation}
where $U$ takes values in $\R^N$, $B$ is a symmetric and nonsingular matrix, and the Hamiltonian $\Hm$ depends on $U$ and $U_x$, the notation $[\cdot]$ is used to indicate this dependence. Furthermore, $\delta$ stands for the \emph{Euler} operator, which can be defined as
\begin{equation*}
    \delta\Hm[U] = \n_{U} \Hm(U,U_x) - \p_x\left(\n_{U_x}\Hm(U,U_x)\right).
\end{equation*}
In practice, in the case $N=1$ we are interested in quasilinear, generalized versions of the Korteweg-de Vries equation
\begin{equation}\label{qKdV}\tag{qKdV}
    \p_t v =  \p_x\left(\delta e[v]\right),
\end{equation}
and in the case $N=2$, by the Euler-Korteweg system, either in Eulerian coordinates,
\begin{equation} \label{EKE}\tag{EKE}
    \left\{\begin{array}{ll}
            \p_t \rho + \p_x\left(\rho u\right) = 0, \\
            \p_t u + u\p_x u + \p_x\left(\delta\Em[\rho]\right) = 0,
            \end{array} \right.
\end{equation}
and in mass Lagrangian coordinates,
\begin{equation*}\label{EKL}\tag{EKL}
    \left\{\begin{array}{ll}
            \p_t v = \p_x u , \\
            \p_t u = \p_x\left(\delta e[v]\right).
            \end{array} \right.
\end{equation*}
For discussion on the physics meanings of the equation, we refer to the introduction of \cite{benzoni-gavagePlanarTravelingWaves2013, benzoni-gavageCoperiodicStabilityPeriodic2016,benzoni-gavageStabilityPeriodicWaves2020}. The system case includes hydrodynamic formulation of the non-linear Schrödinger equation; for similar results on the Schrödinger equation, see \cite{audiardPlanePeriodicWaves2022}. We also point out Section 5.2 of \cite{benzoni-gavageCoperiodicStabilityPeriodic2016} for a discussion on the different coordinates for the Euler-Korteweg system.\newline

We do the same assumption on our abstract system as in \cite{benzoni-gavageStabilityPeriodicWaves2020}. The idea is to have a framework that meets our examples,~\eqref{qKdV},~\eqref{EKE} and~\eqref{EKL}, and in which periodic traveling waves arise as families parametrized by their speed and $(N+1)$ constants of integration. We recall the assumption here.

\begin{hypothèse}\label{hypothèse_Hm}
    We enforce that $N=1$ or $N=2$.
    \begin{itemize}
        \item If $N=1$ then we assume that $U=v$, $B=b\in \R^\ast$ and $\Hm[v] = \Em(v,v_x)$.
        \item If $N=2$ then we assume that $U=(v,u)^{\T}$, $B = \left(\begin{smallmatrix}0&b\\b&0\end{smallmatrix} \right)$ with $b \neq 0$ and
        \begin{equation*}
            \Hm[U] = \I(v,u) + \Em(v,v_x).
        \end{equation*}
    \end{itemize}
    In both cases, we assume that $\Em$ is quadratic in $v_x$ with $\ka := \p_{v_x}\Em > 0$, and, if $N=2$, that $\I$ is quadratic in $u$ with $\tau := \p_{u}\I > 0$.
\end{hypothèse}

Note that~\eqref{eq_hamiltonienne} is a conservation law on $U$, but we also have additional conservation laws: one, arising, from the time translation invariance, on the Hamiltonian $\Hm$, 
\begin{equation*}
    \p_t\Hm[U] = \p_x\left(\frac{1}{2}\delta\Hm[U]\cdot B\delta\Hm[U] + \n_{U_x}\Hm[U]\cdot\p_x\left(B\delta\Hm[U]\right)\right),
\end{equation*}
and one, stemming from the space translation invariance, on the \emph{impulse}, which is defined as 
\begin{equation*}
    Q(U) := \frac{1}{2} U \cdot B^{-1}U,
\end{equation*}
where $\cdot$ is the scalar product on $\R^N$. The latter conservation law is 
\begin{equation*}
    \p_t Q(U) = \p_x\left(U\cdot\delta\Hm[U] + \leg\Hm[U]\right),
\end{equation*} 
where $\leg$ is the formal \emph{Legendre} transform, which is
\begin{equation*}
    \leg\Hm[U] = U_x\cdot \n_{U_x}\Hm[U] - \Hm[U].
\end{equation*}

\subsection{Traveling waves}

We are interested in the traveling waves of~\eqref{eq_hamiltonienne}. We recall from Section 2.2 of \cite{benzoni-gavageStabilityPeriodicWaves2020} that, up to a translation, for a traveling wave $U(t,x) = \un{U}(x-c t)$ of speed $c$, possible profiles $\un{U}$ are parametrized by $(\mu, \lm) \in \R\times \R^N$, and determined by the equations
\begin{equation}\label{eq_1_onde_progressive}
    \delta\left(\Hm + c Q \right)[\un{U}] + \lm = 0,
\end{equation}
\begin{equation}\label{eq_2_onde_progressive}
    \leg\left(\Hm + c Q + \lm\cdot \right)[\un{U}] = \mu.
\end{equation}
These equations are ODEs, so using Assumption~\ref{hypothèse_Hm} we can make them more explicit. This is done in Section 3.1 of \cite{benzoni-gavageStabilityPeriodicWaves2020}, and we recall here the results.\newline

If $N=2$, the profile of a traveling wave, $\un{U} = \left(\begin{smallmatrix}\un{v} \\ \un{u} \end{smallmatrix}\right)$, is parametrized by $(\mu, \lm, c) \in \R \times \R^2 \times \R$, and verify
\begin{equation}\label{eq_onde_périodique_2}
    \left\{\begin{array}{ll}
        \un{u} = g\left(\un{v};\lambda_2,c\right), \\
        \frac{1}{2}\ka(\un{v})\un{v}_x^2 + \w(\un{v};\lm,c) =\mu,
    \end{array} \right.
\end{equation}
where $g\left(v;\lambda_2,c\right) = \frac{-1}{\tau(v)}\left(\frac{c}{b}v + \lambda_2\right)$ and $\w(v;\lm,c) = -f(v) + \frac{1}{2}\tau(v)g(v;\lambda_2,c)^2 - \lambda_1v$.\newline

In the scalar case, $N=1$, the profile of a traveling wave, $\un{v}$, is parametrized by $(\mu, \lm, c) \in \R \times \R \times \R$, and verify
\begin{equation}\label{eq_onde_périodique_1}
    \frac{1}{2}\ka(\un{v})\un{v}_x^2 + \w(\un{v};\lambda,c) =\mu,
\end{equation}
where $\w(v;\lambda,c)= -f(v) - \frac{1}{2}\frac{c}{b}v^2-\lambda v$.\newline

System~\eqref{eq_onde_périodique_2} consists of an Hamiltonian ODE for $v$ with potential $\w$ and an algebraic equation deducing $u$ from $v$. Thus, to understand the profile of traveling waves, we only have to focus on the equation on $v$, even if $N=2$. So, if we fix $(\lm,c) \in \R^N\times \R$, the profile of a traveling wave is determined by the energy level of $\w$. If we fix $\mu$, we are interested in two different cases.
\begin{itemize}
    \item If $v_\ell<v_r$ are two roots of $\mu - \w(\cdot;\lm,c)$, such that $\p_v\w(v_\ell;\lm,c) < 0$, $\p_v\w(v_r;\lm,c) > 0$ and for $v\in \oo{v_\ell}{v_r},~\mu-\w(v;\lm,c) > 0$. The corresponding solution $\un{v}$ of~\eqref{eq_onde_périodique_1}, respectively $\un{U} = (\un{v},\un{u})^{\T}$ of~\eqref{eq_onde_périodique_2}, is periodic. These are, therefore, profiles of periodic waves.

    \item If $v_s$ is a root of $\mu-\w(.;\lm,c)$ and a saddle point of $\w$, that is to say 
    $$\p_v\w(v_s;\lm,c) = 0,\quad \p_v^2\w(v_s;\lm,c) < 0,$$
    and if there exists another root $v^s$ such that $\p_v\w(v^s;\lm,c) > 0$ if $v_s<v^s$ or $\p_v\w(v^s;\lm,c) < 0$ if $v_s>v^s$, up to translation, we have a unique bounded non-constant solution. This solution of~\eqref{eq_onde_périodique_1} if $N=1$ (of~\eqref{eq_onde_périodique_2} if $N=2$), is the profile of a solitary wave, called a bright soliton when $v^s > v_s$, respectively a dark soliton when $v^s < v_s$.
\end{itemize}

In this paper, we are interested in the case where there exists a periodic wave with two pulses in one period, each converging to a different soliton when the period goes to infinity. For this, we need to suppose that there exists an energy level $\mu_s$ of $\w$ such that we have the existence of the profile of two different solitons, a bright and a dark soliton. We are interested in a periodic wave with an energy level $\mu > \mu_s$ close to $\mu_s$. This means that $\w$ needs to have variation like in Table~\ref{fig:variation_potentiel} and Figure~\ref{fig:graphe_potentiel}. 

\begin{table}[H]
$$\begin{array}{|c|ccccccccccccccccc|}
\hline   
v        & &         & v_\ell &          & v_\ell^s &          &         &          & v_s &          &         &          & v_r^s &        &v_r&         & \\ \hline
\p_v^2\w & &         &        &          &          &          &         &          &  -  &          &         &          &       &        &   &         &\\  \hline
\p_v\w   & &         &   -    &          &    -     &          &         &          &  0  &          &         &          &   +   &        & + &         &\\  \hline
         & &\searrow &        &          &          &          &         &          &     &          &         &          &       &        &   & \nearrow& \\
         & &         &   \mu  &          &          &          &         &          &     &          &         &          &       &        &\mu&         & \\ [-5pt]
         & &         &        & \searrow &          &          &         &          &     &          &         &          &       &\nearrow&   &         & \\
\w       & &         &        &          &  \mu_s   &          &         &          &\mu_s&          &         &          & \mu_s &        &   &         & \\
         & &         &        &          &          & \searrow &         & \nearrow &     & \searrow &         & \nearrow &       &        &   &         & \\ 
         & &         &        &          &          &          & \cdots  &          &     &          & \cdots  &          &       &        &   &         & \\ \hline
\end{array}$$
\caption{Variations of the potential $\w$}\label{fig:variation_potentiel} 
\end{table}

\begin{figure}[H]
\begin{center}
\includegraphics[width=110mm]{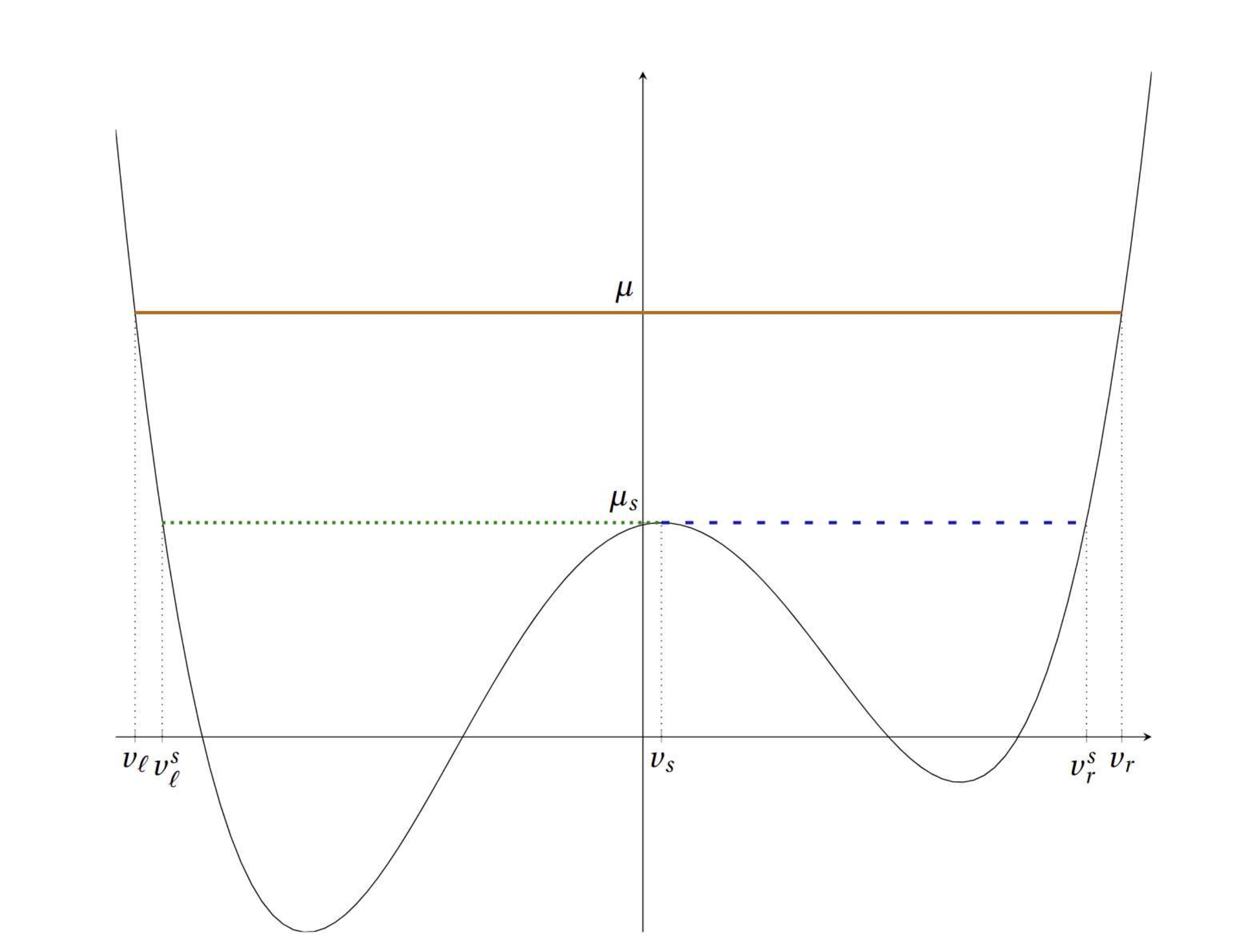}
\caption{"Generic" potential $\w$ that we consider} 
\label{fig:graphe_potentiel} 
\end{center}
\end{figure}

In Figure~\ref{fig:graphe_potentiel}, dashed blue corresponds to the energy levels of the bright soliton, and dotted blue corresponds to the dark soliton. The orange line corresponds to the energy level of the periodic waves under consideration. Those periodic waves, with an energy level $\mu>\mu_s$, are the so-called \emph{two-pulse periodic waves}. For such a potential, Figure~\ref{fig:portrait_de_phase} shows the phase portraits associated with a bright soliton, a dark soliton and a two-pulse periodic wave.

\begin{figure}[H]
\begin{center}
\includegraphics[width=110mm]{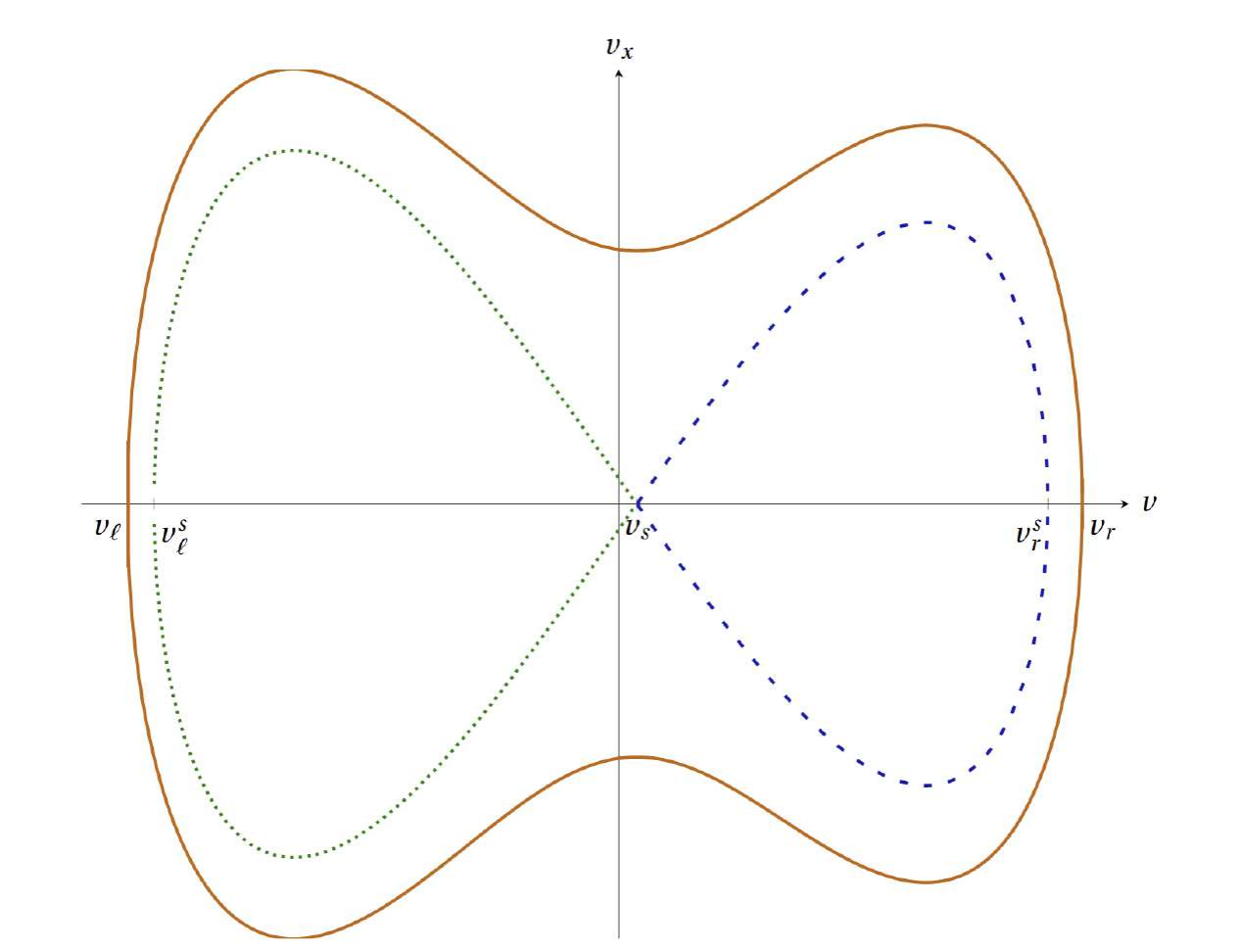}
\caption{Phase portrait for a "generic" potential $\w$} 
\label{fig:portrait_de_phase} 
\end{center}
\end{figure}

Asymptotic cases within homoclinic orbits, periodic profiles with periods that go to infinity, and periodic profiles with amplitudes that go to zero have already been dealt with in \cite{benzoni-gavageStabilityPeriodicWaves2020}. We are interested in orbits approximating the two homoclinics from outside. We draw on Figure~\ref{fig:graphe_onde_periodique} the generic "shape" of a period of such a periodic wave. The part of the profile converging to the bright soliton (resp. dark) is placed on the right (resp. left).

\begin{figure}[H]
\centering
\begin{tikzpicture}
\begin{axis}[
    axis x line=bottom,
    axis y line = left,
    axis lines=middle,
    xlabel=$x$,
    ylabel=$v(x)$,
    every axis y label/.style={at=(current axis.above origin),anchor=south},
    every axis x label/.style={at=(current axis.right of origin),anchor=west},
    xtick = \empty,
    ytick= {-1.21,1.2},
    yticklabels = {$v^\ve_\ell$,$v_r^\ve$}, 
    ymin = -1.5,
    ymax = 1.5,
    scale = 1
]
\addplot[
    domain=0:14, 
    samples=100, 
    color=black
]
{2*exp(x-7)/(1 + exp(2*(x-7)))+ 0.2}; 

\addplot[
    domain=-10:0, 
    samples=100, 
    color=black
]
{-2*sqrt(2)*exp(2*(1.4*x+7))/(1 + exp(4*(1.4*x+7)))+ 0.2}; 

\end{axis}
\end{tikzpicture}
\caption{Period of a "generic" two-pulse periodic wave} 
\label{fig:graphe_onde_periodique} 
\end{figure}
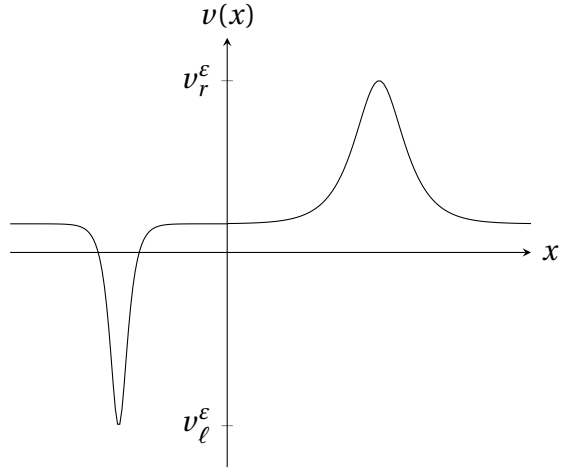

To formalize this discussion, we will now detail the main hypothesis that we have on the potential $\w$.

\begin{hypothèse}\label{hypothèse_potentiel}
    We consider a potential $\w$ that is smoothly defined on $I \times \Lambda$ where $I$ is an open interval in $\R$ and $\Lambda$ is an open subset of $\R^{N+1}$. 
    We also assume the existence of smooth functions $v_\ell^s,~v_s,~v_r^s \colon \Lambda \rightarrow I$, such that for all $(\lm,c) \in \Lambda$,
    \begin{equation*}
        v_\ell^s(\lm,c) < v_s(\lm,c) < v_r^s(\lm,c),
    \end{equation*}
    \begin{equation*}
        \mu_s(\lm,c) := \w\left( v_s(\lm,c);\lm,c\right) = \w\left( v_\ell^s(\lm,c);\lm,c\right) = \w\left( v_r^s(\lm,c);\lm,c\right),
    \end{equation*}
    \begin{equation*}
       \p_v \w\left( v_\ell^s(\lm,c);\lm,c\right) < 0,\quad \p_v \w\left( v_r^s(\lm,c);\lm,c\right) > 0, \quad
        \p_v\w\left( v_s(\lm,c);\lm,c\right) = 0,\quad \p_v^2\w\left( v_s(\lm,c);\lm,c\right) < 0,
    \end{equation*}
    \begin{equation*}
        \forall v \in \left(v_\ell^s(\lm,c),v_s(\lm,c)\right) \cup \left(v_s(\lm,c),v_r^s(\lm,c)\right), \quad \w\left( v;\lm,c\right) < \mu_s.
    \end{equation*}
\end{hypothèse} 

We denote by $\Omega \subset \R^{N+2}$ the open subset defined by
\begin{equation*}
    \Omega = \left\{\mug = (\mu,\lm,c);~(\lm,c)\in \Lambda, \mu > \mu_s(\lm,c)\right\},
\end{equation*}
and consider $v_\ell$, $v_r$ smoothly defined on $\Omega$ and satisfying for all $\mug = (\mu,\lm,c)\in \Omega$,
\begin{equation*}
    v_\ell(\mu,\lm,c) < v_\ell^s(\lm,c) < v_s(\lm,c) < v_r^s(\lm,c) < v_r(\mu,\lm,c),
\end{equation*}
\begin{equation*}
    \mu = \w\left(v_\ell(\mu,\lm,c);\lm,c\right) = \w\left(v_r(\mu,\lm,c);\lm,c\right),
\end{equation*}
\begin{equation*}
    \forall v \in \ff{v_\ell(\mu,\lm,c)}{v_\ell^s(\lm,c)},~\p_v\w(v;\lm,c) < 0 \quad \text{and} \quad \forall v \in \ff{v_r^s(\lm,c)}{v_r(\mu,\lm,c)},~\p_v\w(v;\lm,c) > 0.
\end{equation*}
From now on, we fix $(\lm_s,c_s) \in \Lambda$ and denote by $U_\ell$ the profile of the soliton with endstate $U_s = \left(\begin{smallmatrix} v_s(\lm_s,c_s) \\ g(v_s(\lm_s,c_s);\lm_s,c_s)\end{smallmatrix}\right)$ and which passes through the point $\left(\begin{smallmatrix} v_\ell^s(\lm_s,c_s) \\ g(v_\ell^s(\lm_s,c_s);\lm_s,c_s)\end{smallmatrix}\right)$. He will be called the \emph{left-hand} soliton, it is the dark one. We do the same for the other soliton, $U_r$, which will be called the \emph{right-hand} soliton. It has $U_s$ as an endstate and passes through the point $\left(\begin{smallmatrix} v_r^s(\lm_s,c_s) \\ g(v_r^s(\lm_s,c_s);\lm_s,c_s)\end{smallmatrix}\right)$. It is the bright soliton. We are interested, for $\mug = (\mu,\lm,c) \in \Omega$, in the stability of the two-pulse periodic wave with profile $\un{U}$ and period $X$, parametrized by $\mug$, and more precisely when $\mug$ is close to $\mug_{s\ast} = (\mu_s(\lm_s,c_s),\lm_s,c_s)$ in $\Omega$ so when $X$ is large. 

\subsection{Main result and strategy of proof}

To study the instability of our waves, we will use two different approaches. To explain it, we first need to recall some stability/instability criteria for periodic waves and solitary waves. The first one is for periodic waves and is taken from \cite{benzoni-gavageCoperiodicStabilityPeriodic2016}. To this end, for a periodic wave parametrized by $(c,\lm,\mu) \in \R\times \R^N \times \R$ with profile $\un{U}$ and period $X$, as in \cite{benzoni-gavageStabilityPeriodicWaves2014,benzoni-gavageCoperiodicStabilityPeriodic2016}, we define the abbreviated action integral 
\begin{equation*}
    \Theta(c,\lm,\mu) := \int_0^X \Hm[\un{U}] + cQ(\un{U}) + \lm \cdot \un{U} + \mu \dd x.
\end{equation*}
It has been shown in earlier work that the $(N+2)\times(N+2)$ Hessian matrix $\n_{(\mu,\lm,c)}^2 \Theta$ of the action is important to understand the stability of periodic waves and their modulations, see for example \cite{bronskiIndexTheoremStability2011,johnsonNonlinearStabilityPeriodic2009,benzoni-gavageSlowModulationsPeriodic2014,benzoni-gavageStabilityPeriodicWaves2014,benzoni-gavageCoperiodicStabilityPeriodic2016}. Here we recall the criteria shown in \cite{benzoni-gavageCoperiodicStabilityPeriodic2016}, which is the one that we will use. In the following, $n(\n_{(\mu,\lm,c)}^2 \Theta)$ corresponds to the \emph{negative signature} of $\n_{(\mu,\lm,c)}^2 \Theta$, so it counts the number of negative directions for the associated quadratic form in $\R^{N+2}$.

\begin{theorem}\label{theorem_stability_instability_periodic_wave}
Under Assumption~\ref{hypothèse_Hm}, for a wave locally parametrized by $(c,\lm,\mu) \in \R\times \R^N \times \R$ such that $\p_\mu^2 \Theta$ is nonzero and $\n_{(\mu,\lm,c)}^2 \Theta$ is non singular,
\begin{itemize}[label=\textbullet]
    \item if $n(\n_{(\mu,\lm,c)}^2 \Theta) = N$, then the wave is conditionally orbitally stable with respect to co-periodic perturbations; 
    \item if $n(\n_{(\mu,\lm,c)}^2 \Theta) - N$ is odd, then the wave is spectrally unstable to co-periodic perturbations. More precisely the linearized operator admits a real positive eigenvalue with Floquet exponents $\xi =0$.
\end{itemize}
\end{theorem}

In the first statement "conditionally orbitally stable with respect to co-periodic perturbation" means that a solution to~\eqref{eq_hamiltonienne} stays, as long as it exists, as close as we want, in the energy space, to some translate of the underlying wave. If $N=1$ the energy space is $H^1\left(\R/X\Z  \right)$ and $H^1\left(\R/X\Z  \right)\times L^2\left(\R/X\Z \right)$ if $N=2$. In general, this is not enough to conclude global existence in the energy space because local existence for this type of quasilinear PDE is not known in such low-regularity spaces. Up to our knowledge the best result for $N=1$ is in $H^s\left(\R/X\Z  \right)$ for $s > \frac{7}{2}$, see \cite{mietkaWellposednessQuasilinearKortewegde2017,iandoliCauchyProblemQuasilinear2022} and in $H^{s+1}\left(\R/X\Z \right) \times H^s\left(\R/X\Z \right)$ for $s>\frac{3}{2}$ if $N=2$, see \cite{benzoni-gavageWellposednessOnedimensionalKorteweg2006,benzoni-gavageWellposednessEulerKortewegModel2007}.

Spectral instability means that the linearized equation admits exponentially growing solutions. More precisely, $n(\n_{(\mu,\lm,c)}^2 \Theta) - N$ counts mod 2 the number of positive eigenvalues for the Floquet exponents $\xi = 0$ of the linearized operator. This is done using Evans functions calculations; see \cite{benzoni-gavageCoperiodicStabilityPeriodic2016}.

This type of stability/instability criteria for periodic waves is inspired by the theory of stability for solitary waves in Hamiltonian systems. From this we will only recall the results that we will use here; for reference on the topic, see \cite{grillakisStabilityTheorySolitary1987,debievreOrbitalStabilityAnalysis2015,angulopavaNonlinearDispersiveEquations2009,kapitulaSpectralDynamicalStability2013}. A family of solitary waves can be parametrized by their speed $c$ and endstate $U_s$. For a solitary wave, $U$, we introduce is Boussinesq momentum 
\begin{equation*}
    \mo(c,U_s) = \int_{-\infty}^{\infty} \left(\Hm[U] + cQ( U) + \lm_s \cdot U + \mu_s \right) \dd x,
\end{equation*}
with $\lm_s := -\n_{U}\left(\Hm+cQ\right)(U_s,0) $ and $\mu_s := \left(\Hm + cQ \right)(U_s,0) - \lm_s \cdot U_s $. We have the following stability/instability criteria.

\begin{theorem}\label{theorem_stability_instability_soliton}
Under Assumption~\ref{hypothèse_Hm}, for a solitary wave locally parametrized by $(c,U_s) \in \R \times \R^N$, if we denote $\dd_c$ as the derivative with respect to the speed $c$ of the solitary wave with the final state $U_s$ fixed then
\begin{itemize}[label=\textbullet]
    \item if $\dd_c^2 \mo < 0$, the solitary wave is spectrally unstable, more precisely, the linearized operator have a real positive eigenvalue;
    \item if $\dd_c^2 \mo > 0$, the solitary wave is conditionally orbitally stable.
\end{itemize}
\end{theorem}

As for periodic waves, the first part comes from Evans functions calculations, see \cite{pegoEigenvaluesInstabilitiesSolitary1992} for the case $N=1$ and \cite{benzoni-gavageSpectralTransverseInstability2010} for the case $N=2$. The second part comes from the general theory of Grillakis, Shatah and Strauss \cite{grillakisStabilityTheorySolitary1987,debievreOrbitalStabilityAnalysis2015}.

We can now express our main result on the instability of two-pulse periodic waves.

\begin{theorem}\label{main_results}
Under Assumption~\ref{hypothèse_Hm} and~\ref{hypothèse_potentiel}, for a two-pulse periodic wave of sufficiently large period, if we denote $\mo_\ell$ (respectively $\mo_r$) 
the Boussinesq momentum of the left-hand soliton (respectively right-hand soliton), we have the following result.
\begin{itemize}[label=\textbullet]
    \item If $\dd_c^2 \mo_\ell + \dd_c^2 \mo_r > 0$ then the two-pulse periodic wave is spectrally unstable with respect to co-periodic perturbation. More precisely the linearized operator admits a positive real eigenvalue with Floquet exponent $\xi =0$.
    \item If one of the solitons is spectrally unstable, then the two-pulse periodic wave is spectrally unstable. More precisely the linearized operator admits an eigenvalue in $\C_+ = \{\lambda \in \C \mid \Rel(\lambda) > 0\}$ for each Floquet exponent.
\end{itemize}
\end{theorem}

We first begin with remarks on this result. This result tells us that in all cases, except if the two solitons are spectrally stable and $\dd_c^2 \mo_\ell = \dd_c^2 \mo_r = 0$, two-pulse periodic waves with sufficiently large period are spectrally unstable. The conclusion is only at a linearized level, it is all about the spectrum of the linearized operator around the wave. One may wonder whether some information about what is happening at the non-linear level may be deduced. With our Assumption~\ref{hypothèse_Hm}, if $\ka$ is constant, then we can apply the result of \cite{jinNonlinearModulationalInstability2019}. This tells us that the wave is nonlinearly orbitally unstable for co-periodic perturbation and nonlinearly unstable for localized perturbations. For the general quasilinear case, $\ka$ not constant, up to our knowledge, there is no such result.

We can also ask if there exists an integrable equation that fits in our Assumption~\ref{hypothèse_Hm} and admits two-pulse periodic waves. This is the case for the focusing modified Korteweg-de Vries equation, which is 
\begin{equation}\label{mKdV}\tag{mKdv}
    \p_t v + 6v^2v_x + v_{xxx} = 0.
\end{equation} 
For this equation, our periodic waves are the cnoidal waves, and their stability to co-periodic perturbation has been studied in \cite{kapitulaSpectralOrbitalStability2015} and in \cite{cuiInstabilityBandsPeriodic2025}. For the solitary waves, their stability has been studied in \cite{bonaStabilityInstabilitySolitary1987}. The result is the following. If we fix the endstate of the solitons to be 0, then solitary waves are all orbitally stable, and two-pulse periodic waves of sufficiently large period are spectrally unstable under co-periodic perturbation. So we can see our results as an extension of this one for a quasilinear version of mKdV-type equations.\newline

To prove Theorem~\ref{main_results} we combine two approaches. The first one, inspired from \cite{benzoni-gavageStabilityPeriodicWaves2020}, is to look at the asymptotic expansion of $\n_{(\mu,\lm,c)}^2 \Theta$ when the period goes to infinity and apply Theorem~\ref{theorem_stability_instability_periodic_wave}. This results in Theorem~\ref{instabilité_double_bosse} which shows the first part of Theorem~\ref{main_results}. However, this will not be enough to conclude at the instability of the wave in all cases. This is due to the fact that the criteria of Theorem~\ref{theorem_stability_instability_periodic_wave} only counts mod 2 the number of positive eigenvalues for the Floquet exponents $\xi = 0$ of the linearized operator. The second approach is to study perturbatively the spectrum for all Floquet exponents. More precisely, we link the Evans function of a two-pulse periodic wave to Evans functions of the two solitons that we have at the limit, this is Proposition~\ref{cv_evans_away_essential_spectra}. In the case of a periodic wave that converges to one soliton, this has already been done in \cite{yangConvergencePeriodGoes2019}. In Section~\ref{section_evans_function} we adapt the strategy of \cite{yangConvergencePeriodGoes2019}, to our case. With Proposition~\ref{cv_evans_away_essential_spectra} in hand, we show Corollary~\ref{spectrum_periodic_evans_function} that gives us the second part of Theorem~\ref{main_results}.

\begin{remarque}
    One may ask if alternatively a Krein index method may be used to show our result, as in \cite{bronskiIndexTheoremStability2011}. In most cases, Krein indices are used because we can count mod 2 the number of positive eigenvalues, as done by the criteria of Theorem~\ref{theorem_stability_instability_periodic_wave}. So we except that using such a method would only recover the first part of our theorem. For a detailed discussion on the link of Theorem~\ref{theorem_stability_instability_periodic_wave} with Krein index theory, see remark 3 of \cite{benzoni-gavageCoperiodicStabilityPeriodic2016}. For references on Krein index theory, see \cite{kapitulaCountingEigenvaluesKrein2005,bronskiIndexTheoremStability2011,kapitulaSpectralDynamicalStability2013} and references therein. 
\end{remarque}

\section{Asymptotic expansion of the action}\label{double_bosse}

\subsection{Preparation for the asymptotic}

To conclude whether stability or instability occurs for our periodic waves, motivated by Theorem~\ref{theorem_stability_instability_periodic_wave}, we begin by considering the Hessian of the action. The objective will be, as in \cite{benzoni-gavageStabilityPeriodicWaves2020}, to examine its asymptotic behavior in order to draw conclusions for waves of sufficiently large period.\newline

For $\mug=(\mu,\lm,c) \in \Omega$, we denote by $\un{U}$ the profile of the corresponding periodic wave and by $X = X(\mu,\lm,c)$ its period. We will calculate, when $\mug \rightarrow \mug_{s\ast}$, the asymptotic expansion of the following quantities.
\begin{itemize}[label = \textbullet]
    \item The action
    \begin{equation*}
        \Theta(\mu,\lm,c) = \int_{-\frac{X}{2}}^{\frac{X}{2}} \Hm[\un{U}] + c Q(\un{U}) + \lm\cdot\un{U} + \mu \dd x,
    \end{equation*}
    which, by Assumption~\ref{hypothèse_Hm} and~\eqref{eq_2_onde_progressive}, can be written as
    \begin{equation*}
    \Theta = \int_{-\frac{X}{2}}^{\frac{X}{2}} \un{U}_x \cdot \n_{U_x} \left(\Hm + cQ + \lm\cdot\right)[\un{U}] \dd x 
    = \int_{-\frac{X}{2}}^{\frac{X}{2}} \un{v}_x \p_{v_x}\Em(\un{v},\un{v}_x) \dd x.
    \end{equation*}
    Integrating along the orbit described by $(\un{v},\un{v}_x)$ in the phase portrait and using~\eqref{eq_onde_périodique_2} for $N=2$ and~\eqref{eq_onde_périodique_1} for $N=1$, we obtain
    \begin{equation*}
        \Theta(\mu,\lm,c) = 2 \int_{v_\ell}^{v_r} \sqrt{2\ka(v)\left(\mu - \w(v;\lm,c)\right)} \dd v.
    \end{equation*}
    \item The derivatives of the action which, by Proposition 1 of \cite{benzoni-gavageStabilityPeriodicWaves2014}, are given by
    \begin{align}
        &\p_\mu \Theta = X = \int_{v_\ell}^{v_r} \sqrt{\frac{2\ka(v)}{\mu - \w(v;\lm,c)}} \dd v, \label{period} \\
        &\p_{\lambda_1} \Theta = \int_{-\frac{X}{2}}^{\frac{X}{2}} \un{v} \dd x = \int_{v_\ell}^{v_r} v \sqrt{\frac{2\ka(v)}{\mu - \w(v;\lm,c)}} \dd v, \nonumber \\
        &\p_{\lambda_2} \Theta = \int_{-\frac{X}{2}}^{\frac{X}{2}} g(\un{v};\lm_2,c) \dd x = \int_{v_\ell}^{v_r} g(v;\lambda_2,c) \sqrt{\frac{2\ka(v)}{\mu - \w(v;\lm,c)}} \dd v,\nonumber \\
        &\p_{c} \Theta = \int_{-\frac{X}{2}}^{\frac{X}{2}} Q(\un{U}) \dd x = \int_{v_\ell}^{v_r} q(v;\lambda_2,c) \sqrt{\frac{2\ka(v)}{\mu - \w(v;\lm,c)}} \dd v, \nonumber
    \end{align}
    where $q(v;\lambda_2,c) = Q\left(\begin{smallmatrix} v \\ g(v;\lambda_2,c) \end{smallmatrix} \right) $ if $N=2$ and $q(v;\lambda_2,c) = Q(v)$ if $N=1$.
\end{itemize}
This leads us to introduce a new Lagrangian to study these integrals,
\begin{equation*}
    \la(v;\mu,\lm,c) := \mu - \w(v;\lm,c).
\end{equation*}
From now $\n := \n_{(\mu,\lm,c)}$, so we have
\begin{equation*}
    \n\la(v;\mu,\lm,c) = \begin{pmatrix}1 \\ v \\ g(v;\lambda_2,c) \\ q(v;\lambda_2,c) \end{pmatrix},~\text{if } N=2 \quad \text{and} \quad \n\la(v;\mu,\lm,c) = \begin{pmatrix} 1 \\ v\\Q(v) \end{pmatrix},~\text{if } N=1.
\end{equation*}
From the foregoing, we have
\begin{equation*}
    \n \Theta(\mu,\lm,c) = \int_{v_\ell}^{v_r} \n\la(v;\mu,\lm,c)\sqrt{\frac{2\ka(v)}{\la(v;\mu,\lm,c)}} \dd v.
\end{equation*}

We recall some notation from \cite{benzoni-gavageStabilityPeriodicWaves2020} that we will be using throughout this whole section.

\begin{notation}\label{note_lag}
    For a function $\phi = \phi(v;\mu,\lm,c)$ we note
    \begin{equation*}
        \phi_s = \phi\left(v_s(\lm,c);\mu_s(\lm,c),\lm,c\right).
    \end{equation*}
    For derivatives of $\la$, we introduce
    \begin{equation*}
        \n\la_s = V_s, \quad \n\la_{v,s} = W_s, \quad \n\la_{vv,s} = Z_s.
    \end{equation*}
    In detail, if $N=2$, we have
    \begin{equation*}
        V_s = \begin{pmatrix} 1 \\ v_s \\ g_s \\ q_s \end{pmatrix}, \quad W_s = \begin{pmatrix} 0 \\ 1 \\ g_{v,s} \\ q_{v,s} \end{pmatrix}, \quad 
        Z_s = \begin{pmatrix} 0 \\ 0 \\ g_{v v,s} \\ q_{v v,s} \end{pmatrix},
    \end{equation*}
    and if $N=1$,
    \begin{equation*}
        V_s = \begin{pmatrix} 1 \\ v_s \\ q_s \end{pmatrix}, \quad W_s = \begin{pmatrix} 0 \\ 1 \\ q_{v,s} \end{pmatrix}, \quad 
        Z_s = \begin{pmatrix} 0 \\ 0 \\ q_{v v,s} \end{pmatrix}.
    \end{equation*}
    For the Hessian, if $N=1$, $\n^2 \la = 0$ and if $N=2$ we have
    \begin{equation*}
        \n^2\la_s = - T_s \otimes T_s := -T_sT_s^{\T},
    \end{equation*}
    where $T_s$ is the column vector defined as $$T_s = \frac{1}{\sqrt{\tau(v_s)}} \begin{pmatrix} 0 \\ 0 \\ 1 \\ \frac{v_s}{b} \end{pmatrix}.$$
    For consistency, we will note $T_s = 0$ in the case $N=1$.\\
    Since $v_s$ is a double root of $\la\left(.;\mu_s,\lm,c\right)$ we introduce
    \begin{equation*}
        \ti{R}(v,w,z;\lm,c) = \int_0^1 \int_0^1 t \p_v^2\w\left(w + t(z-w) +t s (v - z);\lm,c\right) \dd s \dd t,
    \end{equation*}
    so, we have
    \begin{equation*}
        \la(v;\mu,\lm,c) = \mu - \w(v,\lm,c) = \mu-\mu_s - (v-v_s)^2 \ti{R}(v , v_s ,v_s;\lm,c). 
    \end{equation*}
    At last, we denote $R_s = \ti{R}(v_s,v_s,v_s;\lm,c) = \frac{1}{2}\p_v^2\w(v_s;\lm,c)$ and
    \begin{equation*}
        \y = \y(v,w,z;\lm,c) = \sqrt{\frac{2\ka(v)}{\abs{\ti{R}(v,w,z;\lm,c)}}}.
    \end{equation*}
\end{notation}

Now we recall some notation of \cite{benzoni-gavageStabilityPeriodicWaves2020} useful to express asymptotic expansions.

\begin{notation}\label{note_dl}
Asymptotic expansions will be expressed in terms of
\begin{equation*}
    a_s = \sqrt{\frac{2}{-\p_v^2\w(v_s;\lm,c)}}, \quad b_s= \frac{1}{3} \frac{\p_v^3\w(v_s;\lm,c)}{\left(\p_v^2\w(v_s;\lm,c) \right)^2},
\end{equation*}
\begin{equation*}
    p_r^s = \frac{1}{\p_v\w(v_r^s(\lm,c);\lm,c)}, \quad p_\ell^s =  \frac{1}{\p_v\w(v_\ell^s(\lm,c);\lm,c)}.
\end{equation*}
\end{notation}

We now detail the strategy that we shall use in this section. We note that all the quantities that we are interested in, are of the form
\begin{equation*}
    I(\phi) := \int_{v_\ell}^{v_r} \frac{\phi(v)}{\sqrt{\mu - \w(v)}} \dd v,
\end{equation*}
where $\phi$ is a smooth function. In Appendix~\ref{dl_intégrale}, we therefore calculate the asymptotic behavior of such an integral for every $\phi$ as a function of $\ve = \sqrt{\mu-\mu_s}$. In the following subsections, we use the result obtained, namely Theorem~\ref{dl_int_phi}, to compute asymptotic expansions of the quantities that we are interested in. 

\subsection{Asymptotic behavior of the action and the period}

We begin by studying the asymptotic behavior of the roots of the function $$\la(v;\mu,\lm,c) = \mu - \w(v;\lm,c).$$ We recall that we denote by $v_\ell$ the root of $\la(.;\mu,\lm,c)$ lower than $v_s$ and $v_r$ the upper root.

\begin{proposition}\label{dl_racines}
    With Notations~\ref{note_lag},~\ref{note_dl} and Assumption~\ref{hypothèse_potentiel}, for $\ve = \sqrt{\mu - \mu_s} \rightarrow 0$, we have the following asymptotic expansion
    \begin{equation*}
        \left\{
        \begin{array}{ll}
            v_r =& v_r^s + p_r^s \ve^2 + \go(\ve^4),\\
            v_{\ell} =& v_\ell^s + p_\ell^s \ve^2 + \go(\ve^4).
        \end{array}
        \right.
    \end{equation*}
\end{proposition}

\begin{proof}
We simply apply the implicit function theorem to $\w(\cdot;\lm,c)$ using that $\p_v \w(v_r^s;\lm,c) \neq 0$ and $\p_v \w(v_\ell^s;\lm,c) \neq 0$.
\end{proof}

Now we calculate the asymptotic behavior of the period and the action. But first, if we denote by $U_\ell$ the left-hand soliton and $U_r$ the right-hand soliton, their Boussinesq momentums are given by
\begin{equation*}
    \mo_\ell = \int_{-\infty}^{\infty} \left(\Hm[ U_\ell] + cQ( U_\ell) + \lm\cdot U_\ell + \mu_s \right) \dd x,
\end{equation*}
\begin{equation*}
    \mo_r = \int_{-\infty}^{\infty} \left(\Hm[ U_r] + cQ( U_r) + \lm\cdot U_r + \mu_s \right) \dd x.
\end{equation*}
Using that $\leg\left(\Hm + cQ + \lm\cdot \right)[ U_\ell] =\leg\left(\Hm + cQ + \lm\cdot \right)[ U_r] = \mu_s $ and integrating along the orbits described by $v_r$ and $v_\ell$ in the phase portrait, one gets
\begin{equation*}
    \mo_\ell = 2 \int_{v_\ell^s}^{v_s} \sqrt{2\ka(v)\left(\mu_s -\w(v;\lm,c)\right)} \dd v,
\end{equation*}
\begin{equation*}
    \mo_r = 2 \int_{v_s}^{v_r^s} \sqrt{2\ka(v)\left(\mu_s -\w(v;\lm,c)\right)} \dd v.
\end{equation*}

\begin{proposition}\label{dl_action_période}
With Notation~\ref{note_lag},~\ref{note_dl} and Assumption~\ref{hypothèse_potentiel} we have, for $\ve = \sqrt{\mu-\mu_s} \rightarrow 0$, the following asymptotic expansion
\begin{equation*}
    \Theta \tend{\ve}{0} \mo_r + \mo_\ell,
\end{equation*}
and 
\begin{equation*}
    X = - 2a_s\sqrt{2\ka(v_s)}\ln\ve + \go(1).
\end{equation*}
\end{proposition}

\begin{proof}
    We begin with the period, we already know that
    \begin{equation*}
        X = \int_{v_\ell}^{v_r} \sqrt{\frac{2\ka(v)}{\mu-\w(v;\lm,c)}} \dd v = I(\sqrt{2\ka}),
    \end{equation*}
    so with Theorem~\ref{dl_int_phi} we have
    \begin{equation*}
        X = -2\sqrt\frac{4\ka(v_s)}{-\p_v^2\w(v_s;\lm,c)}\ln\ve + \go(1).
    \end{equation*}
    The result follows from $a_s = \sqrt{-2/\p_v^2\w(v_s;\lm,c)}$.\newline
    For the action, it is sufficient to write
    \begin{align*}
        \Theta  &= 2 \int_{v_\ell}^{v_r} \sqrt{2\ka(v)\left(\mu -\w(v;\lm,c)\right)}\dd v\\
                &= 2 \int_{v_s}^{v_r} \sqrt{2\ka(v)\left(\mu -\w(v;\lm,c)\right)}\dd v + 2 \int_{v_s}^{v_r} \sqrt{2\ka(v)\left(\mu -\w(v;\lm,c)\right)}\dd v\\
                &=:\Theta_\ell + \Theta_r.
    \end{align*}
    For $\Theta_\ell$, respectively $\Theta_r$, we carry out the change of variable $\V_\ell = v_s + \sig(v_\ell - v_s)$, respectively $\V_r = v_s + \sig(v_r - v_s)$.
    We get
    \begin{align*}
        \Theta_\ell &= 2(v_s-v_\ell)\int_0^1 \sqrt{\ka(\V_\ell)\left(\mu - \w(\V_\ell;\lm,c)\right)} \dd \sig, \\
        \Theta_r &= 2(v_r-v_s)\int_0^1 \sqrt{\ka(\V_r)\left(\mu - \w(\V_r;\lm,c)\right)} \dd \sig.
    \end{align*}
    By the dominated convergence theorem, we have
    \begin{align*}
        \Theta_\ell &\tend{\ve}{0} 2(v_s-v_\ell^s)\int_0^1 \sqrt{\ka(\V_\ell^s)\left(\mu - \w(\V_\ell^s;\lm,c)\right)} \dd \sig, \\
        \Theta_r &\tend{\ve}{0} 2(v_r^s-v_s)\int_0^1 \sqrt{\ka(\V_r^s)\left(\mu - \w(\V_r^s;\lm,c)\right)} \dd \sig,
    \end{align*}
    where $\V_\ell^s = v_s + \sig(v_\ell^s - v_s)$ and $\V_r^s = v_r^s + \sig(v_r^s - v_s)$.\\
    By undoing the change of variable for the limit, $v = \V_\ell^{s}$ and $v = \V_r^{s}$, we obtain that
    \begin{equation*}
        \Theta_\ell \tend{\ve}{0} \mo_\ell \quad \text{and} \quad \Theta_r \tend{\ve}{0} \mo_r.
    \end{equation*}
\end{proof}

\subsection{Asymptotic behavior of the Hessian of the actions}

As in \cite{benzoni-gavageStabilityPeriodicWaves2020}, to compute the asymptotic expansion of the Hessian, we begin by studying the asymptotic behavior of the gradient of the action, which is repeated here
\begin{equation*}
    \n\Theta = \int_{v_\ell}^{v_r} \n \la(v) \sqrt{\frac{2\ka(v)}{\mu - \w(v;\lm,c)}} \dd v.
\end{equation*}
So, using the notation of Appendix~\ref{dl_intégrale}, $\n\Theta = I(\n\la\sqrt{2\ka})$, we may use Theorem~\ref{dl_int_phi} to derive the asymptotic expansion of $\n\Theta$. The result is the following.

\begin{proposition}\label{dl_grad_theta_2_bosse}
    With Notations~\ref{note_lag},~\ref{note_dl} and Assumption~\ref{hypothèse_potentiel}, for $\ve = \sqrt{\mu-\mu_s} \rightarrow 0$, we have the following asymptotic expansion
    \begin{equation*}
    \n \Theta = A_0\ln(\ve) + B_0 + B_1\ve + A_2\ve^2 \ln(\ve) + B_2 \ve^2 + \go\left(\ve^3\ln(\ve)\right),
    \end{equation*}
    where each coefficient is given by
    \begin{enumerate}[label=\roman*)]
        \item $A_0 =-2 \y_{s} V_s$,
        \item $B_0 = \int_{0}^{1} \frac{f(\V_\ell^s,v_s,v_s)-\y_sV_s}{\sigma} \dd \sig + \int_{0}^{1} \frac{f(\V_r^s,v_s,v_s)-\y_sV_s}{\sigma} \dd \sig
               + \ln\left(4(v_r^s -v_s)(v_s-v_\ell^s)(-R_s)\right)\y_sV_s,$
        \item $B_1 = -2 a_s^2\y_{s}V_s,$
        \item $A_2 = \left(\frac{3\ka_{v,s}R_{v,s}}{4\sqrt{2\ka_s}(-R_s)^{5/2}}+\frac{3\sqrt{2\ka_s}R_{vv,s}}{4(-R_s)^{5/2}} + \frac{15\sqrt{2\ka_s}R_{v,s}^2}{8(-R_s)^{7/2}}\right)V_s 
                + \frac{3}{4}\frac{\sqrt{2\ka_s}R_{vv,s}}{4(-R_s)^{5/2}}W_s,$
        \item $B_2 = B_2(\n\la\sqrt{2\ka}),$
    \end{enumerate}
    with $f(v,w,z) = \n\la(v)\sqrt{\frac{2\ka(v)}{\abs{R(v,w,z)}}}$, $\V_\ell^s = v_s + \sig(v_\ell^s-v_s)$ and $\V_r^s = v_s+\sig(v_r^s-v_s)$.
\end{proposition}

\begin{proof}
As mentioned above, we simply use Theorem~\ref{dl_int_phi} with $\phi = \n\la\sqrt{2\ka}$. We only need to simplify each coefficient, for this we repeat that 
$$a_s = \sqrt{\frac{-2}{\p_{v}^2 \w(v_s;c,\lm)}} = \frac{1}{\sqrt{-R_s}},$$
which gives
$$a_s\sqrt{2\ka_s} = \sqrt{\frac{-2\ka_s}{R_s}} = \y_{s}.$$
So for the first coefficient we have
\begin{equation*}
    A_0 = -2\n \la_s \sqrt{\frac{2\ka_s}{-R_s}} = -2\y_s V_s,
\end{equation*}
and also
\begin{equation*}
    B_1 = - 2\n \la_s \frac{\sqrt{2\ka_s}}{-R_s} = - 2 a_s^2\sqrt{2\ka_s}V_s = - 2 a_s\y_sV_s.
\end{equation*}
Using that $\n\la_{v,s} = W_s$ we deduce the coefficient $A_2$. Concerning $B_2$, it can be made explicit using Theorem~\ref{dl_int_phi}, but since we will not use its exact form, we do not detail its calculation here. We still need to calculate $B_0$ which is given by
\begin{equation*}
\begin{aligned}[t]B_0 =&\int_{v_\ell^s}^{v_s-\eta} \n\la(v) \frac{\sqrt{2\ka(v)}}{\sqrt{\mu_s - \w(v)}} \dd v
        + \int_{v_s+\eta}^{v_r^s} \n\la(v) \frac{\sqrt{2\ka(v)}}{\sqrt{\mu_s - \w(v)}} \dd v \\
        &+ \int_{v_s}^{v_s+\eta} \frac{1}{v-v_s}\left(\n\la(v)\frac{\sqrt{2\ka(v)}}{\sqrt{-R(v)}}-\frac{\sqrt{2\ka_s}}{\sqrt{-R_s}}V_s\right) \dd v
        + \int_{v_s-\eta}^{v_s} \frac{-1}{v-v_s}\left(\n\la(v)\frac{\sqrt{2\ka(v)}}{\sqrt{-R(v)}}-\frac{\sqrt{2\ka_s}}{\sqrt{-R_s}}V_s\right)\dd v\\
        &+ 2\frac{\sqrt{2\ka_s}}{\sqrt{-R_s}}\ln(2\eta\sqrt{-R_s})V_s.
    \end{aligned}
\end{equation*}
We use that $\mu_s - \w(v;c,\lm) = -(v-v_s)^2R(v,v_s,v_s;\lm,c)$ and that $f(v,w,z) = \n\la(v)\sqrt{\frac{2\ka(v)}{\abs{R(v,w,z)}}}$,
\begin{equation*}
\begin{aligned}[t]B_0 =&\int_{v_\ell^s}^{v_s-\eta} \left(\frac{f(v,v_s,v_s)}{-(v-v_s)}-\frac{\y_s}{-(v-v_s)}V_s \right)\dd v
        + \int_{v_s+\eta}^{v_r^s} \left(\frac{f(v,v_s,v_s)}{(v-v_s)}-\frac{\y_s}{(v-v_s)}V_s \right)\dd v \\
        &+ \int_{v_s-\eta}^{v_s} \left(\frac{f(v,v_s,v_s)}{-(v-v_s)}-\frac{\y_s}{-(v-v_s)}V_s\right)\dd v
        + \int_{v_s}^{v_s+\eta} \left(\frac{f(v,v_s,v_s)}{(v-v_s)}-\frac{\y_s}{(v-v_s)}V_s \right) \dd v\\
        & + \int_{v_\ell^s}^{v_s-\eta} \frac{\y_s}{-(v-v_s)}V_s \dd v + \int_{v_s+\eta}^{v_r^s} \frac{\y_s}{(v-v_s)}V_s \dd v
        + 2\y_s\ln(2\eta\sqrt{-R_s})V_s.
    \end{aligned}
\end{equation*}
For the integral on the left, we perform the change of variable $\V_\ell^s = v_s + \sig(v_\ell^s-v_s)$ and for the integral on the right $\V_r^s = v_s+\sig(v_r^s-v_s)$. This gives
\begin{equation*}
\begin{aligned}[t]B_0 =&\int_{0}^{1} \frac{f(\V_\ell^s,v_s,v_s)-\y_sV_s}{\sig}\dd \sig
        + \int_{0}^{1} \frac{f(\V_r^s,v_s,v_s)-\y_sV_s}{\sig}\dd \sig \\
        &+\ln\left(\frac{v_r^s-v_s}{\eta}\right)\y_sV_s - \ln\left(\frac{\eta}{v_s-v_\ell^s}\right)\y_sV_s + 2\ln(2\eta\sqrt{-R_s})\y_sV_s,
    \end{aligned}
\end{equation*}
so we have that
\begin{equation*}
   B_0 = \int_{0}^{1} \frac{f(\V_\ell^s,v_s,v_s)-\y_sV_s}{\sigma} \dd \sig + \int_{0}^{1} \frac{f(\V_r^s,v_s,v_s)-\y_sV_s}{\sigma} \dd \sig
               + \ln\left(4(v_r^s -v_s)(v_s-v_\ell^s)(-R_s)\right)\y_sV_s.
\end{equation*}
\end{proof}

From this asymptotic expansion, we deduce the asymptotic behavior of the Hessian of the action. This is the content of the following theorem.

\begin{theorem}\label{dl_hess_theta_2_bosse}
With Notation~\ref{note_lag},~\ref{note_dl} and Assumption~\ref{hypothèse_potentiel}, for $\ve = \sqrt{\mu-\mu_s} \rightarrow 0$, we have the following asymptotic expansion
\begin{equation*}
    \n^2 \Theta = D_0 \ve^{-2} + D_1 \ve^{-1} + C_2 \ln(\ve) + D_2 + \go(\ve\ln(\ve)),
\end{equation*}
where
\begin{enumerate}[label=\roman*)]
    \item $D_0 = -\y_s V_s \otimes V_s,$
    \item $D_1 = -a_s \y_s V_s \otimes V_s,$
    \item $C_2 =\begin{aligned}[t]& \frac{a_s^2}{2} \y_s V_s \otimes Z_s + (2b_s\y_s + a_s^2\y_{v,s})V_s \otimes W_s + 2\y_s T_s\otimes T_s 
                + a_s^2 \y_s W_s \otimes W_s\\ 
                &+ \left(\frac{3\ka_{v,s}R_{v,s}}{4\sqrt{2\ka_s}(-R_s)^{5/2}}+\frac{3\sqrt{2\ka_s}R_{vv,s}}{4(-R_s)^{5/2}} 
                + \frac{15\sqrt{2\ka_s}R_{v,s}^2}{8(-R_s)^{7/2}}\right)V_s \otimes V_s
                + \frac{3}{4}\frac{\sqrt{2\ka_s}R_{vv,s}}{(-R_s)^{5/2}}W_s\otimes V_s,\end{aligned}$
\end{enumerate}
and
\begin{equation*}
S_s\cdot D_2S_s = \dd_c^2 \mo_\ell + \dd_c^2 \mo_r,
\end{equation*}
with $S_s = (-q_s, (B^{-1}U_s)^{\T}, -1)^{\T}$.
\end{theorem}

We recall that we denote $\mo_r$, respectively $\mo_\ell$, the Boussinesq momentum for the right-hand soliton, respectively left-hand soliton and 
$\dd_c$ the derivative with respect to the speed $c$ of the solitary wave with the final state $U_s$ fixed.

\begin{proof}
    Using the expression for $\n \Theta$, we have 
    \begin{align*}
        \n^2 \Theta =& \int_{0}^{1} \n^2 \la(\V) \sqrt{\frac{2\ka(\V)}{\mu - \w(\V;\lm,c)}} \dd \sig + \int_{0}^{1} \n \la_v(\V) \otimes \left((1-\sig)\n v_\ell + \sig \n v_r \right) \sqrt{\frac{2\ka(\V)}{\mu - \w(\V;\lm,c)}} \dd \sig \\
                     &+ \int_{0}^{1} \n \la(\V) \otimes \begin{pmatrix} -1 \\ \n_{\lm} \w(\V;\lm,c) \\ \p_c \w(\V;\lm,c) \end{pmatrix} \frac{\sqrt{2\ka(\V)}}{2\left(\mu - \w(\V;\lm,c)\right)^{3/2}} \dd \sig \\ &-
                       \int_{0}^{1} \n \la(\V) \otimes \left((1-\sig)\n v_\ell + \sig \n v_r \right) \frac{\sqrt{2\ka(\V)}}{2\left(\mu - \w(\V;\lm,c)\right)^{3/2}} \dd \sig,
    \end{align*}
    where $\V = (1-\sig) v_\ell + \sig v_r $.

    We already know, from Appendix~\ref{dl_intégrale}, that for $\phi$ a smooth function $$I(\phi) = \int_{v_\ell}^{v_r} \frac{\phi(v)}{\sqrt{\mu - \w(v)}} \dd v$$ admits an asymptotic expansion in $\ve^k,~k\geq 0$ and $\ve^{k}\ln(\ve),~k\geq0$. By using multiple Taylor expansions as in the proof of Theorem~\ref{dl_int_phi}, we conclude that $$\int_{v_\ell}^{v_r} \frac{\phi(v)}{\left(\mu - \w(v)\right)^{3/2}} \dd v$$ admits an asymptotic expansion in $\ve^k,~k\geq -2$ and $\ve^{k}\ln(\ve),~k\geq0$. So combining this with Proposition~\ref{dl_racines}, we obtain that $\n^2 \Theta$ admits an asymptotic expansion in $\ve^k,~k\geq -2$ and $\ve^{k}\ln(\ve),~k\geq0$. To compute each coefficient of the expansion we will use the expansion of $\n \Theta$, Proposition~\ref{dl_grad_theta_2_bosse}.

    First, we need to calculate $\n \ve$, we know that
    \begin{equation*}
        \ve = \sqrt{\mu -\mu_s} = \sqrt{\mu - \w(v_s;\lm,c)} = \sqrt{\la(v_s;\mu,\lm,c)},
    \end{equation*}
    so using that $\p_v\la(v_s;\mu,\lm,c) = - \p_v\w(v_s;\lm,c) = 0$, we have
    \begin{equation*}
        \n \ve = \frac{1}{2\sqrt{\la(v_s;\mu,\lm,c)}}\n \la_s = \frac{1}{2\ve} V_s.
    \end{equation*}

    Furthermore, with Proposition~\ref{dl_grad_theta_2_bosse}, we also know that
    \begin{equation*}
        \n \Theta = A_0\ln(\ve) + B_0 + B_1\ve + A_2\ve^2 \ln(\ve) + B_2 \ve^2 + \go\left(\ve^3\ln(\ve)\right).
    \end{equation*}
    So, as in \cite{benzoni-gavageStabilityPeriodicWaves2020}, if we denote by $\n_{\text{tot}}$ the 'total gradient' with respect to $(\mu,\lm,c)$, which involves the regular gradient $\n$ but also $\n v_s,~\n v_r^s$ and $\n v_\ell^s$ by the chain rule, we get
    \begin{equation*}
        \n^2\Theta = \begin{aligned}[t]&\left(A_0\ve^{-1} + B_1 + 2A_2\ve\ln(\ve) + (A_2+2B_2)\ve + \go(\ve^2\ln(\ve))\right)\otimes(\frac{1}{2}\ve^{-1}V_s)\\
                    &+ \n_{\text{tot}}A_0\ln(\ve) + \n_{\text{tot}}B_0 + \go\left(\ve\ln(\ve)\right).\end{aligned}
    \end{equation*} 
    Therefore,
    \begin{equation*}
        \n^2 \Theta = D_0 \ve^{-2} + D_1 \ve^{-1} + C_2 \ln(\ve) + D_2 + \go(\ve\ln(\ve)),
    \end{equation*}
    where coefficients are given by
    \begin{enumerate}[label=\roman*)]
        \item $D_0 = \frac{1}{2} A_0 \otimes V_s,$
        \item $D_1 = \frac{1}{2}B_1\otimes V_s,$
        \item $C_2 = \n_{\text{tot}}A_0 + A_2 \otimes V_s,$
        \item $D_2 = \n_{\text{tot}}B_0 + \frac{1}{2}(A_2+2B_2)\otimes V_s.$
    \end{enumerate}
    We only have to identify each coefficient in the same way as in Theorem 5 of \cite{benzoni-gavageStabilityPeriodicWaves2020}. We recall that Proposition~\ref{dl_grad_theta_2_bosse} gives that $A_0 = -2\y_sV_s$ and $B_1 = -2a_s\y_sV_s$, so for the first coefficient we have
    \begin{equation*}
        D_0 = -\y_s V_s \otimes V_s \quad \text{and} \quad D_1 = - a_s \y_s V_s \otimes V_s.
    \end{equation*}
    To determine $C_2$ we need to calculate $ \n_{\text{tot}}A_0$. Since $A_0 = -2\y_sV_s$,
    \begin{equation*}
    \n_{\text{tot}} A_0 = - 2 V_s \otimes \ntot \y_s - 2\y_s\ntot V_s.
    \end{equation*} 
    But $ V_s = \n\la_s$, so by Notation~\ref{note_lag}
    \begin{equation*}
        \n_{\text{tot}}V_s = - T_s\otimes T_s  + W_s \otimes \n v_s.
    \end{equation*}
    To compute $\n v_s$ we use that $v_s$ is a critical point of the potential, $\p_v\w(v_s(c,\lm);\lm,c) = 0$. Taking the gradient of this expression, we get
    \begin{equation*}
        \p_v^2\w(v_s;\lm,c) \n v_s + \n\p_v\w(v_s;\lm,c) = 0 \quad \text{so} \quad \n v_s = - \frac{a_s^2}{2} W_s.
    \end{equation*}
    This gives 
    \begin{equation*}
        \n_{\text{tot}}V_s = - T_s\otimes T_s -\frac{a_s^2}{2}W_s \otimes W_s.
    \end{equation*}
    We now compute $\ntot \y_s$. Due to the symmetry of $\y$ in $w$ and $z$, 
    \begin{equation*}
        \ntot \y_s = \n \y_s + \left(\y_{v,s} + 2\y_{w,s} \right)\n v_s.
    \end{equation*}
    For the first term, since in a neighborhood of $(v_s,v_s,v_s)$
    \begin{equation*}
        \y(v,w,z;\lm,c) = \sqrt{\frac{2\ka(v)}{-R(v,w,z;\lm,c)}},
    \end{equation*}
    we get that
    \begin{equation*}
        \frac{\n \y_s}{\y_s} = -\frac{\n R_s}{2 R_s}.
    \end{equation*}
    But we know that $R_s = -\frac{1}{a_s^2}$ and that 
    \begin{equation*}
        \n R_s = \int_0^1 \int_0^1 t\n(\p_v^2\w)(v_s;\lm,c) \dd u \dd t = -\frac{1}{2} Z_s,
    \end{equation*}
    so 
    \begin{equation*}
        \n \y_s =  -\frac{a_s^2}{4} \y_s Z_s.
    \end{equation*}
    In the same spirit, $ \frac{ \y_{z,s}}{\y_s} = -\frac{ R_{z,s}}{2 R_s}$ and
    \begin{equation*}
        R_{z,s} = \int_0^1 \int_0^1 t^2(1-u) \p_v^3\w(v_s;\lm,c) \dd u \dd t = \frac{12b_s}{a_s^4} \int_0^1t^2\dd t \int_0^1 (1-u) \dd u = \frac{2b_s}{a_s^4}, 
    \end{equation*}
    where we recall that $b_s = \frac{1}{3} \frac{\p_v^3\w(v_s;\lm,c)}{(\p_v^2\w(v_s;\lm,c))^2}$. \emph{In fine}, we get
    \begin{equation*}
        \n_{\text{tot}}\y_s = - \frac{a_s^2}{4} \y_s Z_s - (b_s\y_s + \frac{a_s^2}{2}\y_{v,s})W_s.
    \end{equation*}
    Adding up the different terms, we obtain that
    \begin{equation*}
        \n_{\text{tot}}A_0  = \frac{a_s^2}{2} \y_s V_s \otimes Z_s + (2b_s\y_s + a_s^2\y_{v,s})V_s \otimes W_s + 2\y_s T_s\otimes T_s 
                + a_s^2 \y_s W_s \otimes W_s.
    \end{equation*}
    Since we know $A_2$, we get for $C_2$,
    \begin{equation*}
        C_2 =\begin{aligned}[t]& \frac{a_s^2}{2} \y_s V_s \otimes Z_s + (2b_s\y_s + a_s^2\y_{v,s})V_s \otimes W_s + 2\y_s T_s\otimes T_s 
                + a_s^2 \y_s W_s \otimes W_s\\ 
                &+ \left(\frac{3\ka_{v,s}R_{v,s}}{4\sqrt{2\ka_s}(-R_s)^{5/2}}+\frac{3\sqrt{2\ka_s}R_{vv,s}}{4(-R_s)^{5/2}} 
                + \frac{15\sqrt{2\ka_s}R_{v,s}^2}{8(-R_s)^{7/2}}\right)V_s \otimes V_s
                + \frac{3}{4}\frac{\sqrt{2\ka_s}R_{vv,s}}{(-R_s)^{5/2}}W_s\otimes V_s.\end{aligned}
    \end{equation*}
    We are left to compute $S_s\cdot D_2S_s$. To do so we use that $S_s$ is orthogonal to $V_s, W_s$ and $T_s$. This follows from Lemma 1 of \cite{benzoni-gavageStabilityPeriodicWaves2020} that we recall later on (Lemma~\ref{lemme_orthogonalité}). This gives
    \begin{align*}
        S_s\cdot D_2S_s 
                    =& S_s.\n_{\text{tot}}B_0 S_s \\
                    =& \int_{0}^{1} S_s\cdot\frac{\n_{\text{tot}}f(\V_\ell^s,v_s,v_s)-\n_{\text{tot}}(\y_sV_s)}{\sigma}S_s \dd \sig
                     + \int_{0}^{1} S_s\cdot\frac{\n_{\text{tot}}f(\V_r^s,v_s,v_s)-\n_{\text{tot}}(\y_sV_s)}{\sigma}S_s \dd \sig\\
                     &+ S_s\cdot\left(V_s \otimes \n_{\text{tot}}\left[\ln\left(4(v_r^s -v_s)(v_s-v_\ell^s)(-R_s)\right)\y_s\right]\right)S_s\\
                     &+\ln\left(4(v_r^s -v_s)(v_s-v_\ell^s)(-R_s)\right)\y_s~S_s\cdot \n_{\text{tot}}V_s S_s.
    \end{align*}
    But we already know that
    \begin{align*}
    &\n_{\text{tot}}V_s = - T_s\otimes T_s -\frac{a_s^2}{2}W_s \otimes W_s, \\
    &\n_{\text{tot}}\y_s = - \frac{a_s^2}{4} \y_s Z_s - (b_s\y_s + \frac{a_s^2}{2}\y_{v,s})W_s,
    \end{align*}
    and since $S_s$ is orthogonal to $V_s$, $T_s$ and $W_s$, we get that
    \begin{align*}
        S_s\cdot D_2S_s =& \int_{0}^{1} S_s\cdot\frac{\n_{\text{tot}}f(\V_\ell^s,v_s,v_s)}{\sigma}S_s \dd \sig
                     + \int_{0}^{1} S_s\cdot \frac{\n_{\text{tot}}f(\V_r^s,v_s,v_s)}{\sigma}S_s \dd \sig\\
                        =& \int_{0}^{1} S_s\cdot\left(\n f^{s,\ell}+f_v^{s,\ell}\otimes\n v_\ell^s\right)S_s \frac{\dd \sig}{\sig}
                     + \int_{0}^{1} S_s\cdot \left(\n f^{s,r}+f_v^{s,r}\otimes\n v_r^s\right)S_s \frac{\dd \sig}{\sig}.
    \end{align*}
    Here, we used that $\n v_s = -\frac{a_s^2}{2}W_s$, and we denoted, for any functions $g$, $g^{s,\ell} = g(\V_\ell^s,v_s,v_s)$, and $g^{s,r} = g(\V_r^s,v_s,v_s)$.\newline

    All that remains is to identify each integral. We show that
    \begin{equation*}
        \dd_c^2 \mo_\ell = \int_{0}^{1} S_s\cdot\left(\n f^{s,\ell}+f_v^{s,\ell}\otimes\n v_\ell^s\right)S_s \frac{\dd \sig}{\sig} \quad \text{and} \quad
        \dd_c^2 \mo_r = \int_{0}^{1} S_s\cdot \left(\n f^{s,r}+f_v^{s,r}\otimes\n v_r^s\right)S_s \frac{\dd \sig}{\sig}.
    \end{equation*}
    The calculations are similar for both integrals and are identical to the calculations in the proof of Theorem~5 of \cite{benzoni-gavageStabilityPeriodicWaves2020}. We will only give details for the Boussinesq momentum of the left-hand soliton.\newline

    We introduce $U_\ell$ the left-hand soliton with end state $U_s$ and velocity $c$. The Boussinesq momentum is given by
    \begin{equation*}
        \mo_\ell = \int_{-\infty}^{\infty} \left(\Hm[U_\ell] + cQ(U_\ell) + \lm \cdot U_\ell + \mu_s \right) \dd x,      
    \end{equation*}
    with $\lm = -\n_{U}\left(\Hm+cQ\right)(U_s,0)$ and $\mu_s = -\lm\cdot U_s - (\Hm + cQ)(U_s,0)$.\\
    Since $\dd_c$ is the derivative with respect to the speed $c$ with the end state $U_s$ fixed, we have
    \begin{align*}
        &\dd_c \lm = -\n_{U}Q(U_s) = - B^{-1}U_s,\\
        &\dd_c \mu_s = U_s\cdot B^{-1}U_s - Q(U_s) = Q(U_s).
    \end{align*}
    Furthermore, we know that $\delta\left(\Hm + cQ \right)[U_\ell] + \lm = 0$, so we obtain that
    \begin{align*}
        \dd_c \mo_\ell =& \int_{-\infty}^{\infty} \left(Q(U_\ell)+\delta\left(\Hm + cQ \right)[U_\ell]\cdot \dd_c U_\ell + \lm\cdot\dd_c U_\ell + \dd_c\lm\cdot U_\ell + \dd_c\mu_s \right)\dd x\\
                       =& \int_{-\infty}^{\infty} Q(U_\ell) + Q(U_s) - U_\ell\cdot B^{-1}U_s \dd x\\
                       =& 2 \int_{v_\ell^s}^{v_r^s} \left(Q(\mathcal{U}) + Q(U_s) - \mathcal{U}\cdot B^{-1} U_s \right)\sqrt{\frac{\ka(v)}{2(\mu_s-\w(v;\lm,c))}} \dd v,
    \end{align*}
    where $\mathcal{U}(v) = v$ if $N=1$ and $\mathcal{U}(v) = (v, g(v;\lm,c))^{\T}$ if  $N=2$.\newline 
    However, we know that, for $v\in \oo{v_\ell^s}{v_s}$,
    \begin{equation*}
        \frac{\y(v,v_s,v_s;\lm,c)}{-2(v-v_s)} = \sqrt{\frac{\ka(v)}{2(\mu_s - \w(v;\lm,c))}}.
    \end{equation*}
    So, through the change of variable $\V_\ell^s = v_s + \sig(v_\ell^s - v_s)$ we derive
    \begin{equation*}
        \dd_c \mo_\ell = \int_{0}^{1} \left(Q(\mathcal{U}^{s,\ell}) + Q(U_s) - \mathcal{U}^{s,\ell}\cdot B^{-1} U_s \right)\y^{s,\ell} \frac{\dd \sig}{\sig}.
    \end{equation*}
    In addition, we know that $f^{s,\ell} =\y^{s,\ell}\n\la^{s,\ell} = \y^{s,\ell}\left(1, (\mathcal{U}^{s,\ell})^{\T}, Q(\mathcal{U}^{s,\ell}) \right)^{\T}$ and $S_s = \left(-Q(U_s), (B^{-1}U_s)^{\T}, -1 \right)^{\T}.$
    This gives
    \begin{equation*}
        S_s\cdot f^{s,\ell} = -\left(Q(U_s) + Q(\mathcal{U}^{s,\ell})-\mathcal{U}^{s,\ell}\cdot B^{-1}U_s\right)\y^{s,\ell},
    \end{equation*}
    so
    \begin{equation*}
        \dd_c \mo_\ell = \int_{0}^{1} - S_s\cdot f^{s,\ell} \frac{\dd \sig}{\sig}.
    \end{equation*}
    Using that $\dd_c\lm = -B^{-1}U_s$ we get that $\dd_c a(c,\lm)= \p_ca + \n_{\lm} a\cdot\dd_c\lm = -\n a \cdot S_s$, hence
    \begin{equation*}
        \dd_c^2 \mo_\ell = \int_{0}^{1} S_s\cdot\left(\n f^{s,\ell}+\sigma f_v^{s,\ell}\otimes\n v_\ell^s\right)S_s \frac{\dd \sig}{\sig}.
    \end{equation*} 
    This concludes the proof.
\end{proof}

\subsection{Conclusion on the action of the periodic waves}

We use the asymptotic expansion of $\n^2\Theta$ to deduce $n(\n^2\Theta)$, for $\ve$ small enough, and conclude, with Theorem~\ref{theorem_stability_instability_periodic_wave}, at the instability of two-pulse periodic waves for sufficiently large period when $\dd_c^2\mo_\ell + \dd_c^2\mo_r > 0$. To compute the signature, we express the matrix in a good basis, as in \cite{benzoni-gavageStabilityPeriodicWaves2020}. For clarity, we recall here the Lemma 1 of \cite{benzoni-gavageStabilityPeriodicWaves2020}, which introduces the basis that we use.

\begin{lemma}\label{lemme_orthogonalité}
    We introduce the following symmetric matrix
    \begin{equation*}
        \sm = \begin{pmatrix}
            0 & 0 & -1 \\
            0 & B^{-1} & 0 \\
            -1 & 0 & 0
        \end{pmatrix}.
    \end{equation*}
    We have $\sm V_s = S_s$ and the following orthogonal relations
    \begin{equation*}
        \left\{ \begin{array}{lll}
            V_s\cdot \sm V_s = 0,~V_s\cdot \sm W_s = 0,~V_s\cdot \sm T_s = 0,~V_s\cdot \sm Z_s = - W_s\cdot\sm W_s,\\
            T_s\cdot \sm T_s = 0,~ T_s\cdot \sm Z_s = 0, \\
            E \cdot V_s = 1,~E\cdot W_s = 0,~E\cdot Z_s = 0,~E_s\cdot T_s = 0,
        \end{array} \right.
    \end{equation*}
    where $E = (1,0,0,0)^{\T}$ if $N=2$ and $E = (1,0,0)^{\T}$ if $N=1$.\newline
    Furthermore, if $N=2$, $\left(\sm V_s, \sm W_s, \sm T_s, E\right)$ is a basis of $\R^4$ and if $N=1$, 
    $\left(\sm V_s, \sm W_s, E\right)$ is a basis of $\R^3$.
\end{lemma}

We move on to the main theorem of this section.

\begin{theorem}\label{instabilité_double_bosse}
    With Assumption~\ref{hypothèse_Hm} and~\ref{hypothèse_potentiel}, for a two-pulse periodic wave of sufficiently large period, we have the following properties.
    \begin{itemize}[label = \textbullet]
        \item Its period $X$ is monotonically decreasing with  $\mu$, that is, $\p_\mu X = \p_\mu^2\Theta < 0$.
        \item If $\dd_c^2\mo_\ell + \dd_c^2\mo_r \neq 0$ then the matrix $\n^2 \Theta$ is invertible.
        \item If in addition $\dd_c^2\mo_\ell + \dd_c^2\mo_r > 0$ then $n(\n^2\Theta) = N+1$ and the corresponding wave is spectrally unstable.
        \item If in addition $\dd_c^2\mo_\ell + \dd_c^2\mo_r < 0$ then the matrix $\n^2\Theta$ is negative-definite.
    \end{itemize}
\end{theorem}

\begin{proof}
The proof of the theorem is similar to that of Corollary~2 in \cite{benzoni-gavageStabilityPeriodicWaves2020}. The idea is to write the matrix $\n^2 \Theta$ in the basis given in Lemma~\ref{lemme_orthogonalité} and then compute the principal minors. To conclude, we only have to use Sylvester's rule to know the negative signature of $\n^2 \Theta$. To use the basis of Lemma~\ref{lemme_orthogonalité}, we introduce the following invertible matrices,
\begin{equation*}
    \Ps = \left(E,\sm V_s,\sm T_s,\sm W_s \right)~\text{if}~N=2, \quad \text{and} \quad \Ps = \left(E,\sm V_s,\sm W_s\right)~\text{if}~N=1.
\end{equation*}
Let us start with the case $N=2$. From Theorem~\ref{dl_hess_theta_2_bosse} and orthogonality relations given by Lemma~\ref{lemme_orthogonalité}, we have
\begin{equation*}
\Ps^{\T}\n^2\Theta\Ps = \begin{pNiceArray}{c|c|c|c}
                            \begin{array}{@{}c@{}}
                            -\y_s\ve^{-2}-a_s\y_s\ve^{-1}\\
                            + \go(\ln\ve)
                            \end{array}   & \go(\ln\ve) & \go(\ln\ve) & \go(\ln\ve)\\
                            \cmidrule(l r){1-4}
                            \go(\ln\ve) &  \begin{array}{@{}c@{}} \dd_c^2\mo_\ell + \dd_c^2\mo_r\\ + \go(\ve\ln\ve) \end{array}&   \go(1)  & \go(1)\\
                            \cmidrule(l r){1-4}
                            \go(\ln\ve) & \go(1) & \begin{array}{@{}c@{}} a_s^2\y_s(T_s\cdot\sm W_s)^2\\ \times\ln\ve + \go(1) \end{array} &\begin{array}{@{}c@{}} a_s^2\y_s(T_s\cdot\sm W_s)(W_s\cdot\sm W_s)\\\times\ln\ve + \go(1)\end{array}\\
                            \cmidrule(l r){1-4}
                            \go(\ln\ve) & \go(1) & \begin{array}{@{}c@{}} a_s^2\y_s(T_s\cdot\sm W_s)(W_s\cdot\sm W_s)\\\times\ln\ve + \go(1)\end{array} & \begin{array}{@{}c@{}c@{}}2\y_s(T_s\cdot\sm W_s)^2\ln\ve\\ + a_s^2\y_s(W_s\cdot\sm W_s)^2\ln\ve\\ +\go(1)\end{array}
                        \end{pNiceArray}.
\end{equation*}
The principal minors of this matrix are
\begin{align*}
    &\Delta_{1\times1} = -\y_s\ve^{-2}-a_s\y_s\ve^{-1} + \go(\ln\ve),\\
    &\Delta_{2\times2} = -\y_s\left(\dd_c^2\mo_\ell + \dd_c^2\mo_r\right)\ve^{-2} + \go(\ve^{-1}\ln\ve), \\
    &\Delta_{3\times3} = -a_s^2\y_s^2\left(\dd_c^2\mo_\ell + \dd_c^2\mo_r\right)(T_s\cdot\sm W_s)^2\frac{\ln\ve}{\ve^2} + \go(\ve^{-2}),\\
    &\Delta_{4\times4} = \det(\Ps)^2\det(\n^2\Theta) = -2a_s^2\y_s^3\left(\dd_c^2\mo_\ell + \dd_c^2\mo_r\right)(T_s\cdot\sm W_s)^4 \frac{(\ln\ve)^2}{\ve^2} + \go(\ve^{-2}\ln\ve).
\end{align*}
The sign of $\p_\mu X$, is obtained from $\y_s > 0$, and
\begin{equation*}
\p_\mu X = \p_\mu^2 \Theta = E^{\T}\n^2\Theta E = -\y_s\ve^{-2}-a_s\y_s\ve^{-1} + \go(\ln\ve),
\end{equation*}
which imply that $\p_\mu X<0$ for $\ve$ small enough.

For other properties, since $T_s\cdot\sm W_s = b/\sqrt{\tau(v_s)} \neq 0$ and $\Ps$ is invertible, we have that, if $\dd_c^2\mo_\ell + \dd_c^2\mo_r \neq 0$ then $\n^2\Theta$ is invertible. In this case, by Sylvester's rule, for $\ve$ small enough, $n(\n^2\Theta)$ equals the number of sign changes in
\begin{equation*}
    \left(+,~-,~-\sign(\dd_c^2\mo_\ell + \dd_c^2\mo_r),~\sign(\dd_c^2\mo_\ell + \dd_c^2\mo_r),~-\sign(\dd_c^2\mo_\ell + \dd_c^2\mo_r)\right).
\end{equation*}
so if $\dd_c^2\mo_\ell + \dd_c^2\mo_r > 0$ we have $n(\n^2\Theta) = 3 = N +1$ and we can conclude to the spectral instability of the wave using Theorem~\ref{theorem_stability_instability_periodic_wave}. In the other case, we get that $n(\n^2\Theta) = 4 = N +2$, so $\n^2\Theta$ is negative-definite and we can not conclude to the stability or instability of the wave.\newline

Let us know look at the case $N=1$, here the matrix is given by
\begin{equation*}
    \Ps^{\T}\n^2\Theta\Ps = \begin{pNiceArray}{c|c|c}
                            \begin{array}{@{}c@{}}
                            -\y_s\ve^{-2}-a_s\y_s\ve^{-1}\\
                            + \go(\ln\ve)
                            \end{array}   & \go(\ln\ve)  & \go(\ln\ve)\\
                            \cmidrule(l r){1-3}
                            \go(\ln\ve) &  \begin{array}{@{}c@{}} \dd_c^2\mo_\ell + \dd_c^2\mo_r\\ + \go(\ve\ln\ve) \end{array} & \go(1)\\
                            \cmidrule(l r){1-3}
                            \go(\ln\ve) & \go(1) & \begin{array}{@{}c@{}c@{}}a_s^2\y_s(W_s\cdot\sm W_s)^2\ln\ve\\ +\go(1)\end{array}
    \end{pNiceArray}.
\end{equation*}
Principal minors are given by
\begin{align*}
    &\Delta_{1\times1} = -\y_s\ve^{-2}-a_s\y_s\ve^{-1} + \go(\ln\ve),\\
    &\Delta_{2\times2} = -\y_s\left(\dd_c^2\mo_\ell + \dd_c^2\mo_r\right)\ve^{-2} + \go(\ve^{-1}\ln\ve), \\
    &\Delta_{4\times4} = \det(\Ps)^2\det(\n^2\Theta) = -2a_s^2\y_s^2\left(\dd_c^2\mo_\ell + \dd_c^2\mo_r\right)(W_s\cdot\sm W_s)^4 \frac{\ln\ve}{\ve^2} + \go(\ve^{-2}).
\end{align*}
We still have
\begin{equation*}
\p_\mu X = \p_\mu^2 \Theta = E^{\T}\n^2\Theta E = -\y_s\ve^{-2}-a_s\y_s\ve^{-1} + \go(\ln\ve),
\end{equation*}
so $\p_\mu X<0$ for $\ve$ small enough.

For the other properties, since $W_s\cdot\sm W_s = b \neq 0$ and $\Ps$ is invertible, we have that, if $\dd_c^2\mo_\ell + \dd_c^2\mo_r \neq 0$ then $\n^2\Theta$ is invertible. In this case, Sylvester's rule gives that, for $\ve$ small enough, $n(\n^2\Theta)$ equals the number of sign changes in
\begin{equation*}
    \left(+,~-,~-\sign(\dd_c^2\mo_\ell + \dd_c^2\mo_r),~\sign(\dd_c^2\mo_\ell + \dd_c^2\mo_r)\right).
\end{equation*}
So if $\dd_c^2\mo_\ell + \dd_c^2\mo_r > 0$ we have $n(\n^2\Theta) = 2 = N +1$ and we can conclude to the spectral instability of the wave using Theorem~\ref{theorem_stability_instability_periodic_wave}. In the other case we get that $n(\n^2\Theta) = 3 = N +2$, so $\n^2\Theta$ is negative-definite and we can not conclude to the stability or instability of the wave.
\end{proof}

\begin{remarque}
    In the proof, we can compare the matrix $\Ps^{\T}\n^2\Theta\Ps$ obtained in our limit with the one obtained in the large period limit in Corollary 2 of \cite{benzoni-gavageStabilityPeriodicWaves2020}. Without the first coefficient, which corresponds to $\p_\mu X$, we see that the matrix obtained here is simply the sum of the matrices that would be obtained by the long wavelength limit for the left-hand soliton and the right-hand soliton.
\end{remarque}

\begin{remarque}\label{rmq_stabilité_modulationnelle}
    When $\dd_c^2\mo_\ell + \dd_c^2\mo_r <0$, $\n^2 \Theta$ is negative definite. From Proposition~2 of \cite{benzoni-gavageStabilityPeriodicWaves2014}, this is known to imply that the associated modulation system is hyperbolic. In turn, this rules out modulational instability as encoded by \cite[Theorem~1]{benzoni-gavageStabilityPeriodicWaves2014}.
\end{remarque}

\section{Asymptotic expansion of the Evans function}\label{section_evans_function}

To conclude when $\dd_c^2\mo_\ell + \dd_c^2\mo_r <0$, we need to study the Evans functions for all Floquet exponents of a two-pulse periodic wave and link it to the Evans functions of each homoclinic orbit. We extend the approach of \cite{yangConvergencePeriodGoes2019}. The difference is that, in \cite{yangConvergencePeriodGoes2019}, the periodic waves converges to one soliton. Here each part of the periodic waves converges to a different soliton. So first, we need to be more precise on the behavior at infinity of the soliton and on the "convergence" of two-pulse periodic waves toward the homoclinic orbits.

\subsection{Behavior of the periodic waves for a large period}

Until now, traveling waves were defined up to a translation. Here, we make choices and fix the value at zero of the traveling waves that we consider. Recalling Figure~\ref{fig:portrait_de_phase} and Assumption~\ref{hypothèse_potentiel}, we can fix the following condition for each soliton. First, for the right-hand soliton, $U_r^0$, we choose that
\begin{equation*}
    U_r^0(0) = \left( \begin{matrix} v_r^s \\ g(v_r^s) \end{matrix} \right),~\text{if}~N=2, \quad U_r^0(0) = v_r^s,~\text{if}~N=1.
\end{equation*} 
In the same spirit, for the left-hand soliton we choose 
\begin{equation*}
    U_\ell^0(0) = \left( \begin{matrix} v_\ell^s \\ g(v_\ell^s) \end{matrix} \right),~\text{if}~N=2, \quad U_\ell^0(0) = v_\ell^s,~\text{if}~N=1.
\end{equation*}
By assumption the two solitons have the same endstate
\begin{equation*}
    U_s = \left( \begin{matrix} v_s \\ g(v_s) \end{matrix} \right),~\text{if}~N=2, \quad U_s = v_s,~\text{if}~N=1.
\end{equation*}
As for periodic waves, we choose to place a bright pulse at the right of the origin and a dark one at is left. Moreover for $\ve>0$, we fix $\un{U}^\ve$ the profile of the two-pulse periodic wave parametrized by $\mug = \left(\ve^2 + \mu_s, \lm, c \right) \in \Omega $ with period $X^\ve$, such that,
\begin{equation*}
     \un{U}^\ve(0) = \left( \begin{matrix} v_s \\ g(v_s) \end{matrix} \right),~\text{if}~N=2, \quad  \un{U}^\ve(0) = v_s,~\text{if}~N=1,~\text{and}~v_x^\ve(0) > 0.
\end{equation*}
We need to define the position, on a period, where the periodic wave reaches its maximum and its minimum, as in Figure~\ref{fig:graphe_onde_periodique_Evans}. 

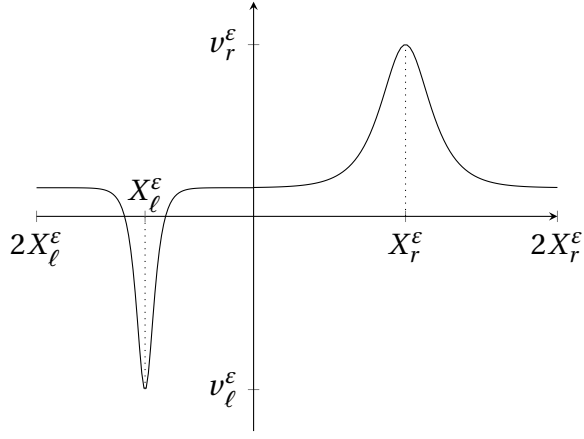
\begin{figure}[H]
\centering
\begin{tikzpicture}
\begin{axis}[
    axis x line=bottom,
    axis y line = left,
    axis lines=middle,
    xtick = {-10,7,14},
    xticklabels ={$2 X^\ve_\ell$,$X^\ve_r$,$2 X^\ve_r$},
    ytick= {-1.21,1.2},
    yticklabels = {$v^\ve_\ell$,$v_r^\ve$},
    extra x ticks={-5},
    extra x tick labels={$X^\ve_\ell$},
    extra x tick style={          
        tick label style={anchor=south} 
    },
    ymin = -1.5,
    ymax = 1.5,
    scale = 1
]
\addplot[
    domain=0:14, 
    samples=100, 
    color=black
]
{2*exp(x-7)/(1 + exp(2*(x-7)))+ 0.2}; 

\addplot[
    domain=-10:0, 
    samples=100, 
    color=black
]
{-2*sqrt(2)*exp(2*(1.4*x+7))/(1 + exp(4*(1.4*x+7)))+ 0.2};

\addplot[dotted, samples=50, smooth,domain=0:6 ] coordinates {(7,0)(7, 1.2)};
\addplot[dotted, samples=50, smooth,domain=0:6 ] coordinates {(-5,0)(-5 , -1.21)};

\end{axis}
\end{tikzpicture}
\caption{"Generic" periodic wave that we are looking at} 
\label{fig:graphe_onde_periodique_Evans} 
\end{figure}

Due to Assumption~\ref{hypothèse_potentiel}, as we can see in the phase portrait, Figure~\ref{fig:portrait_de_phase}, we know that $v_x^\ve$ vanishes only two times on each period. With the value that we have chosen for $x=0$, we denote by $X_r^\ve$ the first positive instant where $v_x^\ve$ vanishes and we denote by $X^\ve_\ell$ the larger $x < 0$ where $v_x^\ve$ vanishes. Using~\eqref{eq_onde_périodique_1} or~\eqref{eq_onde_périodique_2}, and integrating in the phase portrait, we derive
\begin{equation}\label{definition_X_r}
    X^\ve_r = \int_{v_s}^{v_r^\ve} \sqrt{\frac{\ka(v)}{2(\mu^\ve - \w(v;\lm,c))}} \dd v,
\end{equation}
and
\begin{equation}\label{definition_X_l}
    X^\ve_\ell = \int_{v_\ell^\ve}^{v_s} \sqrt{\frac{\ka(v)}{2(\mu^\ve - \w(v;\lm,c))}} \dd v.
\end{equation}
It also follows from both the symmetry of the phase portrait and~\eqref{period}, that
\begin{equation*}
    -X^\ve_\ell + X^\ve_r = \frac{X^\ve}{2}.
\end{equation*}
So we know that on $\ff{2X_\ell^\ve}{2X^\ve_r}$ we have a full period of the periodic wave $\un{U}^\ve$. As in Proposition~\ref{dl_action_période}, from Theorem~\ref{dl_int_phi}, we obtain asymptotic expansions of these quantities.

\begin{lemma}\label{propriétés_périodes} With Notation~\ref{note_dl}, the following asymptotic expansions hold when $\ve \to 0$.
\begin{align*}
    & X^\ve_\ell = \frac{1}{2}a_s\sqrt{2\ka(v_s)}\ln(\ve) + \go(1) \\
    & X^\ve_r = - \frac{1}{2}a_s\sqrt{2\ka(v_s)}\ln(\ve) + \go(1),
\end{align*}
and there exists $X^0 \in \R$ such that
\begin{equation}\label{cv_demi_période}
    X^\ve_r + X^\ve_\ell = X^0 + \go(\ve)
\end{equation}
\end{lemma}

The next goal is to use these quantities to translate the periodic wave and show that each pulse is well-approximated by a truncated soliton. But first, we show that each soliton converges exponentially fast, at $+\infty$ and $-\infty$, to the same constant. The rate of convergence will be useful afterward to define the Evans functions for the solitons and quantify the effect of truncation.

\begin{lemma}\label{limits_soliton}
    There exist some positive constant $C>0$ and $\nu>0$ such that for $i \in \left\{r,\ell\right\}$ we get
    \begin{equation*}
        \forall x \in \R,~\norm{U_i^0(x) - U_s} + \norm{\p_x U_i^0(x)}\leq C \e^{-\nu \abs{x}},
    \end{equation*}
    where $U_s = \left( \begin{matrix} v_s \\ g(v_s) \end{matrix} \right)$, if $N=2$ and $U_s = v_s$, if $N=1$.
\end{lemma}

This is a standard result that we prove as a warm-up for later proofs.

\begin{proof}
    Firstly, the proof is identical for the two solitons and for the limits in $+\infty$ and $-\infty$ so we will write down only the proof for the right-hand soliton and the $+\infty$ limit. In the case $N=2$, we know that the profile $U = \left(\begin{smallmatrix} u \\ v \end{smallmatrix} \right)$ of a traveling wave satisfies~\eqref{eq_onde_périodique_2}. Thus we only need to focus on the equation satisfied by $v$, which is the same as in the case $N=1$. Therefore we write~\eqref{eq_onde_périodique_1} as a first-order ODE with $\mathcal{U} = \left(\begin{smallmatrix} v \\ v_x \end{smallmatrix} \right)$,
    \begin{equation}\label{EDO_soliton_onde_périodique}
        \frac{\dd \mathcal{U}}{\dd x} = F\left(\mathcal{U}\right),
    \end{equation}
    where $F(\mathcal{U}) = \left(\begin{smallmatrix} v_x \\ \frac{-1}{\ka(v)}\left(\frac{1}{2}\p_v\ka(v)v_x^2 + \p_v\w(v;\lambda,c)\right) \end{smallmatrix} \right)$.\newline
    We linearize it at $\mathcal{U}_s := \left(\begin{smallmatrix} v_s \\ 0 \end{smallmatrix} \right)$, knowing that $\dd F(\mathcal{U}_s)$ is diagonalizable and its spectrum is $\left\{ \pm \sigma \right\}$ where $\sigma = \frac{\sqrt{2}}{a_s\sqrt{\ka(v_s)}} > 0 $. We denote by $\mathcal{P}_-$ and $\mathcal{P}_+$ the eigenprojections on the unstable and stable subspaces. For $\delta >0$, since~\eqref{EDO_soliton_onde_périodique} with an initial data admits a unique global solution, Duhamel's formula tells us that, up to a translation, for $x\in \fo{\delta}{+\infty}$,
    \begin{equation*}
        \mathcal{U}(x) = \mathcal{U}_s + \int_\delta^x \e^{-\sig(x-y)} \mathcal{P}_- G(\mathcal{U}(y)) \dd y + \int_{+\infty}^x \e^{\sig(x-y)} \mathcal{P}_+ G(\mathcal{U}(y)) \dd y,
    \end{equation*}
    where $ G(\mathcal{U}(y)) = F(\mathcal{U}(y)) - \dd F(\mathcal{U}_s)\left(\mathcal{U}(y) - \mathcal{U}_s \right)$ because $F(\mathcal{U}_s) = 0 $.\newline
    Then one shows that for $0 < \nu < \sigma$, for every $M>0$, there exists $\delta(M) \geq 0$ such that for $\delta \geq \delta(M) \geq 0$ the map
    \begin{equation*}
        \fonction{\mathcal{T}}{E_{M,\delta}}{E_{M,\delta}}{\mathcal{V}}{\mathcal{T}(\mathcal{V})}
    \end{equation*}
    where 
    \begin{equation*}
        \mathcal{T}(\mathcal{V})(x) = \mathcal{U}_s + \int_\delta^x \e^{-\sig(x-y)} \mathcal{P}_- G(\mathcal{V}(y)) \dd y + \int_{+\infty}^x \e^{\sig(x-y)} \mathcal{P}_+ G(\mathcal{V}(y)) \dd y,
    \end{equation*}
    and
    \begin{equation*}
        E_{M,\delta} = \left\{ \mathcal{V} \in \co(\fo{\delta}{+\infty});~\sup_{x \in \fo{\delta}{+\infty}} \e^{\nu x}\norm{\mathcal{V}(x) - \mathcal{U}_s} \leq M \right\},
    \end{equation*}
    is well defined and admits a unique fixed point. So we fix $M>0$, then $\delta(M) >0$ and we get that $\mathcal{U}$ verifies for $x \in \fo{\delta(M)}{+\infty}$,
    \begin{equation*}
        \norm{\mathcal{U}(x) - \mathcal{U}_s} \leq M \e^{- \nu x}.
    \end{equation*} 
    Since $\mathcal{U}$ is bounded on $\ff{0}{\delta(M)}$, there exists a constant $C>0$ such that for every $x \in \fo{0}{+\infty}$, 
    \begin{equation*}
        \norm{\mathcal{U}(x) - \mathcal{U}_s} \leq C \e^{- \nu x}.
    \end{equation*} 
    Noting that $\mathcal{U} = \left(\begin{smallmatrix} v \\ v_x \end{smallmatrix} \right)$ and $\mathcal{U}_s = \left(\begin{smallmatrix} v_s \\ 0 \end{smallmatrix} \right)$, the result follows.
\end{proof}

Now we prove the claimed convergence.

\begin{lemma}\label{cv_onde_soliton}
There are positive constants $C>0$ and $\theta >0$ such that
\begin{equation*}
    \forall y \in \ff{X^{\ve}_l}{-X^{\ve}_l}, ~\norm{U^{\ve}(y + X^{\ve}_\ell) - U^{0}_\ell(y)} + \norm{\p_x U^{\ve}(y + X^{\ve}_\ell) - \p_x U^{0}_{\ell}(y)} \leq C\e^{\theta X^\ve_\ell},
\end{equation*}
and
\begin{equation*}
    \forall y \in \ff{- X^{\ve}_r}{X^{\ve}_r}, ~\norm{U^{\ve}(y + X^{\ve}_r) - U^{0}_r(y)} + \norm{\p_x U^{\ve}(y + X^{\ve}_r) - \p_x U^{0}_{r}(y)} \leq C\e^{-\theta  X^\ve_r}.
\end{equation*}
\end{lemma}

\begin{proof}
    As in the previous lemma, we only have to focus on the case $N=1$. We write down the proof for the right-hand soliton, and for $y \in \ff{0}{X_r^\ve}$, proofs for the other cases being similar. Firstly, we translate $v^\ve$ so that its value in 0 is close to the value in 0 of $v^0$, for this, we introduced $\tilde{v}^\ve = v(\cdot + X_r^\ve)$, so that $\tilde{v}^\ve(0) = v_r^\ve$ which is close to $v^0(0) = v_r^s$. We set $\mathcal{U}^\ve = \left(\begin{smallmatrix} \ti{v}^\ve \\ \ti{v}_x^\ve \end{smallmatrix} \right)$ and $\mathcal{U}^0 = \left(\begin{smallmatrix} v^0 \\ v_x^0 \end{smallmatrix} \right)$. By~\eqref{EDO_soliton_onde_périodique} 
    we have a first-order ODE 
    \begin{equation*}
        \frac{\dd}{\dd x}\left(\mathcal{U}^\ve - \mathcal{U}^0 \right) = F\left(\mathcal{U}^\ve\right) - F\left(\mathcal{U}^0\right),
    \end{equation*}
    Moreover by definition $\mathcal{U}^\ve(0) = \left(\begin{smallmatrix} v_r^\ve \\ 0 \end{smallmatrix} \right)$ 
    and $\mathcal{U}^\ve(X_r^\ve) =  \left(\begin{smallmatrix} v_s \\ v_x^\ve(X_r^\ve) \end{smallmatrix} \right) $. So using Proposition~\ref{dl_racines}, we get that there exists a constant $K_1 > 0$ such that 
    \begin{equation}
        \norm{\mathcal{U}^\ve(0) - \mathcal{U}^0(0)} \leq K_1 \ve^2. 
    \end{equation}
    By using that $v^\ve$ verifies~\eqref{eq_onde_périodique_1} and the previous lemma, one gets that there exists a constant $C>0$ such that
    \begin{equation*}
        \norm{\mathcal{U}^\ve(X_r^\ve) - \mathcal{U}^0(X_r^\ve)} \leq \norm{\mathcal{U}^\ve(X_r^\ve) - \mathcal{U}_s} + \norm{\mathcal{U}_s - \mathcal{U}^0(X_r^\ve)} \leq C\ve + C \e^{-\nu X_r^\ve}.
    \end{equation*}
    With the Grönwall lemma, we may show that there exist constants $K_1,~K_2 > 0$ such that for every $\delta \in \ff{0}{X_r^\ve}$, for every $x \in  \ff{0}{\delta}$,
    \begin{equation*}
        \norm{\mathcal{U}^\ve(x) - \mathcal{U}^0(x)} \leq K_1 \e^{K_2 \delta} \ve^2.
    \end{equation*}
    Finally as in the previous lemma, we linearize~\eqref{EDO_soliton_onde_périodique} around $\mathcal{U}_s$, and apply Duhamel's formula to get, for $x \in \ff{\delta}{X_r^\ve}$, 
    \begin{align*}
        \mathcal{U}^\ve(x) - \mathcal{U}^0(x)  =& \e^{-\sig(x-\delta)}\mathcal{P}_-\left(\mathcal{U}^\ve(\delta) - \mathcal{U}^0(\delta)\right) + \e^{\sig(x-X^\ve_r)}\mathcal{P}_+\left(\mathcal{U}^\ve(X_r^\ve) - \mathcal{U}^0(X_r^\ve)\right)\\ &+ 
                                                \int_\delta^x \e^{-\sig(x-y)} \mathcal{P}_- \mathcal{G}(\mathcal{U}^\ve(y),\mathcal{U}^0(y)) \dd y + \int_{X^\ve_r}^x \e^{\sig(x-y)} \mathcal{P}_+ \mathcal{G}(\mathcal{U}^\ve(y),\mathcal{U}^0(y)) \dd y.
    \end{align*}
    where $\mathcal{G}(\mathcal{U}^\ve(y),\mathcal{U}^0(y)) = F\left(\mathcal{U}^\ve(y)\right) - F\left(\mathcal{U}^0(y)\right) - \dd F(\mathcal{U}_s)\left(\mathcal{U}^\ve(y) - \mathcal{U}^0(y) \right)$. 
    In order to apply Banach's fixed point theorem, we first fix $\theta$. Now we observe that there exist $\delta \in \fo{0}{+\infty}$ large enough and $\ve_0 > 0$ such that for every $ \ve_0 \geq \ve > 0$, $X_r^\ve > \delta$ and the map
    \begin{equation*}
        \fonction{\mathcal{T}}{E_{\delta}}{E_{\delta}}{\mathcal{V}}{\mathcal{T}(\mathcal{V})},
    \end{equation*}
    where 
    \begin{align*}
        \mathcal{T}(\mathcal{V})(x) =& \mathcal{U}^0(x) + \e^{-\sig(x-\delta)}\mathcal{P}_-\left(\mathcal{U}^\ve(\delta) - \mathcal{U}^0(\delta)\right) + \e^{\sig(x-X^\ve_r)}\mathcal{P}_+\left(\mathcal{U}^\ve(X_r^\ve) - \mathcal{U}^0(X_r^\ve)\right)\\ &+ 
                                     \int_\delta^x \e^{-\sig(x-y)} \mathcal{P}_- \mathcal{G}(\mathcal{V}(y),\mathcal{U}^0(y)) \dd y + \int_{X^\ve_r}^x \e^{\sig(x-y)} \mathcal{P}_+ \mathcal{G}(\mathcal{V}(y),\mathcal{U}^0(y)) \dd y,
    \end{align*}
    and 
    \begin{equation*}
        E_{\delta} = \left\{ \mathcal{V} \in \co(\ff{\delta}{X_r^\ve});~\sup_{x \in \fo{\delta}{X^\ve_r}} \e^{\theta X^\ve_r}\norm{\mathcal{V}(x) - \mathcal{U}^0(x)} \leq 2C\right\},
    \end{equation*}
    is well defined and admits a unique fixed point. To prove this claim we used that Lemma~\ref{limits_soliton} gives that there exists a constant $K_3 > 0$ such that for all $y\in\ff{\delta}{X^\ve_r}$,
    \begin{equation*}
        \norm{\mathcal{G}(\mathcal{V}(y),\mathcal{U}^0(y))} \leq K_3 \norm{\mathcal{V}(y) - \mathcal{U}^0(y)}^2 + K_3 \e^{- \nu y} \norm{\mathcal{V}(y) - \mathcal{U}^0(y)},
    \end{equation*}
    and
    \begin{equation*}
        \norm{\mathcal{G}(\mathcal{V}(y),\mathcal{U}^0(y)) - \mathcal{G}(\mathcal{X}(y),\mathcal{U}^0(y))} \leq 2 K_3 \norm{\mathcal{V}(y) - \mathcal{X}(y)}^2 + K_3 \left(\e^{-\theta X^\ve_r} + \e^{-\nu y} \right)\norm{\mathcal{V}(y) - \mathcal{X}(y)}.
    \end{equation*}
    As a result, we have shown that there exists $\delta \in \fo{0}{+\infty}$ large enough such that for every $\ve > 0$ sufficiently small, we have for $x \in \ff{0}{\delta}$,
    \begin{equation*}
        \norm{\mathcal{U}^\ve(x) - \mathcal{U}^0(x)} \leq K_1 \e^{K_2 \delta} \ve^2,
    \end{equation*}
    and for every $x \in \ff{\delta}{X_r^\ve}$, 
    \begin{equation*}
        \norm{\mathcal{U}^\ve(x) - \mathcal{U}^0(x)} \leq 2 C \e^{-\theta X_r^\ve}.
    \end{equation*}
    Since $X^\ve_r = - \frac{1}{2}a_s\sqrt{2\ka(v_s)}\ln(\ve) + \go(1)$, we proved the claimed bound,
    \begin{equation*}
        \forall x \in \ff{0}{X^\ve_r},~\norm{\mathcal{U}^\ve(x) - \mathcal{U}^0(x)} \leq C \e^{-\theta X^\ve_r}.
    \end{equation*}
\end{proof}

\subsection{Spectral problems and reduction to a constant coefficient system}

We want to find the spectrum of the linearized equation in the mobile frame around the periodic traveling wave $ \un{U}^\ve$. For a traveling wave $\un{U}$ with speed $c$, the linearized equation in the mobile frame is 
\begin{equation*}
    \p_t V = B\p_x\left(\A(\un{U}) V\right) 
\end{equation*}
where
\begin{equation*}
    \A(\un{U}) = \hess\left(\Hm+cQ\right)[\un{U}].
\end{equation*}
In the case where $N=1$, we have
\begin{equation*}
    \A(\un{U}) =  -\p_x(\ka(\un{v})\p_x\cdot) + f''(\un{v}) + \frac{c}{b} + \frac{1}{2}\ka''(\un{v})\un{v}_x^2 - \p_x(\ka'(\un{v})\un{v}_x)
\end{equation*}
and if $N=2$,
\begin{equation*}
    \A(\un{U}) = \begin{pNiceArray}{c|c}
      -\p_x(\ka(\un{v})\p_x\cdot) + f''(\un{v}) + \frac{1}{2}\ka''(\un{v})\un{v}_x^2 - \p_x(\ka'(\un{v})\un{v}_x) + \frac{1}{2}\tau''(\un{v})\un{u}^2 & 
      \tau'(\un{v})\un{u}  + \frac{c}{b}\\
      \hline
      \tau'(\un{v})\un{u} + \frac{c}{b} & \tau(\un{v}) 
   \end{pNiceArray}
\end{equation*}
In what follows, we are interested in the following eigenvalue problems, for $z \in \C$, 
\begin{equation}\label{spectral_problems}
    z V = \Al^\ve V, \quad z V = \Al_\ell^0 V, \quad z V = \Al_r^0  V, \quad z V = \Al_\infty^0 V,
\end{equation}
where $\Al^\ve = B\p_x \A( \un{U}^\ve)$, $\Al_\ell^0 := B\p_x \A(U_\ell^0)$, $\Al_r^0 = B\p_x \A(U_r^0)$ and $\Al_\infty^0 = B\p_x \A(U_s)$. They are differential operators so, for each one, if $z \in \C$ is fixed, the spectral problem~\eqref{spectral_problems}, can be rewritten as a first-order ODE. For the right-hand soliton, we denote it by
\begin{equation}\label{eq_ordre_1_spectre_soliton_droite}
    W' = A^0_r(\cdot,z)W,
\end{equation}
for the left-hand soliton by
\begin{equation}\label{eq_ordre_1_spectre_soliton_gauche}
    W' = A^0_\ell(\cdot,z)W,
\end{equation}
and for the constant $U_s$ by
\begin{equation}\label{eq_ordre_1_spectre_constante}
    W' = A^0_\infty(z)W.
\end{equation}
It follows from Lemma~\ref{limits_soliton}, that the coefficients of differential operators $\Al_\ell^0$ and $\Al_r^0$ converge exponentially fast towards constant-coefficient operators

\begin{lemma}\label{cv_pb_spectral_soliton}
For $M>0$ there exist $C(M)>0$ and $\nu>0$ such that for $i\in \left\{r,\ell \right\}$ and $\abs{z} \leq M$,
    \begin{equation*}
        \forall x \in \R,~\norm{A_i^0(x,z) - A^0_\infty(z)} \leq C(M)\e^{-\nu\abs{x}}.
    \end{equation*}
\end{lemma}
In the same way we write the eigenvalue problem for the periodic wave as a first-order ODE
\begin{equation}\label{eq_ordre_1_spectre_onde_périodique}
    W' = A^\ve(\cdot,z)W.
\end{equation}
The convergence result given by Lemma~\ref{cv_onde_soliton} yields the following bounds. 

\begin{lemma}\label{cv_pb_spectral_periodique}
For $M>0$ there exists $C(M) > 0$ such that for $\abs{z} \leq M$, 
\begin{equation*}
    \forall y \in \ff{X^{\ve}_l}{-X^{\ve}_l}, ~\norm{A^{\ve}(y + X^{\ve}_r,z) - A^{0}_\ell(y,z)}  \leq C(M)\e^{\theta X^{\ve}_\ell},
\end{equation*}
and
\begin{equation*}
    \forall y \in \ff{- X^{\ve}_l}{X^{\ve}_r}, ~\norm{A^{\ve}(y + X^{\ve}_r,z) - A^{0}_r(y,z)} \leq C(M)\e^{-\theta X^{\ve}_r}.
\end{equation*}
\end{lemma}
 
The following lemmas correspond to Lemma~3.1 of \cite{yangConvergencePeriodGoes2019} but is adapted to the fact that here, each pulse of the periodic wave converge to a different soliton.

\begin{lemma}\label{coefficient_constant_droite}
    For all $\ve\geq0$, there exist in a neighborhood of any $\abs{z_0} \leq M$ some bounded and uniformly invertible transformations $P^\ve_{+,r}(y,z) $ and $P^\ve_{-,r}(y,z)$ defined on $y\leq 0$ and respectively $y\geq 0$, analytic in $z$ as functions into $L^{\infty}\left(\fo{0}{\pm \infty} \right)$ such that for any $0<\ti{\nu}<\min(\theta,\nu)$ and for some $C_r>0$
    \begin{equation*}
        P^\ve_{+,r}(X^\ve_r,z) = \id, \quad P^\ve_{-,r}(-X^\ve_r,z) = \id,
    \end{equation*}
    \begin{equation}\label{cv_matrice_P_r}
        \norm{P^\ve_{\pm,r} - P^0_{\pm,r}} \leq C_r \e^{-\ti{\nu} X^{\ve}_r},~\text{for}~y \gtrless 0 
    \end{equation}
    and the change of coordinates $W = P^\ve_{\pm,r}Z$ reduces 
    \begin{equation*}
        W' = A^{\ve}(\cdot + X^{\ve}_r,z) W,
    \end{equation*}
    to the constant-coefficient system
    \begin{equation*}
        Z' = A^0_\infty (z)Z,
    \end{equation*}
    for $y \gtrless 0$ and $y \in \ff{-X^{\ve}_r}{X^{\ve}_r}$.
\end{lemma}

\begin{proof}
With Lemmas~\ref{cv_pb_spectral_soliton} and~\ref{cv_pb_spectral_periodique} in hand, the proof is exactly the same as the one of Lemma~3.1 of \cite{yangConvergencePeriodGoes2019} and is therefore omitted.
\end{proof}

We have the same for the left-hand soliton.

\begin{lemma}\label{coefficient_constant_gauche}
    For all $\ve\geq0$, there exist in a neighborhood of any $\abs{z_0} \leq M$ some bounded and uniformly invertible transformations $P^\ve_{+,\ell}(y,z) $ and $P^\ve_{-,\ell}(y,z)$ defined on $y\leq 0$ and respectively $y\geq 0$, analytic in $z$ as functions into $L^{\infty}\left(\fo{0}{\pm \infty} \right)$ such that for any $0<\tilde{\nu}<\min(\theta,\nu)$ and for some $C_\ell >0$
    \begin{equation*}
        P^\ve_{+,\ell}(X^\ve_\ell,z) = \id, \quad P^\ve_{-,r}(-X^\ve_\ell,z) = \id,
    \end{equation*}
    \begin{equation}\label{cv_matrice_P_l}
        \norm{P^\ve_{\pm,\ell} - P^0_{\pm,\ell}} \leq C_\ell \e^{\ti{\nu} X^{\ve}_\ell},~\text{for}~y \gtrless 0 
    \end{equation}
    and the change of coordinates $W = P^\ve_{\pm,\ell}Z$ reduces 
    \begin{equation*}
        W' = A^{\ve}(\cdot + X^{\ve}_\ell,z) W,
    \end{equation*}
    to the constant-coefficient system
    \begin{equation*}
        Z' = A^0_\infty (z)Z,
    \end{equation*}
    for $y \gtrless 0$ and $y \in \ff{X^{\ve}_\ell}{-X^{\ve}_\ell}$.
\end{lemma}

\subsection{The Evans functions for periodic and homoclinic orbits}

To introduce the Evans functions for the periodic wave $\un{U}^\ve$, we denote by $\Psi^\ve(\cdot,z)$ a solution of~\eqref{eq_ordre_1_spectre_onde_périodique} with an invertible initial data $\Psi^\ve(0,z)$. Using Floquet-Bloch theory, we know from \cite{gardnerStructureSpectraPeriodic1993} that for the periodic wave, the spectrum for localized perturbations is given by the union of $\gamma$-eigenvalues, where for $\gamma \in \C$ with $\abs{\gamma} = 1$, a $\gamma$-eigenvalue $z$ is a zero of the periodic Evans function
\begin{equation}
    E^{\ve}(z,\gamma) = \det\left(\Psi^{\ve}(X^{\ve},z)\Psi^{\ve}(0,z)^{-1} - \gamma \id \right).
\end{equation}
When restricting to co-periodic perturbations, the spectrum is the union of 1-eigenvalues.

For homoclinic orbits, we define their Evans functions as in \cite{yangConvergencePeriodGoes2019,gardnerGapLemmaGeometric1998,barkerMetastabilitySolitaryRoll2011,benzoni-gavageSpectralTransverseInstability2010}. Due to Lemma~\ref{limits_soliton}, we know that the essential spectra of $\Al^0_r$ and $\Al^0_\ell$ is exactly the spectrum of the constant-coefficients problems, which is the spectrum of $\Al^0_\infty$ for the two solitons. One can easily compute this spectrum using the Fourier transform. It has been shown in Appendix~A of \cite{benzoni-gavageStabilityPeriodicWaves2020} that in all cases $\sigma\left(\Al^0_\infty\right) \subset i\R$. Note that this computation relies on the fact that the constant $U_s$ must verify some sign condition to be the endpoint of a soliton. 

Thus we can define an Evans function for homoclinic orbits on $\C_+ = \left\{ z \in \C;~\Re(z) > 0 \right\} $. For $z \in \C_+$, $A^0_\infty(z)$ is hyperbolic, so we know, \cite[Lemma~I.4.10]{katoPerturbationTheoryLinear1995}, that the dimension of its stable and unstable subspace does not depend on $z$. So we denote by $n_s$ the number of negative real part eigenvalues of $A^0_\infty(z)$ which is the dimension of the stable subspace. We also denote by $R^-_\infty(z)$ a matrix whose columns are a basis for the unstable subspace of $A^0_\infty(z)$ and likewise $R^+_\infty(z)$ for the stable subspace of $A^0_\infty(z)$. Furthermore, the stable and unstable subspaces of $A^0_\infty(z)$ are in direct sum, so the matrix $\left(R^-_\infty(z),~R^+_\infty(z)\right)$ is invertible, we denote by $\left(\begin{smallmatrix} L^-_\infty(z) \\ L^+_\infty(z) \end{smallmatrix}\right)$ its inverse. In the foregoing, using Kato's perturbation method \cite[pp.~99-100]{katoPerturbationTheoryLinear1995}, everything can be chosen to be analytic in $z$, because $A^0_\infty$ is analytic in $z$. From now on we will omit the dependence on $z$ for the sake of readability.

To define the Evans function for solitons, we need bases of solutions to~\eqref{eq_ordre_1_spectre_soliton_droite} decaying as $x\rightarrow+\infty$ and $x\rightarrow-\infty$ and the same for~\eqref{eq_ordre_1_spectre_soliton_gauche}. For this, we may use the following matrices 
\begin{equation*}
    R^-_r(x) := P^0_{-,r}(x)\e^{A^0_\infty x}R^{-}_\infty,\quad R^+_r(x) := P^0_{+,r}(x)\e^{A^0_\infty x}R^{+}_\infty, 
\end{equation*}
for the right-hand soliton and
\begin{equation*}
    R^-_\ell(x) := P^0_{-,\ell}(x)\e^{A^0_\infty x}R^{-}_\infty,\quad R^+_\ell(x) := P^0_{+,\ell}(x)\e^{A^0_\infty x}R^{+}_\infty, 
\end{equation*}
for the left-hand soliton, where matrices $P^0_{-,\ell},~ P^0_{+,\ell},~P^0_{-,r},~P^0_{+,r}$ are given by Lemmas~\ref{coefficient_constant_droite} and~\ref{coefficient_constant_gauche}. We can now define the Evans functions for the right-hand soliton and for the left-hand soliton, as it is done, for instance, in \cite{yangConvergencePeriodGoes2019}.

\begin{definition}[\cite{gardnerGapLemmaGeometric1998}]
    On $\C_+$, the Evans function for the right-hand soliton is defined as 
    \begin{equation}\label{evans_soliton_droite}
        D^0_r(z) = \det\begin{pmatrix} L^-_\infty \\ L^+_\infty \end{pmatrix} \det(R^-_r,R^+_r)|_{x=0} = \det\begin{pmatrix} L^-_\infty \\ L^+_\infty \end{pmatrix}
         \det\Big(P^0_{-,r}(0)R^-_\infty,P^0_{+,r}(0)R^+_\infty\Big),
    \end{equation}
    and symmetrically for the left-hand soliton
    \begin{equation}\label{evans_soliton_gauche}
        D^0_\ell(z) = \det\begin{pmatrix} L^-_\infty \\ L^+_\infty \end{pmatrix} \det(R^-_\ell,R^+_\ell)|_{x=0} = \det\begin{pmatrix} L^-_\infty \\ L^+_\infty \end{pmatrix}
         \det\Big(P^0_{-,\ell}(0)R^-_\infty,P^0_{+,\ell}(0)R^+_\infty\Big).
    \end{equation}
\end{definition}

It is well known, \cite{gardnerGapLemmaGeometric1998,masciaPointwiseGreenFunction2003} and references therein, that on $\C_+$, for $i \in \left\{r, \ell\right\}$, $D^0_i$, is analytic and vanishes at $z_*$ if and only if $z_*$ is an eigenvalue of $\Al(U_i)$, furthermore the multiplicity coincide.

\subsection{Convergence of the periodic Evans function}

We are investigating the limit of the periodic Evans function on each compact of $\C_+$. This requires a suitable normalization, as in \cite{yangConvergencePeriodGoes2019,gardnerSpectralAnalysisLong1997}. By Abel's formula, 
\begin{equation*}
    E^{\ve}(z,\gamma) = \ti{E}^{\ve}(z,\gamma) \e^{\int_0^{2 X^{\ve}_r}\tr A^{\ve}(x,z)\dd x},
\end{equation*}
where
\begin{equation*}
    \ti{E}^{\ve}(z,\gamma) := \det\left(\Psi^{\ve}(0,z)\Psi^{\ve}(2 X^\ve_\ell,z)^{-1} - \gamma 
                            \Psi^{\ve}(0,z)\Psi^{\ve}(2 X^\ve_r,z)^{-1} \right).
\end{equation*}
And as in \cite{yangConvergencePeriodGoes2019} we define the rescaled balanced periodic Evans function as 
\begin{equation}\label{fonction_Evans_renormalisé}
    \ti{D}^{\ve}(z,\gamma) = \e^{\tr A_\infty^0 \Pi_s 2 X^\ve_r} \e^{- \tr A_\infty^0 \Pi_u 2 X^\ve_\ell}(-\gamma)^{-n_s}\ti{E}^{\ve}(z,\gamma),
\end{equation}
where $\Pi_u$ and $\Pi_s$ denote the unstable and stable eigenprojections associated with $A^0_\infty(z)$. We now introduce the determinant that appears at the limit.

\begin{definition}\label{objet_limite_fonction_evans}
    On $\C_+$ we define
    \begin{equation*}
        \ti{D}^0(z) = \det\Big(L^-_\infty P^0_{+,\ell}(0)^{-1} P^0_{-,\ell}(0) R^-_\infty\Big)\det\Big(L^+_\infty P^0_{-,r}(0)^{-1} P^0_{+,r}(0) R^+_\infty\Big).
    \end{equation*}
\end{definition}

The goal is to link the zeros of $D^0$ with the zeros of Evans functions for homoclinic orbits. It is done in the following proposition.

\begin{proposition}\label{zero_D}
    For $z_* \in \C_+$, $z_*$ is a zero of $\tilde{D}^0$ if and only if $z_*$ is an eigenvalue of $\A^0_r$ or $\A^0_\ell$. Moreover, the multiplicity of $z_*$ as a root of $\tilde{D}^0$ equals the sum of its multiplicities as an eigenvalue of $\A^0_r$ and $\A^0_\ell$.
\end{proposition}

\begin{proof}
    For the right-hand soliton, we have
    \begin{align*}
        \Big(P^0_{-,r}(0)R^-_\infty,P^0_{+,r}(0)R^+_\infty \Big) = P^0_{-,r}(0) \left(R^-_\infty,~R^+_\infty\right) \begin{pmatrix}\id & 
            L^-_\infty P^0_{-,r}(0)^{-1} P^0_{+,r}(0) R^+_\infty \\ 0 & L^+_\infty P^0_{-,r}(0)^{-1} P^0_{+,r}(0) R^+_\infty\end{pmatrix},
    \end{align*}
    so
    \begin{equation*}
        D^0_r(z) = \det P^0_{-,r}(0) \det\Big(L^+_\infty P^0_{-,r}(0)^{-1} P^0_{+,r}(0) R^+_\infty\Big).
    \end{equation*}
    It is the same for the left-hand soliton,
    \begin{equation*}
        \Big(P^0_{-,\ell}(0)R^-_\infty,P^0_{+,\ell}(0)R^+_\infty\Big) = P^0_{+,\ell}(0) \left(R^-_\infty,~R^+_\infty\right)  \begin{pmatrix}
        L^-_\infty P^0_{+,\ell}(0)^{-1} P^0_{-,\ell}(0) R^-_\infty & 0 \\ L^+_\infty P^0_{+,\ell}(0)^{-1} P^0_{-,\ell}(0) R^-_\infty & \id \end{pmatrix},
    \end{equation*}
    and 
    \begin{equation*}
        D^0_\ell(z) = \det P^0_{+,\ell}(0) \det\Big(L^-_\infty P^0_{+,\ell}(0)^{-1} P^0_{-,\ell}(0) R^-_\infty \Big).
    \end{equation*}
    Using that $P^0_{+,\ell}$ and $P^0_{-,r}$ are uniformly invertible, we can conclude the proof.
\end{proof}

The next step is to prove the convergence of the rescaled periodic Evans function. It is done in the following proposition.

\begin{proposition}\label{cv_evans_away_essential_spectra}
On each compact $K \subset \C_+$ there exists $C >0$ such that
\begin{equation*}
    \forall z \in K, \abs{\gamma} = 1, ~\abs{\ti{D}^{\ve}(z,\gamma) - D^0(z)} \leq C \e^{-\ti{\nu} \frac{X^\ve}{2}}.
\end{equation*}
for any $\ti{\nu} < \min\left(\theta,\nu,\sigma\right)$ where $\sigma$ is the spectral gap, defined as the minimum absolute value of the real parts of the eigenvalues of $A^0_{\infty}(z)$ for $z \in K$.
\end{proposition}

\begin{proof}
    First, it is easy to see that
    \begin{align*}
        \ti{E}^\ve(z,\gamma) &= \begin{aligned}[t] \det\Big(&\Psi^{\ve}(0,z)\Psi^{\ve}(X^\ve_\ell,z)^{-1}\Psi^{\ve}(X^\ve_\ell,z)\Psi^{\ve}(2 X^\ve_\ell,z)^{-1} - 
                                 \Psi^{\ve}(0,z)\Psi^{\ve}(X^\ve_\ell,z)^{-1}\Psi^{\ve}(X^\ve_\ell,z)\Psi^{\ve}(0,z)^{-1}\\ &+
                                 \Psi^{\ve}(0,z)\Psi^{\ve}(X^\ve_r,z)^{-1}\Psi^{\ve}(X^\ve_r,z)\Psi^{\ve}(0,z)^{-1} - 
                                 \gamma \Psi^{\ve}(0,z)\Psi^{\ve}(X^\ve_r,z)^{-1}\Psi^{\ve}(X^\ve_r,z)\Psi^{\ve}(2 X^\ve_r,z)^{-1} \Big)\end{aligned}\\
                              &:= \det\Big(\ti{F}^\ve(z,\gamma) \Big).
    \end{align*}
    But, due to Lemmas~\ref{coefficient_constant_droite} and~\ref{coefficient_constant_gauche}, we may express $\Psi^{\ve}$ with the constant-coefficient problem, 
    \begin{align*}
        \Psi^{\ve}(0,z)\Psi^{\ve}(X^\ve_\ell,z)^{-1} =& \e^{-A^0_\infty X^\ve_\ell}P^\ve_{+,\ell}(0)^{-1},\\
        \Psi^{\ve}(0,z)\Psi^{\ve}(X^\ve_r,z)^{-1} =& \e^{-A^0_\infty X^\ve_r}P^\ve_{-,r}(0)^{-1},\\
        \Psi^{\ve}(X^\ve_\ell,z)\Psi^{\ve}(2 X^\ve_\ell,z)^{-1} =& P^\ve_{-,\ell}(0) \e^{-A^0_\infty X^\ve_\ell},\\
        \Psi^{\ve}(X^\ve_r,z)\Psi^{\ve}(2 X^\ve_r,z)^{-1} =& P^\ve_{+,r}(0) \e^{-A^0_\infty X^\ve_r}.
    \end{align*}
    Thus
    \begin{equation*}
        \ti{F}^\ve(z,\gamma) =\begin{aligned}[t] &\e^{-A^0_\infty X^\ve_\ell}P^\ve_{+,\ell}(0)^{-1}\left(P^\ve_{-,\ell}(0) \e^{-A^0_\infty X^\ve_\ell} - P^\ve_{+,\ell}(0)\e^{A^0_\infty X^\ve_\ell}\right)\\ 
                                        &+ \e^{-A^0_\infty X^\ve_r}P^\ve_{-,r}(0)^{-1} \left(P^\ve_{-,r}(0)\e^{A^0_\infty X^\ve_r} - \gamma P^\ve_{+,r}(0) \e^{-A^0_\infty X^\ve_r} \right).
                                    \end{aligned}
    \end{equation*}
    Using the spectral expansion formula 
    \begin{equation}\label{espansion_spectra}
        \e^{A^0_\infty x } = \e^{A^0_\infty \Pi_u x }R^-_\infty L^-_\infty + \e^{A^0_\infty \Pi_s x }R^+_\infty L^+_\infty,
    \end{equation}
    and the fact that
    \begin{equation*}
        \norm{\e^{A^0_\infty \Pi_s x }} \leq C\e^{-\ti{\nu} x}~\text{for}~x>0\quad \text{and} \quad \norm{\e^{A^0_\infty \Pi_u x }} \leq C\e^{ \ti{\nu} x} ~ \text{for}~x<0,
    \end{equation*}
    we obtain that
    \begin{multline*}
        P^\ve_{-,\ell}(0) \e^{-A^0_\infty X^\ve_\ell} - P^\ve_{+,\ell}(0)\e^{A^0_\infty X^\ve_\ell} = P^\ve_{-,\ell}(0)\e^{- A^0_\infty \Pi_u X^\ve_\ell}R^-_\infty L^-_\infty - P^\ve_{+,\ell}(0)\e^{A^0_\infty \Pi_s X^\ve_\ell}R^+_\infty L^+_\infty  + \go(\e^{\ti{\nu} X^\ve_\ell}) \\
                         = \Big(P^\ve_{-,\ell}(0)R^-_\infty , P^\ve_{+,\ell}(0)R^+_\infty\Big)\Bigg[\begin{pmatrix} L^-_\infty \e^{- A^0_\infty \Pi_u X^\ve_\ell}R^-_\infty & 0 \\ 0 & - L^+_\infty \e^{A^0_\infty \Pi_s X^\ve_\ell}R^+_\infty \end{pmatrix}
                         + \go(\e^{\ti{\nu} X^\ve_\ell})\Bigg]\begin{pmatrix} L^-_\infty \\ L^+_\infty \end{pmatrix},
    \end{multline*}
    and similarly
    \begin{multline*}
        P^\ve_{-,r}(0)\e^{A^0_\infty X^\ve_r} - \gamma P^\ve_{+,r}(0) \e^{-A^0_\infty X^\ve_r} \\ = \Big(P^\ve_{-,r}(0)R^-_\infty , P^\ve_{+,r}(0)R^+_\infty\Big)\Bigg[\begin{pmatrix} L^-_\infty \e^{ A^0_\infty \Pi_u X^\ve_r}R^-_\infty & 0 \\ 0 & -\gamma L^+_\infty \e^{-A^0_\infty \Pi_s X^\ve_r}R^+_\infty \end{pmatrix}
                         + \go(\e^{-\ti{\nu} X^\ve_r})\Bigg]\begin{pmatrix} L^-_\infty \\ L^+_\infty \end{pmatrix}.
    \end{multline*}
    Using that $- X^\ve_\ell +  X^\ve_r = \frac{X^\ve}{2}$ and Abel's formula, we get 
    \begin{equation*}
        (-1)^{-n_s}\e^{\tr A_\infty^0 \Pi_s 2 X^\ve_r} \e^{ \tr A_\infty^0 \Pi_u 2 X^\ve_\ell} = \e^{\tr A_\infty^0 \Pi_s (2 X^\ve_r+X^\ve_\ell)} \e^{ \tr A_\infty^0 \Pi_u X^\ve_\ell } \det \begin{pmatrix} L^-_\infty \e^{ A^0_\infty \Pi_u X^\ve_\ell}R^-_\infty & 0 \\ 0 & - L^+_\infty \e^{-A^0_\infty \Pi_s X^\ve_\ell}R^+_\infty \end{pmatrix}.
    \end{equation*}
    Combining all together,
    \begin{multline*}
        \ti{D}^\ve(z,\gamma) =(\gamma)^{-n_s} \e^{\tr A_\infty^0 \Pi_s (2 X^\ve_r+X^\ve_\ell)} \e^{ \tr A_\infty^0 \Pi_u X^\ve_\ell } 
        \\ \begin{aligned}[t] \times \det \Bigg[& \e^{-A^0_\infty X^\ve_\ell}P^\ve_{+,\ell}(0)^{-1}\left(P^\ve_{-,\ell}(0)R^-_\infty , P^\ve_{+,\ell}(0)R^+_\infty\right)\left(\id + \go(\e^{2 \ti{\nu}  X^\ve_\ell})\right)
                    \\&+ \e^{-A^0_\infty X^\ve_r}P^\ve_{-,r}(0)^{-1} \left(P^\ve_{-,r}(0)R^-_\infty , P^\ve_{+,r}(0)R^+_\infty\right)
                                     \Big(\begin{pmatrix} L^-_\infty \e^{ A^0_\infty \Pi_u (X^\ve_r + X^\ve_\ell) }R^-_\infty & 0 \\ 0 & \gamma L^+_\infty \e^{-A^0_\infty \Pi_s (X^\ve_r + X^\ve_\ell)}R^+_\infty \end{pmatrix}
                                     \\ &+ \go\left(\e^{- \ti{\nu} \frac{X^\ve}{2}} \right) \Big)\Bigg]\det\begin{pmatrix} L^-_\infty \\ L^+_\infty \end{pmatrix}. \end{aligned} 
    \end{multline*}
    Due to Lemma~\ref{propriétés_périodes}, 
    \begin{align}
        \begin{pmatrix} L^-_\infty \e^{ A^0_\infty \Pi_u (X^\ve_r + X^\ve_\ell) }R^-_\infty & 0 \\ 0 & \gamma L^+_\infty \e^{-A^0_\infty \Pi_s (X^\ve_r + X^\ve_\ell)}R^+_\infty \end{pmatrix} &= \begin{pmatrix} L^-_\infty \e^{ A^0_\infty \Pi_u X^0 }R^-_\infty & 0 \\ 0 & \gamma L^+_\infty \e^{-A^0_\infty \Pi_s X^0 }R^+_\infty \end{pmatrix}  + \go\left(\e^{-\ti{\nu} \frac{X^\ve}{2}}\right) \nonumber
         \\ &:= B + \go\left(\e^{-\ti{\nu} \frac{X^\ve}{2}}\right), \label{def_B}
    \end{align}
    so we have
    \begin{multline*}
        \ti{D}^\ve(z,\gamma) =(\gamma)^{-n_s}  \e^{\tr A_\infty^0 \Pi_s (2 X^\ve_r+X^\ve_\ell)} \e^{ \tr A_\infty^0 \Pi_u X^\ve_\ell} \det\left(\id + \go(\e^{-\ti{\nu} \min\left(\frac{X^\ve}{2},\abs{2 X^\ve_\ell}\right)})\right) \det\begin{pmatrix} L^-_\infty \\ L^+_\infty \end{pmatrix} 
        \\ \times \det \Bigg[ \e^{-A^0_\infty X^\ve_\ell}\left( P^\ve_{+,\ell}(0)^{-1}P^\ve_{-,\ell}(0)R^-_\infty , R^+_\infty\right)
                             + \e^{-A^0_\infty X^\ve_r} \left(R^-_\infty , P^\ve_{-,r}(0)^{-1}P^\ve_{+,r}(0)R^+_\infty\right) B \Bigg]. 
    \end{multline*}
    By reusing the spectral expansion formula~\eqref{espansion_spectra},
    \begin{align*}
        \e^{-A^0_\infty X^\ve_\ell}\left( P^\ve_{+,\ell}(0)^{-1}P^\ve_{-,\ell}(0)R^-_\infty , R^+_\infty\right) =& \e^{-A^0_\infty \Pi_u X^\ve_\ell}R^{-}_\infty L^-_\infty\left( P^\ve_{+,\ell}(0)^{-1}P^\ve_{-,\ell}(0)R^-_\infty , R^+_\infty\right) + \go\left(\e^{\ti{\nu} X^\ve_\ell}\right)\\
                    =& \e^{-A^0_\infty \Pi_u X^\ve_\ell}R^{-}_\infty \left(L^-_\infty P^\ve_{+,\ell}(0)^{-1}P^\ve_{-,\ell}(0)R^-_\infty , 0\right) + \go\left(\e^{\ti{\nu} X^\ve_\ell}\right),
    \end{align*}
    and similarly
    \begin{equation*}
        \e^{-A^0_\infty X^\ve_r} \left(R^-_\infty , P^\ve_{-,r}(0)^{-1}P^\ve_{+,r}(0)R^+_\infty\right) B = \e^{-A^0_\infty \Pi_s X^\ve_r}R^+_\infty \left(0 , L^+_\infty P^\ve_{-,r}(0)^{-1}P^\ve_{+,r}(0)R^+_\infty\right) B + \go\left(\e^{- \ti{\nu} X^\ve_r}\right).
    \end{equation*}
    Thanks to~\eqref{def_B}, we obtain 
    \begin{multline*}
        \ti{D}^\ve(z,\gamma) = \gamma^{-n_s} \e^{\tr A_\infty^0 \Pi_s (2 X^\ve_r+X^\ve_\ell)} \e^{ \tr A_\infty^0 \Pi_u X^\ve_\ell }\det\left(R^-_\infty,R^+_\infty\right) \\ 
        \begin{aligned}[t] & \times  
            \det\Bigg[\begin{pmatrix}L^-_\infty\e^{-A^0_\infty \Pi_u X^\ve_\ell}R^{-}_\infty & 0 \\ 0 & L^+_\infty\e^{-A^0_\infty \Pi_s X^\ve_r}R^+_\infty \end{pmatrix}
             + \go\left(\e^{-\ti{\nu} \min(X^\ve_r,\abs{X^\ve_\ell})} \right) \Bigg]\\
            &\times \det\begin{pmatrix}L^-_\infty P^\ve_{+,\ell}(0)^{-1}P^\ve_{-,\ell}(0)R^-_\infty & 0 \\ 0 & L^+_\infty P^\ve_{-,r}(0)^{-1}P^\ve_{+,r}(0)R^+_\infty \end{pmatrix}\\
            &\times \det\begin{pmatrix}\id & 0 \\ 0 & \gamma L^+_\infty \e^{-A^0_\infty \Pi_s X^0 }R^+_\infty \end{pmatrix}
            \times \det\left(\id + \go(\e^{-\ti{\nu} \min\left(\frac{X^\ve}{2},\abs{2 X^\ve_\ell}\right)})\right) \det\begin{pmatrix} L^-_\infty \\ L^+_\infty \end{pmatrix} . \end{aligned}
    \end{multline*}
    So using the Abel's formula and $\left(\begin{smallmatrix} L^-_\infty \\ L^+_\infty\end{smallmatrix}\right) = \left(R^-_\infty,R^+_\infty \right)^{-1}$ we get 
    \begin{align*}
        \ti{D}^\ve(z,\gamma) =& \e^{\tr A_\infty^0 \Pi_s (X^\ve_r+X^\ve_\ell - X^0)}  \det\Big(L^-_\infty P^\ve_{+,\ell}(0)^{-1} P^\ve_{-,\ell}(0) R^-_\infty\Big)\det\Big(L^+_\infty P^\ve_{-,r}(0)^{-1} P^\ve_{+,r}(0) R^+_\infty\Big)\\ 
                                    &\times \det\left(\id + \go\left(\e^{-\ti{\nu} \min\left(\frac{X^\ve}{2},\abs{2 X^\ve_\ell}\right)}\right)\right)^2.
    \end{align*}
    With~\eqref{cv_matrice_P_l},~\eqref{cv_matrice_P_r}, Definition~\ref{objet_limite_fonction_evans} and~\eqref{cv_demi_période}, we deduce
    \begin{equation*}
        \ti{D}^\ve(z,\gamma) = \ti{D}^0(z,\gamma) + \go\left(\e^{-\ti{\nu} \min\left(\frac{X^\ve}{2},\abs{2 X^\ve_\ell}\right)}\right).
    \end{equation*}
    By using that $X^\ve_\ell = - \frac{1}{2}a_s\sqrt{2\ka(v_s)}\ln(\ve) + \go(1)$ and $X^\ve = - 2a_s\sqrt{2\ka(v_s)}\ln\ve + \go(1)$ we arrive at the claimed expansion 
    \begin{equation*}
        \ti{D}^\ve(z,\gamma) = \ti{D}^0(z,\gamma) + \go\left(\e^{-\ti{\nu} \frac{X^\ve}{2}}\right).
    \end{equation*}
\end{proof}

\begin{corollaire}\label{spectrum_periodic_evans_function}
    On each compact $K \subset \C_+$ such that $\ti{D}^0$ does not vanish on $\partial K$, the spectrum of $\Al^\ve$ for $X^\ve$ sufficiently large consists of loops of spectra $z^\ve_{\ast,k}(\gamma),~k=1,\dots,m_r$, within $\go\left(\e^{-\ti{\nu} \frac{X^\ve}{2 m_\ast}}\right)$ of an eigenvalue $z_\ast \in K$ of $\Al^0_r$ or $\Al^0_\ell$. Here $m_\ast$ is the sum of the multiplicities of $z_\ast$ as an eigenvalue of $\Al^0_r$ and $\Al^0_\ell$ and $\ti{\nu}$ is as in Proposition~\ref{cv_evans_away_essential_spectra}. 
\end{corollaire}

\begin{proof}
    It follows from Proposition~\ref{zero_D} and the Rouché theorem, like as for Corollary~5.1 in \cite{yangConvergencePeriodGoes2019}.
\end{proof}

With this result in hand we derive the second half of Theorem~\ref{main_results}. 

\begin{corollaire}
With Assumption~\ref{hypothèse_Hm} and~\ref{hypothèse_potentiel}, if one of the solitons is spectrally unstable, then two-pulse periodic waves of sufficiently large period are spectrally unstable.
\end{corollaire}

\begin{proof}
    If one of the solitons is spectrally unstable, we know that $\Al^0_r$ or $\Al^0_\ell$ admits an eigenvalue $z_\ast \in \C_+$. So with Corollary~\ref{spectrum_periodic_evans_function}, for $\ve$ small enough, $\Al^\ve$ admits a loop of spectrum as close as we want to $z_\ast$. This tells us exactly that the wave is spectrally unstable.  
\end{proof}

\begin{remarque}
With Theorem~\ref{theorem_stability_instability_soliton} and this corollary, we conclude in the case left open in Section~\ref{double_bosse}. If $\dd_c^2 \mo_\ell + \dd_c^2 \mo_r < 0$, then two-pulse periodic waves of sufficiently large period are spectrally unstable.
\end{remarque}

\begin{remarque}
Though it is not useful in the stability analysis, the reader may wonder how the essential spectrum parts of the spectrum of the two solitons impacts the spectrum of the two-solitons. In the single pulse case this is analyzed in \cite{yangConvergencePeriodGoes2019} and the arguments may also be extended to the two-pulses case. Even we do not provide a proof of this claims let us describe the outcome. As for the spectrum under localized perturbations the curves of essential spectrum of the two solitons also perturb continuously so as to provide curves for the two-pulse periodic waves, even though the local parametrization by the Floquet multiplier $\gamma$ becomes singular. If one restricts to co-periodic perturbations, that is if one sets $\gamma\equiv1$, then one obtains a discrete sampling of the perturbed continuous curves.
\end{remarque}

\appendix
\section{Asymptotic behavior of integrals }\label{dl_intégrale}

\subsection{The general integrals}
We enforce Assumption~\ref{hypothèse_potentiel}. The goal is to compute the asymptotic expansion of  
\begin{equation*}
    I(\phi) = \int_{v_\ell}^{v_r} \frac{\phi(v)}{\sqrt{\mu - \w(v)}} \dd v.
\end{equation*}
when $\ve = \sqrt{\mu-\mu_s} \rightarrow 0$ and $\phi$ is a smooth function. \newline
One may check that for $\mu > \mu_s$, $\la(v) = \mu - \w(v)$ admits two simple roots, which are $v_\ell$ and $v_r$. But, at the limits, when $\mu = \mu_s$, $\la(v) = \mu_s - \w(v)$ admits two simple roots $v_\ell^s$ and $v_r^s$, and a double root $v_s$. So the difficulty of this calculation is around the double root that appears at the limit. To make the computation easier we split the integral into different parts, where each root appears separately. To do so, using Assumption~\ref{hypothèse_potentiel}, 
we first introduce $\eta > 0$ such that for all $(\mu,\lm,c) \in\Omega$, in a neighborhood of $(\mu_s,\lm,c)$, we have
\begin{equation}\label{def_eta_v_s}
    \forall v \in \ff{v_s(\lm,c)-\eta}{v_s(\lm,c)+\eta}, \quad \p_v^2\w(v;\lm,c) < 0,
\end{equation}
and
\begin{align}\label{def_eta_1}
    &\forall v \in \ff{v_\ell(\mu,\lm,c)}{v_\ell^s(\lm,c)+\eta}, \quad \p_v\w(v;\lm,c) < 0, \\
    &\forall v \in \ff{v_r^s(\lm,c)-\eta}{v_r(\mu,\lm,c)}, \quad \p_v\w(v;\lm,c) > 0.\label{def_eta_2}
\end{align}
Nonsingular parts are
\begin{equation*}
    I_{m,r}(\phi) = \int_{v_s+\eta}^{v_s^r-\eta} \frac{\phi(v)}{\sqrt{\mu - \w(v)}}\dd v \quad \text{and} \quad I_{m,\ell}(\phi) = \int_{v_\ell^s+\eta}^{v_s-\eta} \frac{\phi(v)}{\sqrt{\mu - \w(v)}}\dd v.
\end{equation*}
Asymptotic expansions for these integrals are calculated in Proposition~\ref{dl_mid} in the next subsection.\newline
Parts with a simple root are
\begin{equation*}
    I_r(\phi) := \int_{v_r^s-\eta}^{v_r} \frac{\phi(v)}{\sqrt{\mu-\w(v)}} \dd v \quad \text{and} \quad I_\ell(\phi) := \int_{v_\ell}^{v_\ell^s+\eta} \frac{\phi(v)}{\sqrt{\mu-\w(v)}} \dd v.
\end{equation*}
Their asymptotic expansion is provided in Propositions~\ref{dl_droite} and~\ref{dl_gauche}.\newline
The hardest part,
\begin{equation*}
    I_{\eta}(\phi) = \int_{v_s-\eta}^{v_s+\eta} \frac{\phi(v)}{\sqrt{\mu - \w(v)}}\dd v,
\end{equation*}
whose asymptotic expansion is given in Subsection~\ref{dl_singularité_milieux}.\newline
Gathering all the results, we obtain the following theorem, which gives the asymptotic behavior of the integral $I(\phi)$.

\begin{theorem}\label{dl_int_phi}
With Notation~\ref{note_lag},~\ref{note_dl} and Assumption~\ref{hypothèse_potentiel}, for $\ve = \sqrt{\mu - \mu_s} \rightarrow 0$, we have the following asymptotic expansion
\begin{equation*}
    I(\phi) = A_0(\phi)\ln(\ve) + B_0(\phi) + B_1(\phi)\ve + A_2(\phi)\ve^2 \ln(\ve) + B_2(\phi) \ve^2 + \go\left(\ve^3\ln(\ve)\right),
\end{equation*}
with
\begin{enumerate}[label=\roman*)]
    \item $A_0(\phi) = -\frac{2\phi_s}{\sqrt{-R_s}}$,
    \item $B_0(\phi) =  \begin{aligned}[t] &\int_{v_\ell^s}^{v_s-\eta} \frac{\phi(v)}{\abs{v-v_s}\sqrt{\abs{R(v)}}} \dd v + \int_{v_s+\eta}^{v_r^s} \frac{\phi(v)}{\abs{v-v_s}\sqrt{\abs{R(v)}}} \dd v
                        +2\frac{\phi_s}{\sqrt{-R_s}}\ln(2\eta\sqrt{-R_s})\\ &+ \int_{v_s}^{v_s+\eta}\frac{1}{v-v_s}\left(\frac{\phi(v)}{\sqrt{-R(v)}}-\frac{\phi_s}{\sqrt{-R_s}}\right) \dd v
                        + \int_{v_s-\eta}^{v_s}\frac{-1}{v-v_s}\left(\frac{\phi(v)}{\sqrt{-R(v)}}-\frac{\phi_s}{\sqrt{-R_s}}\right)\dd v,\end{aligned}$
    \item $B_1(\phi) = -2\frac{\phi_s}{(-R_s)}$,
    \item $A_2(\phi) = 3\frac{\phi_{v,s}R_{v,s}+\phi_sR_{vv,s}}{4(-R_s)^{5/2}} + \frac{15\phi_sR_{v,s}^2}{8(-R_s)^{7/2}}$,
    \item $B_2(\phi) = \begin{aligned}[t]
                        &-\int_{v_s+\eta}^{v_s^r-\eta} \frac{1}{2} \frac{\phi(v)}{(\mu_s - \w(v))^{3/2}}\dd v - \int_{v_\ell^s+\eta}^{v_s-\eta} \frac{1}{2} \frac{\phi(v)}{(\mu_s - \w(v))^{3/2}}\dd v\\
                        &+ \int_{v_\ell^s}^{v_\ell^s+\eta} \frac{\frac{v_\ell^s+\eta-v}{\eta} p_\ell^s\ti{f}_v(v,v_r^s,v_\ell^s) +\left(p_r^s + p_\ell^s \right)\ti{f}_z(v,v_r^s,v_\ell^s)}{\sqrt{-(v-v_r^s)(v - v_\ell^s)}}\dd v\\
                        &- \int_{v_\ell^s}^{v_\ell^s+\eta} \frac{\ti{f}(v,v_r^s,v_\ell^s)}{2}\frac{p_r^s+\delta^s p_\ell^s}{\sqrt{(v-v_\ell^s)}(v_r^s-v)^{3/2}} \dd v\\
                        &+ \int_{v_r^s-\eta}^{v_r^s} \frac{\frac{v-v_r^s+\eta}{\eta} p_r^s\ti{f}_v(v,v_r^s,v_\ell^s) +\left(p_r^s + p_\ell^s \right)\ti{f}_z(v,v_r^s,v_\ell^s)}{\sqrt{-(v-v_r^s)(v - v_\ell^s)}}\dd v\\ 
                        &+ \int_{v_r^s-\eta}^{v_r^s} \frac{\ti{f}(v,v_r^s,v_\ell^s)}{2}\frac{p_\ell^s+\delta^s p_r^s}{\sqrt{(v_r^s-v)}(v-v_\ell^s)^{3/2}} \dd v
                        + B_2^{\eta,-}(\phi) + B_2^{\eta,+}(\phi)\end{aligned}$
    \end{enumerate}
    where $B_2^{\eta,+}(\phi)$ and $B_2^{\eta,-}(\phi)$ are defined in Proposition~\ref{dl_int_positive} and~\ref{dl_int_negative}, $\ti{f}(v,w,z) = \frac{\phi(v)}{\sqrt{\ti{R}(v,w,z)}}$, $R(v) = \ti{R}(v,v_s,v_s)$, $R_s = R(v_s)$, $\phi_s = \phi(v_s)$ and
    \begin{equation*}
        \ti{R}(v,w,z) = \int_{0}^{1} \int_0^1 t\p_{v}^2\w\left(w+t(w-z)+t u(v-z)\right)\dd u \dd t.
    \end{equation*}
\end{theorem}

\subsubsection{Rootless parts}

\begin{proposition}\label{dl_mid}
    For $\ve\rightarrow0$ we have the following asymptotic expansion
    \begin{equation*}
        I_{m,r}(\phi) = \int_{v_s+\eta}^{v_s^r-\eta} \frac{\phi(v)}{\sqrt{\mu_s - \w(v)}} \dd v -\ve^2 \int_{v_s+\eta}^{v_s^r-\eta} \frac{1}{2} \frac{\phi(v)}{(\mu_s - \w(v))^{3/2}} \dd v + \go(\ve^4),
    \end{equation*}
    and
    \begin{equation*}
        I_{m,\ell}(\phi) = \int_{v_\ell^s+\eta}^{v_s-\eta} \frac{\phi(v)}{\sqrt{\mu_s - \w(v)}} \dd v -\ve^2 \int_{v_\ell^s+\eta}^{v_s-\eta} \frac{1}{2} \frac{\phi(v)}{(\mu_s - \w(v))^{3/2}}\dd v + \go(\ve^4).
    \end{equation*}
\end{proposition}

\begin{proof}
This follows readily by integration from
\begin{equation*}
    \frac{1}{\sqrt{\ve^2+\mu_s - \w(v)}} = \frac{1}{\sqrt{\mu_s - \w(v)}} -\frac{1}{2} \frac{\ve^2}{(\mu_s - \w(v))^{3/2}} + \go(\ve^4).
\end{equation*}
\end{proof}

\subsubsection{Single-root parts}

In the same spirit of \cite{benzoni-gavageStabilityPeriodicWaves2020}, using Taylor's formula we write
\begin{equation*}
    \mu - \w(v) = -(v-v_r)(v-v_\ell)\ti{R}(v,v_r,v_\ell),
\end{equation*}
where $\ti{R}(v,w,z) := \int_{0}^{1} \int_0^1 t\p_{v}^2\w\left(w+t(w-z)+tu(v-z)\right)\dd u \dd t$.\newline
First, due to~\eqref{def_eta_1},~\eqref{def_eta_2},we can assume that $\eta$ is such that  $$\forall v\in\ff{v_l}{v_l^s+\eta}\cup\ff{v_r^s-\eta}{v_r}, \quad  \ti{R}(v,v_r,v_l) >0.$$
Furthermore, with a change of variable, we have
\begin{align*}
    I_r(\phi) 
              = \int_0^1 \frac{\phi(\V_r)}{\sqrt{(1-\sig)(\delta_r+\sig)\ti{R}(\V_r,v_r,v_\ell)}}\dd \sig,
\end{align*}
where $\V_r = v_r^s - \eta + \sig  (v_r-v_r^s+\eta)$ and $\delta_r = \frac{v_r^s-\eta - v_\ell}{v_r-v_r^s+\eta}$.\newline
The same applies to the other integral, with $\V_\ell = v_\ell^s + \eta + \sig(v_\ell - v_\ell^s - \eta)$ and $\delta_\ell = \frac{v_r-v_\ell^s-\eta}{v_\ell^s+\eta-v_\ell}$,
\begin{equation*}
    I_\ell(\phi) = \int_0^1 \frac{\phi(\V_\ell)}{\sqrt{(1-\sig)(\delta_\ell+\sig)\ti{R}(\V_\ell,v_r,v_\ell)}}\dd \sig.
\end{equation*}
We begin with the right-hand integral, $I_r(\phi)$.

\begin{proposition}\label{dl_droite}
    The following asymptotic expansion holds 
    \begin{align*}
        I_r(\phi) =& \int_{v_r^s-\eta}^{v_r^s} \frac{\ti{f}(v,v_r^s,v_\ell^s)}{\sqrt{-(v-v_r^s)(v - v_\ell^s)}}\dd v + \ve^2 \int_{v_r^s-\eta}^{v_r^s} \frac{\frac{v-v_r^s+\eta}{\eta} p_r^s\ti{f}_v(v,v_r^s,v_\ell^s) +\left(p_r^s + p_\ell^s \right)\ti{f}_z(v,v_r^s,v_\ell^s)}{\sqrt{-(v-v_r^s)(v - v_\ell^s)}}\dd v\\ 
        &+ \ve^2 \int_{v_r^s-\eta}^{v_r^s} \frac{\ti{f}(v,v_r^s,v_\ell^s)}{2}\frac{p_\ell^s+\delta^s p_r^s}{\sqrt{(v_r^s-v)}(v-v_\ell^s)^{3/2}} \dd v+ \go(\ve^4),
    \end{align*}
    where $\ti{f}(v,w,z) = \frac{\phi(v)}{\sqrt{\ti{R}(v,w,z)}}$.
\end{proposition}

\begin{proof}
    We define
    \begin{equation*}
        \ti{f}(v,w,z) = \frac{\phi(v)}{\sqrt{\ti{R}(v,w,z)}},
    \end{equation*}
    and note that since $\ti{R}$ is a symmetric function of $(w,z)$, $\ti{f}$ is also symmetric in $(w,z)$. Then we introduce $\V_r^s = v_r^s - \eta + \eta \sig$, $\delta^s = \frac{v_r^s-\eta - v_\ell^s}{\eta}$ and for every function $g$ depending on $(v,w,z)$, $g^{s,r} := g(\V_r^s,v_r^s,v_\ell^s)$. We start by developing $\ti{f}$ in power of $\ve$. For this, we use the asymptotic expansion of $v_r$ and $v_\ell$ computed in Proposition~\ref{dl_racines}.
    \begin{align*}
        \ti{f}(\V_r,v_r,v_\ell) =& \ti{f}^{s,r} + (\V_r - \V_r^s)\ti{f}_v^{s,r} + (v_r-v_r^s)\ti{f}_w^{s,r} + (v_\ell - v_\ell^s)\ti{f}_z^{s,r} + \go\left((v_\ell - v_\ell^s)^2 + (v_r-v_r^s)^2 + (\V_r - \V_r^s)^2 \right)\\
                                =& \ti{f}^{s,r} + \sig(v_r-v_r^s)\ti{f}_v^{s,r} + \left((v_r-v_r^s) + (v_\ell - v_\ell^s) \right)\ti{f}_z^{s,r} +\go\left((v_\ell - v_\ell^s)^2 + (v_r-v_r^s)^2 + (\V_r - \V_r^s)^2 \right)\\
                                =& \ti{f}^{s,r} + \ve^2\sig p_r^s\ti{f}_v^{s,r} + \ve^2\left(p_r^s + p_\ell^s \right)\ti{f}_z^{s,r} +\go\left(\ve^4\right).
    \end{align*}
    We also need to calculate the asymptotic expansion of  $\delta_r$.
    \begin{align*}
        \delta_r =& \frac{v_r^s-\eta-v_\ell}{v_r-v_r^s+\eta}\\
                 =& \left(v_r^s-\eta-v_\ell^s-\ve^2p_\ell^s + \go(\ve^4)\right)\left(\eta + \ve^2p_r^s + \go(\ve^4)\right)^{-1}\\
                 =& \delta^s - \ve^2\frac{p_\ell^s+\delta^sp_r^s}{\eta}  + \go(\ve^4).
    \end{align*}
    With this in hands, we have
    \begin{align*}
        \frac{\ti{f}(\V_r,v_r,v_\ell)}{\sqrt{\delta_r + \sig}} =& \left(\ti{f}^{s,r} + \ve^2\sig p_r^s\ti{f}_v^{s,r} + \ve^2\left(p_r^s + p_\ell^s \right)\ti{f}_z^{s,r} +\go\left(\ve^4\right)\right)\left(\frac{1}{\sqrt{\delta^s + \sig}} + \frac{1}{2}\frac{p_\ell^s+\delta^sp_r^s}{\eta(\delta^s+\sig)^{3/2}}\ve^2 + \go(\ve^4)\right)\\
                =& \frac{\ti{f}^{s,r}}{\sqrt{\delta^s + \sig}} + \frac{\sig p_r^s\ti{f}_v^{s,r} +\left(p_r^s + p_\ell^s \right)\ti{f}_z^{s,r}}{\sqrt{\delta^s + \sig}}\ve^2 + \frac{\ti{f}^{s,r}}{2}\frac{p_\ell^s+\delta^sp_r^s}{\eta(\delta^s+\sig)^{3/2}}\ve^2 + \go(\ve^4).
    \end{align*}
    Thus we conclude
    \begin{multline*}
        I_r(\phi) = \int_0^1\frac{\ti{f}^{s,r}}{\sqrt{(1-\sig)(\delta^s + \sig)}}\dd \sig + \ve^2 \int_0^1 \frac{\sig p_r^s\ti{f}_v^{s,r} +\left(p_r^s + p_\ell^s \right)\ti{f}_z^{s,r}}{\sqrt{(1-\sig)(\delta^s + \sig)}}\dd \sig + \ve^2 \int_0^1 \frac{\ti{f}^{s,r}}{2}\frac{p_\ell^s+\delta^sp_r^s}{\eta\sqrt{1-\sigma}(\delta^s+\sig)^{3/2}} \dd \sig\\ + \go(\ve^4),
    \end{multline*}
    which, with the change of variable $v=v_r^s-\eta+\sig \eta$, yields the desired result.
\end{proof}

We do the same for the left-hand integral and obtain the following result.

\begin{proposition}\label{dl_gauche}
    We have the following asymptotic expansion
    \begin{align*}
        I_\ell(\phi) =& \int_{v_\ell^s}^{v_\ell^s+\eta} \frac{\ti{f}(v,v_r^s,v_\ell^s)}{\sqrt{-(v-v_r^s)(v - v_\ell^s)}}\dd v + \ve^2 
        \int_{v_\ell^s}^{v_\ell^s+\eta} \frac{\frac{v_\ell^s+\eta-v}{\eta} p_\ell^s\ti{f}_v(v,v_r^s,v_\ell^s) +\left(p_r^s + p_\ell^s \right)\ti{f}_z(v,v_r^s,v_\ell^s)}{\sqrt{-(v-v_r^s)(v - v_\ell^s)}}\dd v\\ 
        &- \ve^2 \int_{v_\ell^s}^{v_\ell^s+\eta} \frac{\ti{f}(v,v_r^s,v_\ell^s)}{2}\frac{p_r^s+\delta^s p_\ell^s}{\sqrt{(v-v_\ell^s)}(v_r^s-v)^{3/2}} \dd v+ \go(\ve^4).
    \end{align*}
    where $\ti{f}(v,w,z) = \frac{\phi(v)}{\sqrt{\ti{R}(v,w,z)}}$.
\end{proposition}

\subsubsection{The double-root part}\label{dl_singularité_milieux}

We are now interested in the part around the double root $v_s$, 
\begin{align*}
    I_{\eta}(\phi) =& \int_{v_s-\eta}^{v_s+\eta} \frac{\phi(v)}{\sqrt{\mu - \w(v)}}\dd v\\
    =& \int_{-\eta}^{\eta} \frac{\phi(v_s+w)}{\sqrt{\ve^2 - w^2 \ti{R}(v_s+w,v_s,v_s)}}\dd w \\
    =& \int_{\frac{-\eta}{\ve}}^{\frac{\eta}{\ve}} \frac{\phi(v_s+\ve\sig)}{\sqrt{1 - \sig^2 \ti{R}(v_s+\ve\sig,v_s,v_s)}} \dd \sig,
\end{align*}
where we have performed the change of variable $w=\ve\sig$ and used that $$\mu - \w(v) = \ve^2 + \mu_s -\w(v) = \ve^2 -(v-v_s)^2\ti{R}(v,v_s,v_s)$$
From now on, for the sake of concision, we will write $R(v) =\ti{R}(v,v_s,v_s) $. Thanks to~\eqref{def_eta_v_s}, we may assume that $\eta$ is such that $R(v)<0$ for $v \in \ff{v_s-\eta}{v_s+\eta}$. To simplify computations, we split the integral into
\begin{equation*}
    \intwp \frac{\phi(v_s+w)}{\sqrt{\ve^2-w^2R(v_s+w)}} \dd w = \intsigp \frac{\phi(v_s+\ve\sig)}{\sqrt{1 - \sig^2 \ti{R}(v_s+\ve\sig)}} \dd \sig,
\end{equation*}
and
\begin{equation*} 
    \intwn \frac{\phi(v_s+w)}{\sqrt{\ve^2-w^2R(v_s+w)}} \dd w = \intsign \frac{\phi(v_s+\ve\sig)}{\sqrt{1 - \sig^2 \ti{R}(v_s+\ve\sig)}} \dd \sig .
\end{equation*}
The expansion of the latter integral will be deduced from the expansion of the former integral. Thus we begin with the first integral. \newline

The expansion begins with a $\ln(\ve)$ term then an $\go(1)$ term follows. The goal is to identify them. For this, using Taylor's formula we get
\begin{equation*}
    \frac{\phi(v_s+h)}{\sqrt{ 1 - \sig^2 R(v_s+h)}} = \frac{\phi_s}{\sqrt{1-\sig^2R_s}} + h \int_{0}^{1} \frac{\phi_v(v_s+th)}{\sqrt{ 1 - \sig^2 R(v_s+th)}} + \phi(v_s+t h) \frac{\sig^2R_v(v_s+th)}{2(1-\sig^2R(v_s+th))^{3/2}}\dd t.
\end{equation*}
We deal with each term separately. The first term grows as $\ln(\ve)$. Indeed Lemma~\ref{ln}, gives
\begin{align*}
    I_{\eta,+}^{-1}(\phi) :=& \intsigp \frac{\phi_s}{\sqrt{1-\sig^2R_s}} \dd \sig \\
    =& \phi_s \int_{0}^{\frac{\eta\sqrt{-R_s}}{\ve}} \frac{1}{\sqrt{1+x^2}} \frac{\dd x}{\sqrt{-R_s}} \\
    =& -\frac{\phi_s}{\sqrt{-R_s}} \ln(\ve) + \frac{\phi_s}{\sqrt{-R_s}}\ln(2\eta\sqrt{-R_s}) + \frac{\phi_s}{\sqrt{-R_s}} \frac{\ve^2}{-4\eta^2R_s} + \go(\ve^3).
\end{align*}
The other term is bounded since
\begin{align*}
    I_{\eta,+}^0(\phi) :=& \intsigp \ve\sig \int_{0}^{1} \frac{\phi_v(v_s+t\ve\sig)}{\sqrt{ 1 - \sig^2 R(v_s+t\ve\sig)}} + \phi(v_s+t\ve\sig) \frac{\sig^2R_v(v_s+t\ve\sig)}{2(1-\sig^2R(v_s+t\ve\sig))^{3/2}}\dd t \dd\sig \\
    =& \intwp \int_{0}^{1} \frac{\phi_v(v_s+t w)w}{\sqrt{ \ve^2 - w^2 R(v_s+t w)}} + \phi(v_s+t w) \frac{w^3R_v(v_s+t w)}{2(\ve^2-w^2R(v_s+t w))^{3/2}}\dd t \dd w.
\end{align*}
and we can simply use the dominated convergence theorem to obtain
\begin{equation*}
    I_{\eta,+}^0(\phi) \tend{\ve}{0} I_{\eta,+}^{0,\text{lim}}(\phi) := \intwp \int_{0}^{1} \frac{\phi_v(v_s+t w)}{\sqrt{-R(v_s+t w)}} + \phi(v_s+t w) \frac{R_v(v_s+t w)}{2(-R(v_s+t w))^{3/2}} \dd t \dd w.
\end{equation*}
Thereby, we only have to find the asymptotic expansion of $I_{\eta,+}^0(\phi) - I_{\eta,+}^{0,\text{lim}}(\phi)$. This is done in the next proposition.

\begin{proposition}\label{dl_int_positive}
We have the following asymptotic expansion
\begin{equation*}
    \intwp \frac{\phi(v_s+w)}{\sqrt{\ve^2-w^2R(v_s+w)}} \dd w = A_0^{\eta,+}(\phi) \ln(\ve) + B_0^{\eta,+}(\phi) + B_1^{\eta,+}(\phi)\ve + A_2^{\eta,+}(\phi)\ve^2\ln(\ve) + B_2^{\eta,+}(\phi)\ve^2 + \go\left(\ve^3\ln(\ve)\right),
\end{equation*}
where each coefficient is given by
\begin{enumerate}[label=\roman*)]
    \item $A_0^{\eta,+}(\phi) =-\frac{\phi_s}{\sqrt{-R_s}}, $
    \item $B_0^{\eta,+}(\phi) = \frac{\phi_s}{\sqrt{-R_s}}\ln\left(2\eta\sqrt{-R_s}\right) +  \intwp \frac{1}{w}\left(\frac{\phi(v_s+w)}{\sqrt{-R(v_s+w)}} - \frac{\phi_s}{\sqrt{R_s}} \right)\dd w,$
    \item $B_1^{\eta,+}(\phi) = - \frac{\phi_s}{-R_s} - \frac{\phi_sR_{v,s}}{(-R_s)^2},$
    \item $A_2^{\eta,+}(\phi) = \frac{1}{4(-R_s)^{3/2}}\phi_{v,s} + \frac{3}{8(-R_s)^{5/2}}\phi_sR_{v,s} + 3\frac{\phi_{v,s}R_{v,s}+\phi_sR_{vv,s}}{8(-R_s)^{5/2}} + \frac{15\phi_sR_{v,s}^2}{16(-R_s)^{7/2}},$
    \item $\begin{aligned}[t] &B_2^{\eta,+}(\phi) = \frac{\phi_s}{4\eta^2(-R_s)^{3/2}} + \frac{\phi_s}{2\eta(-R_s)^{3/2}} +\frac{\phi_{v,s}}{2(-R_s)^{3/2}}\left(\frac{1}{4} - \frac{1}{2} \ln(2\eta\sqrt{-R_s})\right) + \frac{\phi_sR_{v,s}}{4(-R_s)^{5/2}}\left(\frac{5}{4} - \frac{3}{2} \ln\left(2\eta\sqrt{-R_s}\right)\right)\\
                        &+\frac{3}{4}\frac{\phi_sR_{v,s}}{\eta(-R_s)^{5/2}} + \frac{3\phi_sR_{v,s}^2}{8(-R_s)^{7/2}}\left(\frac{31}{12} - \frac{5}{2}        \ln\left(2\eta\sqrt{-R_s}\right)\right) + \frac{\phi_{v,s}R_{v,s}+\phi_sR_{vv,s}}{4(-R_s)^{5/2}}\left(\frac{5}{4} - \frac{3}{2} \ln\left(2\eta\sqrt{-R_s}\right)\right) \\
                        &+ \intwp \int_0^1 \int_0^1 t^2(1-u) \Bigg[-\frac{1}{2}\frac{\phi_{vv}}{(-R)^{3/2}} -\frac{3}{2}\frac{\phi_v R_v}{(-R)^{5/2}} - \frac{3}{4}\frac{\phi R_{vv}}{(-R)^{5/2}}- \frac{15}{8}\frac{\phi R_v^2}{(-R)^{7/2}}\Bigg](v_s+u t w) \dd u \dd t \dd w \end{aligned}$\end{enumerate}
                        $$+ \intwp\int_0^1 \int_0^1 t^2(1-u) \Bigg[-3\frac{\phi_{vv}R_{v}+2\phi_vR_{vv}+\phi R_{vvv}}{4(-R)^{5/2}} - \frac{15}{4}\frac{\phi_v R_v^2 + \frac{3}{2}\phi R_vR_{vv}}{(-R)^{7/2}} - \frac{105}{16}\frac{\phi R_v^3}{(-R)^{9/2}} \Bigg](v_s+u t w) \dd u \dd t \dd w.$$

\end{proposition}

\begin{proof}
As discussed above there only remains to expand
\begin{align*}
    I_{\eta,+}^0(\phi) - I_{\eta,+}^{0,\text{lim}}(\phi) =& \intwp \int_{0}^{1} \frac{\phi_v(v_s+t w)w}{\sqrt{ \ve^2 - w^2 R(v_s+t w)}} - \frac{\phi_v(v_s+t w)}{\sqrt{-R(v_s+t w)}} \dd t \dd w\\ &+ \intwp \int_{0}^{1} \phi(v_s+t w) \frac{w^3R_v(v_s+t w)}{2(\ve^2-w^2R(v_s+t w))^{3/2}} - \phi(v_s+t w) \frac{R_v(v_s+t w)}{2(-R(v_s+t w))^{3/2}} \dd t \dd w \\
    =&: I_{\eta,+}^{0,1}(\phi) + I_{\eta,+}^{0,2}(\phi)
\end{align*}
We begin by dealing with the first term, 
\begin{equation*}
    I_{\eta,+}^{0,1}(\phi) = \intsigp \int_0^1 \ve \phi(v_s + t\ve\sigma)\left(\frac{\sigma}{\sqrt{1-\sig^2R(v_s+t\ve\sig)}} - \frac{1}{\sqrt{-R(v_s+t\ve\sig)}}\right) \dd t \dd \sig.
\end{equation*}
As before, we use Taylor's formula,
\begin{align*}
    \frac{\sigma\phi(v_s + th)}{\sqrt{1-\sig^2R(v_s+th)}} - \frac{\phi(v_s + th)}{\sqrt{-R(v_s+th)}} =& \frac{\phi_s\sig}{\sqrt{1-\sig^2R_s}} - \frac{\phi_s}{\sqrt{-R_s}}\\ &+ t h\left(\frac{\phi_{v,s}\sig}{\sqrt{1-\sig^2R_s}} - \frac{\phi_{v,s}}{\sqrt{-R_s}} + \phi_s\frac{R_{v,s}\sig^3}{2(1-\sig^2R_s)^{3/2}} - \phi_s\frac{R_{v,s}}{2(-R_s)^{3/2}} \right)\\
    &+ (t h)^2 \!\begin{aligned}[t]
    \int_0^1 & (1-u)f_1(v_s + u t h,\sig) \dd u,
    \end{aligned}
\end{align*}
where
\begin{align*}
    f_1(\cdot,\sig) =& \frac{\phi_{vv}\sig}{\sqrt{1-\sig^2R}} - \frac{\phi_{vv}}{\sqrt{-R_s}} + 2\phi_v\left(\frac{R_v\sig^3}{2(1-\sig^2R)^{3/2}} - \frac{R_{v}}{2(-R)^{3/2}} \right)
    + \phi\left(\frac{R_{vv}\sig^3}{2(1-\sig^2R)^{3/2}} - \frac{R_{vv}}{2(-R)^{3/2}} \right)\\ &+ \frac{3}{4}\phi\left(\frac{R_v^2\sig^5}{(1-\sig^2R)^{5/2}} - \frac{R_{v}^2}{(-R)^{5/2}} \right).
\end{align*}
Thus
\begin{align*}
    I_{\eta,+}^{0,1}(\phi) =& \phi_s \ve \intsigp \frac{\sig}{\sqrt{1-\sig^2R_s}} - \frac{1}{\sqrt{-R_s}} \dd \sig + \frac{1}{2}\phi_{v,s} \ve^2 \intsigp \frac{\sig^2}{\sqrt{1-\sig^2R_s}} - \frac{\sig}{\sqrt{-R_s}} \dd \sig\\
    &+ \frac{1}{4}\phi_sR_{v,s} \ve^2 \intsigp \frac{\sig^4}{(1-\sig^2R_s)^{3/2}} - \frac{\sig}{(-R_s)^{3/2}} \dd \sig 
    +  \intsigp \int_0^1 \ve (t\ve\sig)^2 \int_0^1 (1-u)f(v_s + u t\ve\sig,\sig) \dd u \dd t \dd \sig.
\end{align*}
We can expand the first three terms in powers of $\ve$ using the asymptotic expansions given by Lemma~\ref{i_c,+,0,1}. For the first term, we have
\begin{align*}
    \phi_s \ve \intsigp \frac{\sig}{\sqrt{1-\sig^2R_s}} - \frac{1}{\sqrt{-R_s}} \dd \sig =& \frac{\phi_s}{\sqrt{-R_s}} \ve \intsigp \frac{\sqrt{-R_s}\sig}{\sqrt{1-\sig^2R_s}} - 1 \dd \sig\\
    =& \frac{\phi_s}{-R_s} \ve \int_0^{\frac{\eta\sqrt{-R_s}}{\ve}} \frac{x}{\sqrt{1+x^2}} - 1 \dd x\\
    =& \frac{\phi_s}{-R_s} \ve \left(- 1 +\frac{1}{2}\frac{\ve}{\eta\sqrt{-R_s}} + \go(\ve^3)\right)\\
    =& - \frac{\phi_s}{-R_s} \ve + \frac{1}{2}\frac{\phi_s}{\eta(-R_s)^{3/2}}\ve^2 + \go(\ve^4).
\end{align*}
For the second one, we have
\begin{align*}
    \frac{1}{2}\phi_{v,s} \ve^2 \intsigp \frac{\sig^2}{\sqrt{1-\sig^2R_s}} - \frac{\sig}{\sqrt{-R_s}} \dd \sig =& \frac{1}{2(-R_s)}\phi_{v,s} \ve^2 \intsigp \frac{(-R_s)\sig^2}{\sqrt{1-\sig^2R_s}} - \sig\sqrt{-R_s} \dd \sig\\
    =& \frac{1}{2(-R_s)^{3/2}}\phi_{v,s} \ve^2 \int_0^{\frac{\eta\sqrt{-R_s}}{\ve}} \frac{x^2}{\sqrt{1+x^2}} - x \dd x \\
    =& \frac{1}{2(-R_s)^{3/2}}\phi_{v,s}\ve^2 \left(\frac{1}{4} + \frac{1}{2}\ln(\ve) - \frac{1}{2} \ln(2\eta\sqrt{-R_s}) + \go(\ve^2) \right)\\
    =& \frac{1}{4(-R_s)^{3/2}}\phi_{v,s}\ve^2\ln(\ve) + \frac{1}{2(-R_s)^{3/2}}\phi_{v,s}\left(\frac{1}{4} - \frac{1}{2} \ln(2\eta\sqrt{-R_s})\right)\ve^2 + \go(\ve^4).
\end{align*}
At last, for the third one, we have
\begin{align*}
    \frac{1}{4}\phi_sR_{v,s} \ve^2 \intsigp \frac{\sig^4}{(1-\sig^2R_s)^{3/2}} - \frac{\sig}{(-R_s)^{3/2}} \dd \sig =& \frac{1}{4(-R_s)^{5/2}}\phi_sR_{v,s} \ve^2 \int_0^{\frac{\eta\sqrt{-R_s}}{\ve}} \frac{x^4}{(1+x)^{3/2}}-x \dd x \\
    =& \frac{1}{4(-R_s)^{5/2}}\phi_sR_{v,s} \ve^2 \left(\frac{3}{2}\ln(\ve) + \frac{5}{4} - \frac{3}{2} \ln\left(2\eta\sqrt{-R_s}\right) + \go(\ve^2). \right) \\
    =& \frac{3}{8(-R_s)^{5/2}}\phi_sR_{v,s} \ve^2\ln(\ve) + \frac{1}{4(-R_s)^{5/2}}\phi_sR_{v,s}\left(\frac{5}{4} - \frac{3}{2} \ln\left(2\eta\sqrt{-R_s}\right)\right) \ve^2\\ &+ \go(\ve^4)
\end{align*}
The last term is equivalent to a multiple of $\ve^2$. To obtain the corresponding constant we may use the dominated convergence theorem. First, we carry out the following change of variable, $w = \ve \sig$ and then we can take the limit $\ve \to 0$. This gives
\begin{multline*}
    \frac{1}{\ve^2} \intsigp \int_0^1 \ve (t\ve\sig)^2 \int_0^1 (1-u)f_1(v_s + u t\ve\sig,\sig) \dd u \dd t \dd \sig \\ = \intwp \int_0^1 \int_0^1 t^2(1-u) 
     \!\begin{aligned}[t] \frac{1}{\ve^2}\Bigg[&\frac{\phi_{vv}w^3}{\sqrt{\ve^2-w^2R}} - \frac{\phi_{vv}w^2}{\sqrt{-R}}\\
    &+ 2\phi_v\left(\frac{R_vw^5}{2(\ve^2-w^2R)^{3/2}} - \frac{R_{v}w^2}{2(-R)^{3/2}} \right)\\
    &+ \phi\left(\frac{R_{vv}w^5}{2(\ve^2-w^2R)^{3/2}} - \frac{R_{vv}w^2}{2(-R)^{3/2}} \right) \\
    &+ \frac{3}{4}\phi\left(\frac{R_v^2w^7}{(\ve^2-w^2R)^{5/2}} - \frac{R_{v}^2w^2}{(-R)^{5/2}} \right) \Bigg](v_s+u t w)\dd u \dd t \dd w\end{aligned}\\
    \tend{\ve}{0} \intwp \int_0^1 \int_0^1 t^2(1-u) 
    \Bigg[-\frac{1}{2}\frac{\phi_{vv}}{(-R)^{3/2}} -\frac{3}{2}\frac{\phi_v R_v}{(-R)^{5/2}} - \frac{3}{4}\frac{\phi R_{vv}}{(-R)^{5/2}}- \frac{15}{8}\frac{\phi R_v^2}{(-R)^{7/2}}\Bigg](v_s+u t w) \dd u \dd t \dd w.
\end{multline*}
To derive the next term in the expansion in powers of $\ve$, we study the difference with the limit. Going back to the $\sig$ variable, and as before, using Taylor's formula in $\ve\sig$, one shows that this term is $\go\left(\ve^3\ln(\ve)\right)$.\newline
By grouping all the terms together, we obtain the desired expansion of $I_{\eta,+}^{0,1}(\phi)$ up to $\go\left(\ve^3\ln(\ve)\right)$
\begin{align*}
    I_{\eta,+}^{0,1}(\phi) =& - \frac{\phi_s}{-R_s} \ve +  \left(\frac{1}{4(-R_s)^{3/2}}\phi_{v,s} + \frac{3}{8(-R_s)^{5/2}}\phi_sR_{v,s} \right)\ve^2\ln(\ve)\\ 
                        &+ \left(\frac{1}{2}\frac{\phi_s}{\eta(-R_s)^{3/2}} +\frac{1}{2(-R_s)^{3/2}}\phi_{v,s}\left(\frac{1}{4} - \frac{1}{2} \ln(2\eta\sqrt{-R_s})\right) + \frac{1}{4(-R_s)^{5/2}}\phi_sR_{v,s}\left(\frac{5}{4} - \frac{3}{2} \ln\left(2\eta\sqrt{-R_s}\right)\right)\right)\ve^2\\
                        &+ \intwp \int_0^1 \int_0^1 t^2(1-u) 
    \Bigg[-\frac{1}{2}\frac{\phi_{vv}}{(-R)^{3/2}} -\frac{3}{2}\frac{\phi_v R_v}{(-R)^{5/2}} - \frac{3}{4}\frac{\phi R_{vv}}{(-R)^{5/2}}- \frac{15}{8}\frac{\phi R_v^2}{(-R)^{7/2}}\Bigg](v_s+u t w) \dd u \dd t \dd w \ve^2\\ 
    &+ \go\left(\ve^3\ln(\ve)\right).
\end{align*}
To finish the proof we have to deal with
\begin{equation*}
    I_{\eta,+}^{0,2}(\phi) = \intsigp \int_{0}^{1} \ve \phi(v_s+t\ve\sig) \left(\frac{\sig^3R_v(v_s+t\ve\sig)}{2(1-\sig^2R(v_s+t\ve\sig))^{3/2}} - \frac{R_v(v_s+t\ve\sig)}{2(-R(v_s+t\ve\sig))^{3/2}}\right) \dd t \dd \sig.
\end{equation*}
The proof is essentially identical. We first begin by using Taylor's formula in $\ve\sigma$ and derive
 \begin{align*}
     I_{\eta,+}^{0,2}(\phi) =&  \frac{\ve}{2} \phi_sR_{v,s} \intsigp \frac{\sig^3}{(1-\sig^2R_s)^{3/2}} - \frac{1}{(-R_s)^{3/2}} \dd \sig \\
                          &+ \frac{\ve^2}{2}(\phi_{v,s}R_{v,s}+\phi_sR_{vv,s}) \intsigp \frac{\sig^4}{2(1-\sig^2R_s)^{3/2}} - \frac{\sig}{2(-R_s)^{3/2}} \dd \sig \\
                          &+ \frac{3}{4}\frac{\ve^2}{2}\phi_sR_{v,s}^2 \intsigp \frac{\sig^6}{(1-\sig^2R_s)^{5/2}} - \frac{\sig}{(-R_s)^{5/2}} \dd\sig\\
                          &+ \intsigp \int_0^1 \ve (t\ve\sig)^2 \int_0^1 (1-u)f_2(v_s + u t\ve\sig,\sig) \dd u \dd t \dd \sig.
 \end{align*}
 where 
 \begin{multline*}
     f_2(.,\sig) = \frac{\phi_{vv}R_{v}+2\phi_vR_{vv}+\phi R_{vvv}}{2(1-\sig^2R)^{3/2}}\sig^3 - \frac{\phi_{vv}R_{v}+2\phi_vR_{vv}+\phi R_{vvv}}{2(-R)^{3/2}}
                  + \frac{6}{4}\frac{\phi_v R_v^2 + \frac{3}{2}\phi R_vR_{vv}}{(1-\sig^2R)^{5/2}}\sig^5 \\
                   - \frac{6}{4}\frac{\phi_v R_v^2 + \frac{3}{2}\phi R_vR_{vv}}{(-R)^{5/2}}
                  + \frac{15}{8} \frac{\phi R_v^3}{(1-\sig^2R)^{7/2}}\sig^7 - \frac{15}{8} \frac{\phi R_v^3}{(-R)^{7/2}}.
 \end{multline*}
As before, we treat each integral separately. For the first three terms, we use the asymptotic expansions from Lemma~\ref{I_eta,+,0,2}. For the first integral, we have
\begin{align*}
     \frac{\ve}{2} \phi_sR_{v,s} \intsigp \frac{\sig^3}{(1-\sig^2R_s)^{3/2}} - \frac{1}{(-R_s)^{3/2}} \dd \sig
                                    =& - \frac{\phi_sR_{v,s}}{(-R_s)^2} \ve + \frac{3}{4}\frac{\phi_sR_{v,s}}{\eta(-R_s)^{5/2}}\ve^2 + \go(\ve^4).
\end{align*}
For the second one, we have
\begin{multline*}
    \frac{\phi_{v,s}R_{v,s}+\phi_sR_{vv,s}}{4}\ve^2 \intsigp \frac{\sig^4}{(1-\sig^2R_s)^{3/2}} - \frac{\sig}{(-R_s)^{3/2}} \dd \sig 
                    = 3\frac{\phi_{v,s}R_{v,s}+\phi_sR_{vv,s}}{8(-R_s)^{5/2}} \ve^2\ln(\ve) \\+ \frac{\phi_{v,s}R_{v,s}+\phi_sR_{vv,s}}{4(-R_s)^{5/2}}\left(\frac{5}{4} - \frac{3}{2} \ln\left(2\eta\sqrt{-R_s}\right)\right) \ve^2 + \go(\ve^4).
\end{multline*}
As for the third one, we have
\begin{multline*}
    \frac{3\phi_sR_{v,s}^2\ve^2}{8} \intsigp \frac{\sig^6}{(1-\sig^2R_s)^{5/2}} - \frac{\sig}{(-R_s)^{5/2}} \dd\sig
    = \frac{15\phi_sR_{v,s}^2}{16(-R_s)^{7/2}}\ve^2\ln(\ve) + \frac{3\phi_sR_{v,s}^2}{8(-R_s)^{7/2}}\left(\frac{31}{12} - \frac{5}{2} \ln\left(2\eta\sqrt{-R_s}\right)\right)\ve^2\\ + \go(\ve^4).
\end{multline*}
As before, the last term is of order $\ve^2$. We perform the change of variable $w = \ve \sig$ and use the dominated convergence theorem to compute the limit. We obtain 
\begin{multline*}
    \frac{1}{\ve^2} \intsigp \int_0^1 \ve (t\ve\sig)^2 \int_0^1 (1-u)f_2(v_s + u t\ve\sig,\sig) \dd u \dd t \dd \sig \\= \intwp \int_0^1 \int_0^1 t^2(1-u)\frac{1}{\ve^2} 
    \!\begin{aligned}[t] \Bigg[&\frac{\phi_{vv}R_{v}+2\phi_vR_{vv}+\phi R_{vvv}}{2(\ve^2-w^2R)^{3/2}}w^5 - \frac{\phi_{vv}R_{v}+2\phi_vR_{vv}+\phi R_{vvv}}{2(-R)^{3/2}}w^2\\
                  &+ \frac{6}{4}\frac{\phi_v R_v^2 + \frac{3}{2}\phi R_vR_{vv}}{(\ve^2-w^2R)^{5/2}}w^7 - \frac{6}{4}\frac{\phi_v R_v^2 + \frac{3}{2}\phi R_vR_{vv}}{(-R)^{5/2}}w^2\\
                  &+ \frac{15}{8} \frac{\phi R_v^3}{(\ve^2-w^2R)^{7/2}}w^9 - \frac{15}{8} \frac{\phi R_v^3}{(-R)^{7/2}}w^2
                    \Bigg](v_s+u t w)\dd u \dd t \dd w\end{aligned}\\
        \tend{\ve}{0} \!\begin{aligned}[t] \intwp \int_0^1 \int_0^1 & t^2(1-u) \Bigg[-3\frac{\phi_{vv}R_{v}+2\phi_vR_{vv}+\phi R_{vvv}}{4(-R)^{5/2}} - \frac{15}{4}\frac{\phi_v R_v^2 + \frac{3}{2}\phi R_vR_{vv}}{(-R)^{7/2}}- \frac{105}{16}\frac{\phi R_v^3}{(-R)^{9/2}} \Bigg](v_s+u t w)\\ &\dd u \dd t \dd w. 
         \end{aligned}
\end{multline*}
As before, continuing the expansion shows that the next term is $\go(\ve^3\ln(\ve))$.\newline
By grouping all the terms together, we obtain the expansion of $I_{\eta,+}^{0,2}(\phi)$ up to $\go(\ve^3\ln(\ve))$,
\begin{multline*}
    I_{\eta,+}^{0,2}(\phi) = - \frac{\phi_sR_{v,s}}{(-R_s)^2} \ve + \left(3\frac{\phi_{v,s}R_{v,s}+\phi_sR_{vv,s}}{8(-R_s)^{5/2}} + \frac{15\phi_sR_{v,s}^2}{16(-R_s)^{7/2}} \right) \ve^2\ln(\ve)\\
                        + \left(\frac{3}{4}\frac{\phi_sR_{v,s}}{\eta(-R_s)^{5/2}} + \frac{3\phi_sR_{v,s}^2}{8(-R_s)^{7/2}}\left(\frac{31}{12} - \frac{5}{2} \ln\left(2\eta\sqrt{-R_s}\right)\right) + \frac{\phi_{v,s}R_{v,s}+\phi_sR_{vv,s}}{4(-R_s)^{5/2}}\left(\frac{5}{4} - \frac{3}{2} \ln\left(2\eta\sqrt{-R_s}\right)\right) \right) \ve^2\\
                        + \ve^2 \intwp\int_0^1 \int_0^1\Bigg[-3\frac{\phi_{vv}R_{v}+2\phi_vR_{vv}+\phi R_{vvv}}{4(-R)^{5/2}} - \frac{15}{4}\frac{\phi_v R_v^2 + \frac{3}{2}\phi R_vR_{vv}}{(-R)^{7/2}} - \frac{105}{16}\frac{\phi R_v^3}{(-R)^{9/2}} \Bigg](v_s+u t w)
                         t^2(1-u)\dd u \dd t \dd w\\ + \go(\ve^3\ln(\ve)).
\end{multline*}
To conclude, we simply gather the various asymptotic expansions. Besides, to simplify notation, we observe that 
\begin{align*}
    I_{\eta,+}^{0,\text{lim}}(\phi) =& \intwp \int_{0}^{1} \frac{\phi_v(v_s+t w)}{\sqrt{-R(v_s+t w)}} + \phi(v_s+t w) \frac{R_v(v_s+t w)}{2(-R(v_s+t w))^{3/2}}\dd t \dd w\\
                                 =& \intwp \frac{1}{w}\left(\frac{\phi(v_s+w)}{\sqrt{-R(v_s+ w)}} - \frac{\phi_s}{\sqrt{R_s}} \right)\dd w.
\end{align*}
\end{proof}

We can now deal with the negative integral, whose expansion is given by the following proposition.

\begin{proposition}\label{dl_int_negative}
We have the following asymptotic expansion
\begin{equation*}
    \intwn \frac{\phi(v_s+w)}{\sqrt{\ve^2-w^2R(v_s+w)}} \dd w = A_0^{\eta,-}(\phi) \ln(\ve) + B_0^{\eta,-}(\phi) + B_1^{\eta,-}(\phi)\ve + A_2^{\eta,-}(\phi)\ve^2\ln(\ve) + B_2^{\eta,-}(\phi)\ve^2 + \go\left(\ve^3\ln(\ve)\right),
\end{equation*}
where each coefficient is given by
\begin{enumerate}[label=\roman*)]
    \item $A_0^{\eta,-}(\phi) =-\frac{\phi_s}{\sqrt{-R_s}} , $
    \item $B_0^{\eta,-}(\phi) = \frac{\phi_s}{\sqrt{-R_s}}\ln\left(2\eta\sqrt{-R_s}\right) +  \intwn -\frac{1}{w}\left(\frac{\phi(v_s+w)}{\sqrt{-R(v_s+w)}} - \frac{\phi_s}{\sqrt{R_s}} \right)\dd w ,$
    \item $B_1^{\eta,-}(\phi) = - \frac{\phi_s}{-R_s} + \frac{\phi_sR_{v,s}}{(-R_s)^2},$
    \item $A_2^{\eta,-}(\phi) = - \frac{\phi_{v,s}}{4(-R_s)^{3/2}} - \frac{3\phi_sR_{v,s}}{8(-R_s)^{5/2}} + 3\frac{\phi_{v,s}R_{v,s}+\phi_sR_{vv,s}}{8(-R_s)^{5/2}} + \frac{15\phi_sR_{v,s}^2}{16(-R_s)^{7/2}},$
    \item $\begin{aligned}[t] &B_2^{\eta,-}(\phi) = \frac{\phi_s}{4\eta^2(-R_s)^{3/2}} + \frac{\phi_s}{2\eta(-R_s)^{3/2}} - \frac{\phi_{v,s}}{2(-R_s)^{3/2}}\left(\frac{1}{4} - \frac{1}{2} \ln(2\eta\sqrt{-R_s})\right) - \frac{\phi_sR_{v,s}}{4(-R_s)^{5/2}}\left(\frac{5}{4} - \frac{3}{2} \ln\left(2\eta\sqrt{-R_s}\right)\right)\\
                        &-\frac{3}{4}\frac{\phi_sR_{v,s}}{\eta(-R_s)^{5/2}} + \frac{3\phi_sR_{v,s}^2}{8(-R_s)^{7/2}}\left(\frac{31}{12} - \frac{5}{2}\ln\left(2\eta\sqrt{-R_s}\right)\right) + \frac{\phi_{v,s}R_{v,s}+\phi_sR_{vv,s}}{4(-R_s)^{5/2}}\left(\frac{5}{4} - \frac{3}{2} \ln\left(2\eta\sqrt{-R_s}\right)\right) \\
                        &+ \intwn \int_0^1 \int_0^1 t^2(1-u) \Bigg[-\frac{1}{2}\frac{\phi_{vv}}{(-R)^{3/2}} -\frac{3}{2}\frac{\phi_v R_v}{(-R)^{5/2}} - \frac{3}{4}\frac{\phi R_{vv}}{(-R)^{5/2}}- \frac{15}{8}\frac{\phi R_v^2}{(-R)^{7/2}}\Bigg](v_s+u t w) \dd u \dd t \dd w \end{aligned}$\end{enumerate}
                        $$+ \intwn \int_0^1 \int_0^1 t^2(1-u)\Bigg[3\frac{\phi_{vv}R_{v}+2\phi_vR_{vv}+\phi R_{vvv}}{4(-R)^{5/2}} + \frac{15}{4}\frac{\phi_v R_v^2 + \frac{3}{2}\phi R_vR_{vv}}{(-R)^{7/2}} + \frac{105}{16}\frac{\phi R_v^3}{(-R)^{9/2}} \Bigg](v_s+u t w) \dd u \dd t \dd w.$$
\end{proposition}

\begin{proof}
We use the expansion already computed in the case of the positive integral. To do so, we introduce $\ti{w} = - w$, and note that
\begin{equation*}
    \intwn \frac{\phi(v_s+w)}{\sqrt{\ve^2-w^2R(v_s+w)}} \dd w = \intwp \frac{\phi(v_s-\ti{w})}{\sqrt{\ve^2-\ti{w}^2R(v_s-\ti{w})}} \dd \ti{w}
\end{equation*}
We therefore define $\ti{R}(v_s+h) := R(v_s-h)$, that is $\ti{R}(v) = R(2v_s-v)$ and likewise $\ti{\phi}(v_s+h) = \phi(v_s-h)$. Thus, we obtain that
\begin{equation*}
    \intwn \frac{\phi(v_s+w)}{\sqrt{\ve^2-w^2R(v_s+w)}} \dd w = \intwp \frac{\ti{\phi}(v_s+\ti{w})}{\sqrt{\ve^2-\ti{w}^2\ti{R}(v_s+\ti{w})}} \dd \ti{w}
\end{equation*}
We can now use Proposition~\ref{dl_int_positive}, and undo the change of variable $\ti{w} = - w$ to conclude.
\end{proof}

\subsection{Asymptotic expansions of different integrals}

To finish, we gather here some asymptotic expansions of integrals that we have used in the previous section.

\begin{lemma}\label{ln}
When $\ve \rightarrow 0$, 
\begin{equation*}
    \int_{0}^{\frac{\eta\sqrt{-R_s}}{\ve}} \frac{1}{\sqrt{1+x^2}} \dd x = -\ln(\ve) + \ln(2\eta\sqrt{-R_s}) + \frac{\ve^2}{4\eta^2(-R_s)} + \go(\ve^4).
\end{equation*}
\end{lemma}
\begin{proof}
    We compute explicitly the integral and then expand it in powers of $\ve$,
    \begin{align*}
             \int_{0}^{\frac{\eta\sqrt{-R_s}}{\ve}} \frac{1}{\sqrt{1+x^2}} \dd x =& \left[\frac{1}{2} \ln\left(\frac{\sqrt{1+x^2} + x}{\sqrt{1+x^2}-x}\right) \right]_{0}^{\frac{\eta\sqrt{-R_s}}{\ve}}\\
                                                                              =& \frac{1}{2} \ln\left(\frac{2\eta^2(-R_s) + 2 \eta\sqrt{-R_s}\sqrt{\ve^2 + \eta^2(-R_s)} + \ve^2}{\ve^2} \right) \\                                                                              
                                                                              =& -\ln(\ve) + \ln\left(2\eta\sqrt{-R_s}\right) + \frac{\ve^2}{4\eta^2(-R_s)} + \go(\ve^4).
         \end{align*}
\end{proof}

 \begin{lemma}\label{i_c,+,0,1}
    When $\ve \rightarrow 0$, 
     \begin{enumerate}[label=\roman*)]
         \item $$\int_0^{\frac{\eta\sqrt{-R_s}}{\ve}} \frac{x}{\sqrt{1+x^2}}\dd x = \frac{\eta\sqrt{-R_s}}{\ve} - 1 +\frac{1}{2}\frac{\ve}{\eta\sqrt{-R_s}} + \go(\ve^3), $$
         \item $$\int_{0}^{\frac{\eta\sqrt{-R_s}}{\ve}} \frac{x^2}{\sqrt{1+x^2}} \dd x = \frac{1}{2}\frac{\eta^2(-R_s)}{\ve^2} + \frac{1}{2}\ln(\ve) + \frac{1}{4} - \frac{1}{2} \ln(2\eta\sqrt{-R_s}) + \go(\ve^2),$$
         \item $$\int_0^{\frac{\eta\sqrt{-R_s}}{\ve}} \frac{x^4}{(1+x)^{3/2}} \dd x = \frac{1}{2}\left(\frac{\eta\sqrt{-R_s}}{\ve}\right)^2 + \frac{3}{2}\ln(\ve) + \frac{5}{4} - \frac{3}{2} \ln\left(2\eta\sqrt{-R_s}\right) + \go(\ve^2).$$
     \end{enumerate}
 \end{lemma}

 \begin{proof}
     Again we compute explicitly the integrals and then use classical asymptotic expansions.
     \begin{enumerate}[label=\roman*)]
         \item \begin{align*}
             \int_0^{\frac{\eta\sqrt{-R_s}}{\ve}} \frac{x}{\sqrt{1+x^2}} \dd x =& \left[\sqrt{1+x^2}\right]_0^{\frac{\eta\sqrt{-R_s}}{\ve}}\\
                                                                            =& \frac{\eta\sqrt{-R_s}}{\ve} - 1 +\frac{1}{2}\frac{\ve}{\eta\sqrt{-R_s}} + \go(\ve^3).
         \end{align*}
         
         \item \begin{align*}
             \int_{0}^{\frac{\eta\sqrt{-R_s}}{\ve}} \frac{x^2}{\sqrt{1+x^2}} \dd x =& \left[\frac{1}{2}x\sqrt{x^2+1} - \frac{1}{4}\ln\left(\frac{\sqrt{1+x^2} + x}{\sqrt{1+x^2}-x}\right)\right]_{0}^{\frac{\eta\sqrt{-R_s}}{\ve}} \\
                                                                                =& \frac{1}{2}\left(\frac{\eta\sqrt{-R_s}}{\ve}\right)^2 + \frac{1}{2}\ln(\ve) +\frac{1}{4} - \frac{1}{2}\ln\left(2\eta\sqrt{-R_s}\right) + \go(\ve^2).
         \end{align*}
         \item \begin{align*}
             \int_0^{\frac{\eta\sqrt{-R_s}}{\ve}} \frac{x^4}{(1+x)^{3/2}} \dd x =& \left[\frac{x^3+3x}{2\sqrt{x^2+1}} - \frac{3}{4}\ln\left(\frac{\sqrt{1+x^2} + x}{\sqrt{1+x^2}-x}\right) \right]_0^{\frac{\eta\sqrt{-R_s}}{\ve}}\\
                                                                             =& \frac{1}{2}\left(\frac{\eta\sqrt{-R_s}}{\ve}\right)^2 + \frac{3}{2}\ln(\ve) + \frac{5}{4} - \frac{3}{2} \ln\left(2\eta\sqrt{-R_s}\right) + \go(\ve^2).
         \end{align*}
     \end{enumerate}
 \end{proof}

 \begin{lemma}\label{I_eta,+,0,2}
     When $\ve \rightarrow 0$, 
     \begin{enumerate}[label=\roman*)]
         \item $$\int_0^{\frac{\eta\sqrt{-R_s}}{\ve}} \frac{x^3}{(1+x^2)^{3/2}} \dd x = \frac{\eta\sqrt{-R_s}}{\ve} -2  + \frac{3}{2} \frac{\ve}{\eta\sqrt{-R_s}} + \go(\ve^3),$$
         \item $$\int_0^{\frac{\eta\sqrt{-R_s}}{\ve}} \frac{x^6}{(1+x^2)^{5/2}} \dd x = \frac{1}{2}\left(\frac{\eta\sqrt{-R_s}}{\ve}\right)^2 + \frac{5}{2}\ln(\ve) + \left(\frac{31}{12} - \frac{5}{2} \ln\left(2\eta\sqrt{-R_s}\right)\right) + \go(\ve^2).$$
     \end{enumerate}
 \end{lemma}

 \begin{proof}
    This follows from
     \begin{enumerate}[label=\roman*)]
        \item \begin{align*}
            \int_0^{\frac{\eta\sqrt{-R_s}}{\ve}} \frac{x^3}{(1+x^2)^{3/2}} \dd x =& \left[\frac{x^2+2}{\sqrt{x^2+1}} \right]_0^{\frac{\eta\sqrt{-R_s}}{\ve}}\\
                                                                              =& \frac{\eta\sqrt{-R_s}}{\ve} - 2 + \frac{3}{2} \frac{\ve}{\eta\sqrt{-R_s}} + \go(\ve^3).
        \end{align*}
     
        \item \begin{align*}
             \int_0^{\frac{\eta\sqrt{-R_s}}{\ve}} \frac{x^6}{(1+x^2)^{5/2}} \dd x =& \left[ \frac{3x^5+20x^3+15x}{6(x^2+1)^{3/2}} - \frac{5}{4}\ln\left(\frac{\sqrt{1+x^2} + x}{\sqrt{1+x^2}-x}\right)  \right]_0^{\frac{\eta\sqrt{-R_s}}{\ve}}\\
                                                                               =&\frac{1}{2}\left(\frac{\eta\sqrt{-R_s}}{\ve}\right)^2 + \frac{5}{2}\ln(\ve) + \left(\frac{31}{12} - \frac{5}{2} \ln\left(2\eta\sqrt{-R_s}\right)\right) + \go(\ve^2).
         \end{align*}
     \end{enumerate}
 \end{proof}

\addcontentsline{toc}{section}{Références}
\bibliographystyle{alphaabbr}
\bibliography{biblio.bib}

\begin{thebibliography}{DBGN15}

\bibitem[AP09]{angulopavaNonlinearDispersiveEquations2009}
J.~Angulo~Pava.
\newblock {\em Nonlinear Dispersive Equations. {{Existence}} and Stability of Solitary and Periodic Travelling Wave Solutions.}, volume 156 of {\em Math. {{Surv}}. {{Monogr}}.}
\newblock Providence, RI: American Mathematical Society (AMS), 2009.

\bibitem[AR22]{audiardPlanePeriodicWaves2022}
C.~Audiard and L.~M. Rodrigues.
\newblock About plane periodic waves of the nonlinear {{Schr\"odinger}} equations.
\newblock {\em Bulletin de la Soci\'et\'e Math\'ematique de France}, 150(1):111--207, 2022.

\bibitem[BJRZ11]{barkerMetastabilitySolitaryRoll2011}
B.~Barker, M.~A. Johnson, L.~M. Rodrigues, and K.~Zumbrun.
\newblock Metastability of solitary roll wave solutions of the {{St}}. {{Venant}} equations with viscosity.
\newblock {\em Physica D. Nonlinear Phenomena}, 240(16):1289--1310, 2011.

\bibitem[{Ben}10]{benzoni-gavageSpectralTransverseInstability2010}
S.~{Benzoni-Gavage}.
\newblock Spectral transverse instability of solitary waves in {{Korteweg}} fluids.
\newblock {\em Journal of Mathematical Analysis and Applications}, 361(2):338--357, 2010.

\bibitem[{Ben}13]{benzoni-gavagePlanarTravelingWaves2013}
S.~{Benzoni-Gavage}.
\newblock Planar traveling waves in capillary fluids.
\newblock {\em Differential and Integral Equations. An International Journal for Theory \& Applications}, 26(3-4):439--485, 2013.

\bibitem[BDD06]{benzoni-gavageWellposednessOnedimensionalKorteweg2006}
S.~{Benzoni-Gavage}, R.~Danchin, and S.~Descombes.
\newblock Well-posedness of one-dimensional {{Korteweg}} models.
\newblock {\em Electronic Journal of Differential Equations (EJDE)}, 2006:35, 2006.

\bibitem[BDD07]{benzoni-gavageWellposednessEulerKortewegModel2007}
S.~{Benzoni-Gavage}, R.~Danchin, and S.~Descombes.
\newblock On the well-posedness for the {{Euler-Korteweg}} model in several space dimensions.
\newblock {\em Indiana University Mathematics Journal}, 56(4):1499--1579, 2007.

\bibitem[BMR16]{benzoni-gavageCoperiodicStabilityPeriodic2016}
S.~{Benzoni-Gavage}, C.~Mietka, and L.~M. Rodrigues.
\newblock Co-periodic stability of periodic waves in some {{Hamiltonian PDEs}}.
\newblock {\em Nonlinearity}, 29(11):3241--3308, 2016.

\bibitem[BMR20]{benzoni-gavageStabilityPeriodicWaves2020}
S.~{Benzoni-Gavage}, C.~Mietka, and L.~M. Rodrigues.
\newblock Stability of periodic waves in {{Hamiltonian PDEs}} of either long wavelength or small amplitude.
\newblock {\em Indiana University Mathematics Journal}, 69(2):545--619, 2020.

\bibitem[BMR21]{benzoni-gavageModulatedEquationsHamiltonian2021}
S.~{Benzoni-Gavage}, C.~Mietka, and L.~M. Rodrigues.
\newblock Modulated equations of {{Hamiltonian PDEs}} and dispersive shocks.
\newblock {\em Nonlinearity}, 34(1):578--641, 2021.

\bibitem[BNR14a]{benzoni-gavageSlowModulationsPeriodic2014}
S.~{Benzoni-Gavage}, P.~Noble, and L.~M. Rodrigues.
\newblock Slow {{Modulations}} of {{Periodic Waves}} in {{Hamiltonian PDEs}}, with {{Application}} to {{Capillary Fluids}}.
\newblock {\em Journal of Nonlinear Science}, 24(4):711--768, August 2014.

\bibitem[BNR14b]{benzoni-gavageStabilityPeriodicWaves2014}
S.~{Benzoni-Gavage}, P.~Noble, and L.~M. Rodrigues.
\newblock Stability of periodic waves in {{Hamiltonian PDEs}}.
\newblock {\em Journ\'ees \'equations aux d\'eriv\'ees partielles}, pages 1--22, August 2014.

\bibitem[BSS87]{bonaStabilityInstabilitySolitary1987}
J.~L. Bona, P.~E. Souganidis, and W.~A. Strauss.
\newblock Stability and instability of solitary waves of {{Korteweg-de Vries}} type.
\newblock {\em Proceedings of the Royal Society. London. Series A. Mathematical, Physical and Engineering Sciences}, 411(1841):395--412, 1987.

\bibitem[BHJ16]{bronskiModulationalInstabilityEquations2016}
J.~C. Bronski, V.~M. Hur, and M.~A. Johnson.
\newblock Modulational instability in equations of {{KdV}} type.
\newblock In {\em New Approaches to Nonlinear Waves}, volume 908 of {\em Lecture {{Notes}} in {{Phys}}.}, pages 83--133. Springer, Cham, 2016.

\bibitem[BJK11]{bronskiIndexTheoremStability2011}
J.~C. Bronski, M.~A. Johnson, and T.~Kapitula.
\newblock An index theorem for the stability of periodic travelling waves of {{Korteweg-de Vries}} type.
\newblock {\em Proceedings of the Royal Society of Edinburgh. Section A. Mathematics}, 141(6):1141--1173, 2011.

\bibitem[BGdR25]{bukiedaOrbitalStabilityPlane2025}
E.~Bukieda, L.~Gar{\'e}naux, and B.~de~Rijk.
\newblock Orbital {{Stability}} of {{Plane Waves}} in the {{Klein-Gordon Equation}} against {{Localized Perturbations}}, June 2025.

\bibitem[CP25]{cuiInstabilityBandsPeriodic2025}
S.~Cui and D.~E. Pelinovsky.
\newblock Instability bands for periodic travelling waves in the modified {{Korteweg}}--de {{Vries}} equation.
\newblock {\em Proceedings A}, 481(2320):Paper No. 20240993, 18, 2025.

\bibitem[DBGN15]{debievreOrbitalStabilityAnalysis2015}
S.~De~Bi{\`e}vre, F.~Genoud, and S.~R. Nodari.
\newblock Orbital stability: Analysis meets geometry.
\newblock In {\em Nonlinear Optical and Atomic Systems. {{At}} the Interface of Physics and Mathematics. {{Based}} on Lecture Notes given at the 2013 {{Painlev\'e-CEMPI-PhLAM}} Thematic Semester.}, pages 147--273. Cham: Springer; Lille: Centre Europ\'een pour les Math\'ematiques, la Physiques et leurs Interactions (CEMPI), 2015.

\bibitem[Gar93]{gardnerStructureSpectraPeriodic1993}
R.~A. Gardner.
\newblock On the structure of the spectra of periodic travelling waves.
\newblock {\em Journal de Math\'ematiques Pures et Appliqu\'ees. Neuvi\`eme S\'erie}, 72(5):415--439, 1993.

\bibitem[Gar97]{gardnerSpectralAnalysisLong1997}
R.~A. Gardner.
\newblock {Spectral analysis of long wavelength periodic waves and applications.}
\newblock {\em Journal f\"ur die reine und angewandte Mathematik}, 491:149--182, 1997.

\bibitem[GZ98]{gardnerGapLemmaGeometric1998}
R.~A. Gardner and K.~Zumbrun.
\newblock The gap lemma and geometric criteria for instability of viscous shock profiles.
\newblock {\em Communications on Pure and Applied Mathematics}, 51(7):797--855, July 1998.

\bibitem[GSS87]{grillakisStabilityTheorySolitary1987}
M.~Grillakis, J.~Shatah, and W.~Strauss.
\newblock Stability theory of solitary waves in the presence of symmetry. {{I}}.
\newblock {\em Journal of Functional Analysis}, 74:160--197, 1987.

\bibitem[HK08]{haragusSpectraPeriodicWaves2008}
M.~H{\v a}r{\v a}gu{\c s} and T.~Kapitula.
\newblock On the spectra of periodic waves for infinite-dimensional {{Hamiltonian}} systems.
\newblock {\em Physica D. Nonlinear Phenomena}, 237(20):2649--2671, 2008.

\bibitem[Ian22]{iandoliCauchyProblemQuasilinear2022}
F.~Iandoli.
\newblock On the {{Cauchy}} problem for quasi-linear {{Hamiltonian KdV-type}} equations.
\newblock In {\em Qualitative Properties of Dispersive {{PDEs}}}, volume~52 of {\em Springer {{INdAM Ser}}.}, pages 167--186. Springer, Singapore, 2022.

\bibitem[JLL19]{jinNonlinearModulationalInstability2019}
J.~Jin, S.~Liao, and Z.~Lin.
\newblock Nonlinear {{Modulational Instability}} of {{Dispersive PDE Models}}.
\newblock {\em Archive for Rational Mechanics and Analysis}, 231(3):1487--1530, March 2019.

\bibitem[Joh09]{johnsonNonlinearStabilityPeriodic2009}
M.~A. Johnson.
\newblock Nonlinear stability of periodic traveling wave solutions of the generalized {{Korteweg-de Vries}} equation.
\newblock {\em SIAM Journal on Mathematical Analysis}, 41(5):1921--1947, 2009.

\bibitem[JNRZ14]{johnsonBehaviorPeriodicSolutions2014}
M.~A. Johnson, P.~Noble, L.~M. Rodrigues, and K.~Zumbrun.
\newblock Behavior of periodic solutions of viscous conservation laws under localized and nonlocalized perturbations.
\newblock {\em Inventiones Mathematicae}, 197(1):115--213, 2014.

\bibitem[KD15]{kapitulaSpectralOrbitalStability2015}
T.~Kapitula and B.~Deconinck.
\newblock On the spectral and orbital stability of spatially periodic stationary solutions of generalized {{Korteweg-de Vries}} equations.
\newblock In {\em Hamiltonian Partial Differential Equations and Applications. {{Selected}} Papers Based on the Presentations at the Conference on {{Hamiltonian PDEs}}: Analysis, Computations and Applications, {{Toronto}}, {{Canada}}, {{January}} 10--12, 2014}, pages 285--322. Toronto: The Fields Institute for Research in the Mathematical Sciences; New York, NY: Springer, 2015.

\bibitem[KKS05]{kapitulaCountingEigenvaluesKrein2005}
T.~Kapitula, P.~G. Kevrekidis, and B.~Sandstede.
\newblock Counting eigenvalues via the {{Krein}} signature in infinite-dimensional {{Hamiltonian}} systems.
\newblock {\em Physica D: Nonlinear Phenomena}, 201(1):199--201, February 2005.

\bibitem[KP13]{kapitulaSpectralDynamicalStability2013}
T.~Kapitula and K.~Promislow.
\newblock {\em Spectral and {{Dynamical Stability}} of {{Nonlinear Waves}}}, volume 185 of {\em Applied {{Mathematical Sciences}}}.
\newblock Springer New York, New York, NY, 2013.

\bibitem[Kat95]{katoPerturbationTheoryLinear1995}
T.~Kato.
\newblock {\em Perturbation {{Theory}} for {{Linear Operators}}}, volume 132 of {\em Classics in {{Mathematics}}}.
\newblock Springer Berlin Heidelberg, Berlin, Heidelberg, 1995.

\bibitem[MZ03]{masciaPointwiseGreenFunction2003}
C.~Mascia and K.~Zumbrun.
\newblock Pointwise {{Green}} function bounds for shock profiles of systems with real viscosity.
\newblock {\em Archive for Rational Mechanics and Analysis}, 169(3):177--263, 2003.

\bibitem[Mie17]{mietkaWellposednessQuasilinearKortewegde2017}
C.~Mietka.
\newblock On the well-posedness of a quasi-linear {{Korteweg-de Vries}} equation.
\newblock {\em Annales Math\'ematiques Blaise Pascal}, 24(1):83--114, 2017.

\bibitem[PW92]{pegoEigenvaluesInstabilitiesSolitary1992}
R.~L. Pego and M.~I. Weinstein.
\newblock Eigenvalues, and instabilities of solitary waves.
\newblock {\em Philosophical Transactions of the Royal Society of London. Series A. Mathematical, Physical Sciences and Engineering}, 340(1656):47--94, 1992.

\bibitem[SS01]{sandstedeStabilityPeriodicTravelling2001}
B.~Sandstede and A.~Scheel.
\newblock On the {{Stability}} of {{Periodic Travelling Waves}} with {{Large Spatial Period}}.
\newblock {\em Journal of Differential Equations}, 172(1):134--188, May 2001.

\bibitem[YZ19]{yangConvergencePeriodGoes2019}
Z.~Yang and K.~Zumbrun.
\newblock Convergence as period goes to infinity of spectra of periodic traveling waves toward essential spectra of a homoclinic limit.
\newblock {\em Journal de Math\'ematiques Pures et Appliqu\'ees}, 132:27--40, December 2019.

\end{thebibliography}

\end{document}